\documentclass[a4paper,12pt]{article}
\usepackage{graphicx,amssymb,amstext,amsmath}

\usepackage[latin1]{inputenc}
\usepackage{amsthm}
\usepackage{dsfont}
\usepackage{amsmath}
\usepackage{amsfonts}
\usepackage{amssymb}
\usepackage{mathrsfs}
\usepackage{graphicx}
\usepackage{cancel}
\usepackage{graphicx,color}
\usepackage[all]{xy}
\usepackage{titlesec}
\usepackage[total={5.2in,9.1in},
top=1.4in, left=1.55in, includefoot]{geometry}

\newtheorem{T}{Theorem}[section]
\newtheorem{Defi}[T]{Definition}
\newtheorem{R}[T]{Remark}
\newtheorem{Pro}[T]{Proposition}
\newtheorem{C}[T]{Corollary}
\newtheorem{Le}[T]{Lemma}

\newcommand{\ds}{\displaystyle}
\newcommand{\B}{\mathbb{B}}
\newcommand{\A}{\mathbb{A}}
\newcommand{\re}{\mathbb{R}}

\newcommand{\cR}{\mathcal{R}}
\newcommand{\cS}{\mathcal{S}}
\newcommand{\cL}{\mathbb{L}}
\newcommand{\cM}{\mathbb{M}}

\titleformat{\paragraph}
{\normalfont\normalsize\bfseries}{\theparagraph}{1em}{}
\titlespacing*{\paragraph}
{0pt}{3.25ex plus 1ex minus .2ex}{1.5ex plus .2ex}

\title{\textbf{On the Lagrange and Markov
Dynamical Spectra for Geodesic Flows in Surfaces
with Negative Curvature}}
\author{ Carlos Gustavo T. de A. Moreira\footnote{\text{C. G. Moreira is partially supported by CNPq,  the Palis Balzan Prize and FAPERJ.}}\\
Sergio Augusto Roma\~na Ibarra \footnote{S. Roma\~na is partially supported by CNPq, Capes, the Palis Balzan Prize and FAPERJ.} }
\date{}
\begin{document}
\maketitle

\begin{abstract}
\noindent We consider the Lagrange and the Markov dynamical spectra associated with the geodesic flow on  surfaces of negative curvature. We show that for a large set of real functions on the unit tangent bundle and  typical metrics with  negative curvature and finite volume, both the Lagrange and the Markov dynamical spectra have non-empty interiors.
\end{abstract}

\tableofcontents
 \section{Introduction}
The classical {\it Lagrange spectrum} from the theory of Diophantine approximations can be defined as follows:
We associate, to each irrational number $\alpha$, its best constant of approximation (Lagrange value of $\alpha$), given by \begin{eqnarray*}k(\alpha)&=&\sup\left\{k>0:\left|\alpha-\frac{p}{q}\right|<\frac{1}{kq^2} \ \text{has infinite rational solutions $\frac{p}{q}$ }\right\}\\
&=&\limsup_{ \stackrel{p,q\to\infty}{p,q\in \mathbb N}}\left|q(q\alpha-p)\right|^{-1}\in \re\cup\{+\infty\}.
\end{eqnarray*}
The Lagrange spectrum is the set $$\cL=\{k(\alpha):\alpha\in \re\setminus \mathbb{Q} \ \text{and} \ k(\alpha)<\infty\}.$$
\noindent Another interesting set related to Diophantine approximations is the classical \textit{Markov spectrum} defined by
\begin{equation}\label{SMClassic}
\cM=\left\{\left(\inf_{(x,y)\in \mathbb{Z}^{2}\setminus(0,0)}|f(x,y)|\right)^{-1}:f(x,y)=ax^2+bxy+cy^2 \ \text{with} \ b^2-4ac=1\right\}.\end{equation}
\noindent Perron (see, for instance, \cite{CF}) gave a dynamical characterization of the Lagrange and Markov spectra, as follows:\\

\noindent Let $\Sigma=({\mathbb{N}^*})^{\mathbb{Z}}$ and $\sigma\colon \Sigma \to \Sigma$ the shift defined by $\sigma((a_n)_{n\in\mathbb{Z}})=(a_{n+1})_{n\in\mathbb{Z}}$. If $f\colon \Sigma \to \re$ is defined by $f((a_n)_{n\in\mathbb{Z}})=\alpha_{0}+\beta_{0}=[a_0,a_1,\dots]+[0,a_{-1},a_{-2},\dots]$, then 
$$\cL=\left\{\limsup_{n\to\infty}f(\sigma^{n}(\underline{\theta})):\underline{\theta}\in \Sigma\right\}$$
and  
$$\cM=\left\{\sup_{n\in\mathbb{Z}}f(\sigma^{n}(\underline{\theta})):\underline{\theta}\in \Sigma\right\}.$$
As a consequence of this characterization it follows that $\cL$ and $\cM$ are closed sets of real numbers and $\cL\subset \cM$.

In 1947, M. Hall \cite{Hall} proved that $\cL$ (and thus also $\cM$) contains a whole half-line (for instance $[6,+\infty)$. 

In 1975, G. Freiman (cf. \cite{F} and \cite{CF}) determined the precise beginning of Hall's ray (the biggest half-line contained in $\cL$, which turns out to be also the biggest half-line contained in $\cM$), which is 
$$
\frac{2221564096 + 283748\sqrt{462}}{491993569} \cong 4,52782956616\dots\,.
$$ 
Both Hall's and Freiman's results are directly related to the study of arithmetic sums of regular Cantor sets.

Perron characterization of the classical Markov and Lagrange spectra inspired a more general notion of {\it dynamical} Markov and Lagrange spectra (see for instance \cite{MR}):

Let $M$ be a smooth manifold, $T=\mathbb{Z}$ or $\mathbb{R}$, and $\phi=(\phi^t)_{t\in T}$ be a discrete-time ($T=\mathbb{Z}$) or 
continuous-time ($T=\mathbb{R}$) smooth dynamical system on $M$, that is, $\phi^t:M\to M$ {is a}  smooth 
diffeomorphism, {where} $\phi^0=\textrm{id}$,
and $\phi^t\circ\phi^s=\phi^{t+s}$ for all $t,s\in T$. 
Given a compact invariant subset $\Lambda\subset M$ and a function $f:M\to\mathbb{R}$, we define the 
\emph{dynamical Markov {spectrum}}, {denoted by} $\cM(\phi, \Lambda, f)$,
\emph{Lagrange spectrum},  {denoted by}  
$\cL(\phi, \Lambda, f)$  as 
$$\cM(\phi, \Lambda, f)=\{m_{\phi, f}(x): x\in\Lambda\}, \quad
 {(}\cL(\phi, \Lambda, f)=\{\ell_{\phi, f}(x): x\in\Lambda\}{)}$$
where  
$$m_{\phi, f}(x):=\sup\limits_{t\in \re} f(\phi^t(x)), \quad {(}\ell_{\phi, f}(x):=\limsup\limits_{t\to+\infty} f(\phi^t(x)){)}.$$
It is not difficult to show that $\cL(\phi, \Lambda, f)\subset \cM(\phi, \Lambda, f)$  \cite{MR}.\\
\ \\
When $\Lambda$ is the whole manifold, we denote 
$\cM(\phi, M, f)=\cM(\phi, f)$ and\break $\cL(\phi, M, f)=\cL(\phi, f)$.\\

A class of dynamical spectra which is closely related to the classical ones is given by spectra associated to geodesic flows on surfaces of constant negative curvature. In 1986, A. Haas \cite{Haas} proved the existence of an analogous to Hall's ray (a half-line contained in the dynamical spectra) for the geodesic flow on the quotient of $\mathbb{H}^{2}$ by a Fuchsian group of $SL(2,\re)$ and  A. Haas and C. Series  for Hecke group  \cite{HS}, this last related with Lagrange spectra and cusp excursions in modular surface (cf. \cite{OurBook}). In 1997, Thomas A. Schmidt and Mark Sheigorn \cite{SS} proved that Riemann surfaces have  Hall's ray in every cusp. Recently, in \cite{AMU2} it was proved that Hall's rays are persistent on each cusp of a Riemann surface (so in constant negative curvature).

The existence of Hall's rays for dynamical spectra was also proved in other contexts. In 2015, P. Hubert, L. Marchese, and C. Ulcigrai \cite{HMU} showed the existence of Hall's ray in the context of Teichm\" uller dynamics, more precisely for moduli surfaces, using renormalization. In 2016, M. Artigiani, L. Marchese, C. Ulcigrai \cite{AMU} showed that Veech surfaces also have a Hall's ray. In the context of non-archimedean quadratic Lagrange spectra, some new results is this direction can be found in \cite{PP2} and \cite{Bugeaud}\\

\indent For manifolds of finite volume, variable negative curvature and dimension greater than or equal to $3$, in \cite{PP} it was proved the existence of Hall's ray  for the \emph{height function} associated with an \emph{end} of the manifold (see the beginning of Section \ref{SHF} for more details of the \emph{height function}). In the same paper \cite[page 278]{PP},  J. Parkkonen and F. Paulin expected that the existence of Hall's ray could be false for surfaces of variable negative curvature. \\
\noindent This paper is inspired by the question: \\
\ \\
\textbf{Question:}
Is there Hall's ray for surfaces of variable negative curvature and finite volume?\\
\indent In this work, we show a positive result in the direction of this question: we prove that these spectra have typically non-empty interiors for surfaces of variable negative curvature and finite volume.

The continuity of the Hausdorff dimension of the intersection of the spectra with the half-line $(-\infty, t)$ was studied at \cite{CMR} for geodesic flows on surfaces of negative curvature, generalizing results of \cite{M} and \cite{CMM}.  

We refer to the book \cite{CF} for results on the classical spectra and to the book \cite{OurBook} for more recent results on classical and dynamical spectra.\\

In \cite{R}, the second author studied Markov and Lagrange dynamical spectra for Anosov flows on manifolds of dimension $3$. It has been proved that for suitable small perturbations of the flow the dynamical Lagrange and Markov spectra have persistently non-empty interior for typical real functions. Although surfaces of pinched negative curvature have geodesic flows of the Anosov type, the perturbations constructed in \cite{R} are not necessarily perturbations by geodesic flows. Indeed, even a small local perturbation of a Riemannian metric in a neighborhood of a point affects the corresponding geodesic flow in the whole fibers of the unit tangent bundle over this neighborhood. Thus, in this paper, we study the Lagrange and Markov spectra for the geodesic flow of surfaces of pinched negative curvature. More specifically,  we pursue two goals. The first goal is to extend the result {in} \cite{R} in the context of geodesic flows (see Theorem \ref{T1} and Theorem \ref{T2} for details), 
which requires sophisticated perturbation techniques of Riemannian metrics.
The second goal is to obtain a version of Theorem \ref{T2} for the 
\emph{height function}, associated with an \emph{end} of the manifold (see Theorem \ref{T3}). 

\subsection{Statement of the Main Results}
Let  $M$ be a complete surface, and let $g_0$ be a smooth ($C^r$, $r\geq 3$) pinched negatively curved riemannian metric on $M$, 
\emph{i.e}., the curvature is bounded above and  below by two negative constants and  finite volume.  Let $\phi=(\phi^t_{g_{_{0}}})_{t\in\mathbb{R}}$ be
the geodesic flow on the unit tangent bundle $S^{g_{_{0}}}M$ of $M$ with respect to $g_0$. We denote {\boldmath$\mathcal{G}$}$^{3}(M)$ the space of $C^{3}$ riemannian metric on 
$M$. For a metric $g\in${\boldmath$\mathcal{G}$}$^{3}(M)$  we denote {by} $S^{g}M$ the unit tangent bundle of the metric $g$ and $\phi_g$ is the
vector field  derivative of the geodesic flow $(\phi^t_{g})_{t\in \re}$ of the metric $g$.\\
\ \\
In the sequel, given $A\subset \re$, $\text{int}\, A$ denotes the interior of $A$.\\
\noindent  Our first result in this context is:
\begin{T}\label{T1}Arbitrarily close to  $g_{_{0}}$ there is an open set $\mathcal{G}\subset${\boldmath$\mathcal{G}$}$^{3}(M)$  such that for any $g\in \mathcal{G}$, there exists  a dense and $C^{2}$-open subset $\mathcal{H}_g\subset C^{2}(S^{g}M,\re)$  so that 
$$\emph{int}\, \cM(\phi_{g},f)\neq \emptyset \ \ \text{and} \ \ \emph{int}\, \cL(\phi_{g},f)\neq \emptyset,$$
whenever $f\in \mathcal{H}_g$.\\
Moreover, the above statement holds persistently: for any $\tilde{g}\in \mathcal{G}$, it holds for any $(f,g)$ in a suitable neighborhood of
$\mathcal{H}_{\tilde{g}}\times \{\tilde{g}\}$ in $C^{2}({M,\re})\times${\boldmath$\mathcal{G}$}$^{3}(M)$. 

\end{T}


The next result is a version of Theorem \ref{T1} for the restricted case, where the set of functions is a composition of functions on the 
manifold $M$ with the canonical projection.
\begin{T}\label{T2}Arbitrarily close to  $g_{_{0}}$ there is an open set 
$\mathcal{G}\subset${\boldmath$\mathcal{G}$}$^{3}(M)$
 such that for any
$g\in \mathcal{G}$, there exists  a dense and $C^{2}$-open subset $\widetilde{\mathcal{H}}_g\subset C^{2}(M,\re)$  so that 
$$\emph{int}\, \cM(\phi_{g},f\circ \pi)\neq \emptyset \ \ \text{and} \ \ \emph{int}\, \cL(\phi_{g},f\circ \pi)\neq \emptyset,$$
whenever $f\in \widetilde{\mathcal{H}}_g$.
Here  $\pi_{g}\colon S^{g}M\to M$ is the canonical projection. 
Moreover, the above statement holds persistently: for any $\tilde{g}\in \mathcal{G}$, it {also} holds for any $(f,g)$ in a suitable neighborhood of $\widetilde{\mathcal{H}}_{\tilde{g}}\times \{\tilde{g}\}$ in $C^{2}({M,\re})\times${\boldmath$\mathcal{G}$}$^{3}(M)$. 
\end{T}
The final result {presented}  in this {paper}  states that 
for small perturbations of the metric, the Lagrange and Markov spectra  defined in \cite{PP} contain intervals. 
More precisely, {we can state this as the following theorem}.
\begin{T}\label{T3} Let $M$ be a complete and non-compact surface. Let $e$ be an \textit{end} of $M$. Then  arbitrarily close to $g$ there 
exists an open set $\mathcal{G}\subset${\boldmath$\mathcal{G}$}$^{3}(M)$  such that for any $\tilde{g}\in \mathcal{G}$,  
$$\emph{int}\, \cM(\phi_{\tilde{g}}, h^{\tilde{g}}_{e}\circ \pi)\neq \emptyset \ \ \text{and} \ \ \emph{int}\, \cL(\phi_{\tilde{g}},h^{\tilde{g}}_{e}\circ \pi)\neq \emptyset,$$
where $h^{\tilde{g}}_{e}$ is the height function associated to end $e$ with the metric $\tilde{g}$.
\end{T}

\noindent The problem of finding intervals in the classical Lagrange and Markov spectra is closely related to the study of the fractal geometry of regular Cantor sets related to the Gauss map. However, the results about the existence of Hall's ray cited above  do not involve fractal geometry or the study of re\-gu\-lar Cantor sets. \\
In the present study of {the} two-dimensional spectra, recent results on fractal geometry of re\-gu\-lar Cantor sets are (again) a key ingredient in the proof of our results about dynamical Lagrange and Markov spectra 
associated with geodesic flows on surfaces of negative curvature. We use and adapt in this work techniques from \cite{MY}, \cite{MY1}, \cite{MR}, and \cite{R}.

\indent The previous theorems are far from being direct consequences of \cite{MR} and \cite{R}. In fact,
the proofs of the previous results depend strongly on local perturbations of diffeomorphisms or vector fields. In the present situation, we need to perform perturbations of the Riemannian
metrics, which produce necessarily global modifications of the geodesic field - we need to develop delicate combinatorial and geometrical techniques in order to emulate local perturbations of the geodesic flow near hyperbolic sets by perturbing the Riemannian metrics.

\indent Indeed, as it is well known, perturbation results within the set of Riemannian metrics are usually hard, basically because when we change the metric in a neighborhood of a point of
the manifold we affect all the geodesics intersecting this neighborhood; in other words, even if the size of our neighborhood in the manifold is small, the effect of the perturbation in the unit tangent bundle could be large. However, the main difficulty in this part of the work is to produce, just by performing small perturbations on the metrics, a sufficiently rich family of perturbations (with ``enough independence") of a horseshoe invariant by a Poincar\'e map associated with the geodesic flow in order to produce, via the probabilistic method of \cite{MY1}, stable intersections of the stable Cantor set of this horseshoe and the unstable Cantor set of another invariant horseshoe {(see Subsections \ref{PropV} and \ref{RCS} for the definition of regular Cantor sets and the stable intersection of two Regular Cantor sets)}. We develop new techniques of Fractal Geometry of combinatorial nature  to address this problem. 

The paper is organized as follows: In Section \ref{Prel}, we recall some classical results of hyperbolic dynamics
which are relevant to {our}  work. In Section \ref{HD>1}, we use a
construction {from} 
\cite[Lemma 4.10, and Corollary 5]{R} to get a basic set for the geodesic flow with Hausdorff dimension  close enough to $3$. Using this
basic set, we construct a finite number of (disjoint) cross sections to the geodesic flow, and we show that the Poincar\'e 
(first return) map of
the union of sections has a basic set - a horseshoe - with Hausdorff dimension close enough to $2$ (see \cite[Section 3]{R}). The last
arguments allow us to reduce Theorem \ref{T1} and Theorem \ref{T2} to a bi-dimensional version (see Theorem
\ref{Theorem 1} and Theorem 
\ref{Theorem 2}). In Section \ref{sec MLSGF}, we construct explicitly
the set $\mathcal{H}_{g}$ and $\widetilde{\mathcal{H}}_{g}$, which has a non-trivial construction. In Section \ref{Family of Pert}, we develop techniques of perturbations of Riemannian metrics together with further combinatorial techniques in order to adapt constructions of \cite{MY1} in the context of our work, which  {most}  important and {hardest} part of {this} paper.  In  Section \ref{P. MT},
using the results of Section \ref{HD>1}, \cite{MR}, \cite{MY}, \cite{MY1} and some combinatorial techniques develop in \cite[Section 4.3.2]{R},
we prove Theorem \ref{Theorem 1} and Theorem \ref{Theorem 2}. Finally,
in Section \ref{SHF}, we prove Theorem \ref{T3}. \\


\section{Preliminaries}\label{Prel}
Let $N$ be a complete Riemannian manifold and $\phi^t:N \rightarrow N$ a $C^{r}$-flow on $N$. We say that an invariant set $\Lambda\subset N$ is \emph{hyperbolic} for $\phi^{t}$ if: there exists a  continuous splitting $T_{\Lambda}N=E^{s}\oplus \phi\oplus E^{u}$ such that  for each $\theta\in \Lambda$
\begin{eqnarray*}
	d\phi^t_{\theta} (E^s(\theta)) &=& E^s(\phi^t(\theta)),\\
	d\phi^t_{\theta} (E^u(\theta)) &=& E^u(\phi^t(\theta)),\\
	||D\phi^t_{\theta}\big{|}_{E^s}|| &\leq& C \lambda^{t},\\
	||D\phi^{-t}_{\theta}\big{|}_{E^u}|| &\leq& C \lambda^{t},\\
	\end{eqnarray*}
for all $t\geq 0$ with $ C > 0$ and $0 < \lambda <1$, where  $\phi$ is the vector field derivative of the geodesic vector flow.\\
The bundles $E^{s}$ and $E^{u}$ are called stable and unstable bundles, respectively. They are known to be uniquely integrable (see \cite{K}).\\ A hyperbolic set $\Lambda$ is called a \emph{basic set} if satisfies the following three properties:
\begin{enumerate}
\item[(a)] the set of periodic orbits contained in $\Lambda$ are dense in $\Lambda$,
\item[(b)] $\phi^t$ is transitive,
\item[(c)] There is an open set $U\supset \Lambda$ so that $\Lambda=\bigcap_{t\in \re}\phi^t(U)$.
\end{enumerate}
Basic sets play an important role in this paper because they have well-understood  fractal geometry. For diffeomorphisms on a surface, no trivial basic sets are also called \emph{horseshoe}.
When the entire manifold $N$ is a hyperbolic set, then we said that the flow is an Anosov flow.  The central examples of Anosov flows, which will be treated in this work, are provided by geodesic flows. 
\subsection{Geodesic Flow}\label{GF}

Given a Riemannian manifold $M$, denoted by $TM$ the tangent bundle and
$SM=\{(x,v)\in TM:\Vert v\Vert=1\}$ the unit tangent bundle of $M$. 
For $\theta=(p,v)\in SM$, we denoted by $\gamma_{_{\theta}}(t)$ the unique geodesic with initial conditions $\gamma_{_{\theta}}(0)=p$ and $\gamma_{_{\theta}}'(0)=v$. The family of diffeomorphisms  $\phi^t:SM \to SM$  given by  $\phi^t(\theta)=(\gamma_{_{\theta}}(t),\gamma_{_{\theta}}'(t))$ and satisfying $\phi^{t+s}=\phi^{t}\circ \phi^{s}$ for all $t,s\in \mathbb{R}$ is called  \emph{geodesic flow}. We endowed $SM$ with the \emph{Sakaki metric} (see Subsection \ref{Sasaki}).\\
A classic result due to D. Anosov (cf. \cite{A}, \cite{K1} and \cite{K}) states that  complete manifolds of curvature bounded between two negative constants (pinched negative curvature) have Anosov geodesic flow. Moreover, in this condition, if $M$ has finite volume, then the
non-wandering set of the geodesic flow  is equal to $SM$, and $\phi^t$ is transitive. (cf.  \cite{P} and \cite[chapter 3]{K1}).\\
For the purposes of this paper, from now on, we consider $(M,g_{_{0}})$ a complete $C^r$, {$r\geq 3$}, a  Riemannian surface with pinched negative curvature,  $SM$ the unit tangent bundle, which has dimension $3$, and whose geodesic flow ${\phi}^{t}=(\phi^{t}_{g_{_{0}}})_{_{t\in\re}}$ is Anosov.  


\section{Hyperbolic Set for Two-dimensional Dynamic}\label{HD>1}

The goal of this section is to reduce  Theorem \ref{T1} and Theorem \ref{T2} to  two-dimensional versions (see Theorem  \ref{Theorem 1} and Theorem \ref{Theorem 2}).

 
\subsection{Dimension Reduction via Poincar\'e Maps}\label{Dim. Red}
The geodesic flow of a complete surface of pinched negative curvature and finite volume carries many basic sets. In \cite{R}  was proved that, in this case, the geodesic flow has a basic set with Hausdorff dimension close to $3$.
  
\begin{Le}\label{L1-HD>1}\emph{\cite[Corollary 5]{R}}
There is a basic set $\Lambda$ for $\phi^{t}$ with Hausdorff dimension arbitrarily  close to $3$.
\end{Le}
To reduce the dynamic of $\phi^t$ to two dimensional dynamical, we need the concept of \emph{Good Cross-Section} introduced at \cite[Section 3]{R}.
\begin{Defi}\label{D1GCS} Let $\Lambda$ be a {compact} subset of $M$.\\
We say that a compact cross-section $S$ is a \emph{Good Cross-Section} $($or simply \emph{GCS}$)$ for $\Lambda$ if 
\begin{center}
$d(\Lambda \cap S, \partial\,S)>0$
\end{center}
where $d$ is the intrinsic distance in $S$.\\
If $S$ is a \emph{GCS} and $x\in \Lambda\cap S$, we say that $S$ is a GCS at $x$.
\end{Defi}
The transitivity of the geodesic flow allows us to construct GCS's for any points of a hyperbolic set (see \cite[Section 3]{R} for more details). Furthermore, its specific constructions allow to prove the following lemma, which allows reducing the dynamics of $\Lambda$ to a Poincar\'e map (see also \cite[Remark 7]{R}).

\begin{Le}\emph{\cite[Lemma 3.5]{R}}\label{L7}
There are $\gamma>0$ and smooth-GCS, ${\Sigma}_{i}$, $i=1,\dots, k$ such that 

\begin{equation}\label{eq6}
\ds\Lambda\subset \bigcup^{k}_{i=1}\phi^{(-\gamma,\gamma)}(\emph{int}\,{\Sigma}_i)
\end{equation}
with ${\Sigma}_i\cap {\Sigma}_j=\emptyset$.
\end{Le}

As was exploited in \cite[Section 3.2]{R}, we describe the dynamics of $\phi^t$ on $\Lambda$  in terms of Poincar\'e maps. 
More precisely, the GCS $\Sigma_i$, $1\leq i\leq k$, of Lemma \ref{L7} was constructed such that 
the $\phi^t$-orbit of any point of $\Lambda$ intersects $\Sigma:=\bigsqcup\limits_{i=1}^k \Sigma_i$, the subset $\Delta:=\Lambda\cap\Sigma$ is
disjoint from the boundary $\partial\Sigma:=\bigsqcup\limits_{i=1}^k\partial\Sigma_i$. Consider the Poincar\'e (first return) map $\mathcal{R}:D_{\mathcal{R}}\to \Sigma$ from a neighborhood $D_{\mathcal{R}}\subset\Sigma$ of $\Delta$ to $\Sigma$ sending $y\in D_{\mathcal{R}}$ to the point  $\mathcal{R}(y) = \phi^{t_+(y)}_{g_{_{0}}}(y)$ where the forward $\phi^t$-orbit of $y$ first hits $\Sigma$, and $\cR^{-1}(y)= \phi_{g_0}^{t_{-}(y)}(y)$, whenever $\cR^{-1}(y)$ is defined. Since $\Lambda$ is a basic set for $\phi^t$, it is not difficult to prove that
\begin{Le}\label{L7.1}
The set $\Delta$ is a basic set for $\mathcal{R}$.
\end{Le}

The splitting $E^{s}\oplus \phi\oplus E^{u}$ over a neighborhood of $\Lambda$ defines a continuous splitting $E^{s}_{\Sigma}\oplus E^{u}_{\Sigma}$ of the tangent bundle $T\Sigma$ given by 
\begin{eqnarray*}
E^{s}_{\Sigma}(y)=E^{cs}_{y}\cap T_{y}\Sigma \ \text{and} \ E^{u}_{\Sigma}(y)=E^{cu}_{y}\cap T_{y}\Sigma,
\end{eqnarray*}
where $E_{y}^{cs}=E^{s}_y\oplus  \text{span}\,{\phi(y)}$ and $E_{y}^{cu}=E^{u}_y\oplus \text{span}\,{\phi(y)}$. Denote by $W^{s}_{\cR}(x)$ the stable manifold and $W^{u}_{\cR}(x)$ the unstable manifold of $x\in \Delta$ (see Appendix \ref{HF}).

The relation between the Hausdorff dimension of $\Delta$ and $\Lambda$ is described by the following lemma (compare with Lemma 3.8 
{in} \cite{R}): 

\begin{Le}\emph{\cite[Lemma 3.8]{R}}\label{l.dimK} In the previous setting, one has $\textrm{HD}(\Lambda)=\textrm{HD}(\Delta)+1$. In particular, $HD(\Delta)>1$.
\end{Le}
Let $\mathcal{V}\subset$ {\boldmath$\mathcal{G}$}$^{3}(M)$ a neighborhood of $g_{0}$ such that for every $g\in \mathcal{V}$, the hyperbolic set $\Lambda$ has a hyperbolic continuation $\Lambda_{g}$ for $\phi_g=(\phi^{t}_{{g}})_{t\in \re}$, the geodesic flow  of the metric $g$, and $HD(\Lambda_{g})>2$. 

Given $g, g_{_{1}} \in \mathcal{V}$, we define the diffeomorphism $\mathcal{S}^{g}_{g_{_{1}}}\colon S^{g}M \to S^{g_{_{1}}}M$ by $$\mathcal{S}^{g}_{g_{_{1}}}(x,v)=\bigg(x, \frac{v}{||v||_{g_{_{1}}}}\bigg),$$ where $S^{{g}}M$ and $S^{g_{_{1}}}M$ are 
the unit tangent bundle associated to metric ${g}$ and $g_{_{1}}$, respectively. It is easy to check $(\mathcal{S}^{g}_{g_{_{1}}})^{-1}=\mathcal{S}^{g_{_{1}}}_{g}$. In order not to saturate the notation, whenever $g_{_{1}}=g_{_{0}}$, we write $\mathcal{S}^{g}_{g_{_{0}}}:=\mathcal{S}^{g}$ and define the transformation $\mathcal{T}^{g}$ as the inverse of $\mathcal{S}^{g}_{g_{_{0}}}$, \emph{i.e.}, and $\mathcal{T}^{g}:= \mathcal{S}^{g_{_{0}}}_{g}$. Put   $\Sigma_{g}:=\mathcal{T}_{g}(\Sigma)$ and $\cR^{g}\colon \Sigma_{g} \to \Sigma_{g}$ the Poincar\'e map of $\phi^t_{g}$.\\
The diffeomorphism $\mathcal{S}^{g}_{g_{_{1}}}$ is $C^2$. Moreover, if $\mathcal{V}$ is small enough, then $\mathcal{S}^{g}_{g_{_{1}}}$ is $C^2$-close to the identity, in the following sense: Let
$(x,v)\in S^{{g}}M$ and $\psi_{{g}}\colon U\subset \mathbb{R}^{3} \to S^{{g}}M$ a chart of $S^{{g}}M$, with
$\psi_{{g}}(0)=(x,v)$ and $\psi_{g_{_{1}}} \colon V\subset \mathbb{R}^{3} \to S^{g_{_{1}}}M$ a chart of $S^{g_{_{1}}}M$ such that $\psi_{g_{_{1}}}(0)=\cS^{g}_{{g_{_{1}}}}(x,v)$, then 
$\psi^{-1}_{g}\circ \cS^{g}_{{g_{_{1}}}} \circ \psi_{{g}} \colon U \to V$ is $C^{2}$-diffeomorphism close to the identity of $\mathbb{R}^3$.

Note also that, $J_{g_{_{1}}}^{g}(t,\cdot)= \cS^{g}_{g_{_{1}}}\circ \phi^t_{g}\circ \cS^{g_{_{1}}}_{g}(\cdot)$ defines a flow on $S^{g_{_{1}}}M$. Since $g$ and $g_1$ are $C^3$-close, then $J_{g_{_{1}}}^{g}(t,\cdot)$ is $C^2$-close to $\phi^{t}_{g_{_{1}}}$, and therefore, the Poincar\'e map $\cR^{J_{g_{_{1}}}^{g}}\colon \Sigma_{g_{_{1}}}\to \Sigma_{g_{_{1}}}$, associated to $J_{g_{_{1}}}^{g}$, is $C^2$-close to $\cR^{g_{_{1}}}$.\\
If $g_{_{0}}=g_{_{1}}$, then $\cR^{{J^{^{g}}_{g_{_{0}}}}}\colon \Sigma\to \Sigma$ is $C^2$-close to $\cR$ and we denoted $\widetilde{\Delta}_g$ the hyperbolic continuation of $\Delta$ by $\cR^{{J^{^{g}}_{g_{_{0}}}}}$, in fact, a basic set.
\begin{Le}\label{lemma_conjugate}
The diffeomorphisms $\cR^{{J^{^{g}}_{g_{_{0}}}}}$ and $\cR^{{g}}$ are conjugate, \emph{i.e.}, 
$$\mathcal{T}^{g}\circ \cR^{{J^{^{g}}_{g_{_{0}}}}}= \cR^{g}\circ \mathcal{T}^{g}$$
\end{Le}
\begin{proof}
Simply we observe that $\cS^{g}\circ \phi_g^{t}(x,v)= J_{g_{_{0}}}^{g}(t, \cS^{g}(x,v))$ and therefore $$\phi^t_{g}\circ \mathcal{T}^g= \mathcal{T}^{g}\circ J_{g_{_{0}}}^{g}(t, \cdot).$$
\end{proof}
The last lemma implies that the set $\Delta_g:=\mathcal{T}^{g}(\widetilde{\Delta}_{g})$ is a basic set for $\cR^g$. Without loss of generality, from now on, we call $\Delta_g$ \emph{the hyperbolic continuation} of $\Delta$ by $\cR^g$.
If $\mathcal{V}$ is small enough, then from the continuity of the Hausdorff dimension for basic sets (cf. \cite{PT}), we have that $HD(\Delta_g)>1$.  Also, 
given a basic set $\Theta$ for $\cR$ we denote $\Theta_g$ the hyperbolic continuation of $\Theta$ by $\mathcal{R}^{g}$. \\
\begin{R}\label{R1-conjugates}
\indent As mentioned before, if $g, g_1\in \mathcal{V}$, then  $\cR^{J_{g_{_{1}}}^{^{g}}}\colon \Sigma_{g_{_{1}}}\to \Sigma_{g_{_{1}}}$ is $C^2$-close to $\cR^{g_{_{1}}}$, then we call $\widetilde{\Delta}_{g_1,g}$ the hyperbolic continuation of $\Delta_{g_{1}}$ by $\cR^{J_{g_{_{1}}}^{^{g}}}$. Analogously, as \emph{Lemma \ref{lemma_conjugate}}, we have that $\cR^{J_{g_{_{1}}}^{^{g}}}$ and $\cR^g$ are conjugate and 
\begin{equation}\label{eq1-conjugates}
\mathcal{S}^{g_{_{1}}}_{g}\circ \cR^{{J^{^{g}}_{g_{_{0}}}}}= \cR^{g}\circ \mathcal{S}^{g_{_{1}}}_{g},
\end{equation}
which implies that $\Delta_{g_{_{1}},g}:= \mathcal{S}^{g_{_{1}}}_{g}(\widetilde{\Delta}_{g_{_{1}},g})$ is a basic set for $\cR^g$ with $HD(\Delta_{g_{_{1}},g})>1$. 
\end{R}

From now on, we only consider the  metrics on $\mathcal{V}$.

\subsubsection{The dynamical Lagrange and Markov spectra of $\Lambda$ and $\Delta$}

\indent The dynamical Lagrange and Markov spectra of $\Lambda$ and $\Delta$ are related in the following way.
Given a function $F\in C^{s}(SM,\mathbb{R})$, $s\geq 1$, let us denote by $f =\max F_{\phi}\colon D_{\cR}\to\re$ the function
$$\text{max} F_{\phi}(x):=\max_{0\leq t \leq t_{+}(x)}F(\phi^{t}(x)).$$
\begin{R}
{We have that} $f =\max F_{\phi}$ might not be $C^1$ in general. 
\end{R}
{We see that}
$$\limsup_{n\to +\infty}f(\cR^n(x))=\limsup_{t\to + \infty}F(\phi^t(x))$$
and 
$$\sup_{n\in \mathbb{Z}}f(\cR^n(x))=\sup_{t\in \re}F(\phi^t(x))$$
for all $x\in \Delta$. In particular,

\begin{equation}\label{Continuos and Discrete}
\cL(\phi, \Lambda, F)=\cL(\cR, \Delta, f) \ \ \text{and} \ \ \cM(\phi, \Lambda, F)=\cM(\cR, \Delta, f).
\end{equation}
Thus, the relation (\ref{Continuos and Discrete})  reduces Theorem \ref{T1}, and Theorem \ref{T2} to the following theorems:
\begin{T}\label{Theorem 1}In the setting of Theorem $\ref{T1}$, arbitrarily close to $g_{_{0}}$ there is an open set $\mathcal{G}\subset${\boldmath$\mathcal{G}$}$^{3}(M)$  such that for any $g\in \mathcal{G}$ one can find a dense and open $C^{2}$-open subset  $\mathcal{H}_{g, \Lambda}\subset C^{2}(S^gM,\re)$, so that 
$$\emph{int}\, \cM(\cR^{g},\Delta_{g}, \max F_{\phi_g})\neq \emptyset \, \ \ \  \, \text{and}\, \ \ \  \, \emph{int}\, \cL(\cR^{g},\Delta_{\phi_g}, \max F_{\phi_g})\neq \emptyset$$
whenever $F\in {\mathcal{H}_{g,\Lambda}}$.\\
Here $\Delta_{g}$ denoted the hyperbolic continuation of  $\Delta$ by the Poincar\'e map $\cR^{g}$.
\end{T}

\begin{T}\label{Theorem 2}In the setting of \emph{Theorem $\ref{T2}$}, arbitrarily close to $g_{_{0}}$ there is an open set 
$\mathcal{G}\subset${\boldmath$\mathcal{G}$}$^{3}(M)$ such that for any $g\in \mathcal{G}$ one can find a dense and open $C^{2}$-open subset 
$\widetilde{\mathcal{H}}_{g, \Lambda}\subset C^{2}(M,\re)$, so that 
$$\emph{int}\, \cM(\cR^{g},\Delta_{g}, \max (f\circ\pi)_{\phi_g})\neq \emptyset \,\ \   \text{and}\,\ \   \,  \, \emph{int}\, \cL(\cR^{g},\Delta_{g}, \max (f\circ\pi)_{\phi_g})\neq \emptyset$$
whenever $f\in {\widetilde{\mathcal{H}}_{g,\Lambda}}$. 
Here $\Delta_{g}$ denotes the hyperbolic continuation of  $\Delta$ by the Poincar\'e map $\cR^{g}$.
\end{T}
The following sections focus on doing the proofs of the above theorems.
\section{Construction of the Typical Functions}\label{sec MLSGF}
 

In this section, we construct explicitly the section of real functions that work for  Theorem \ref{Theorem 1} and Theorem \ref{Theorem 2}.


\subsection{Construction of Typical Functions for Theorem \ref{T1}}\label{Description of function H_1}
The construction of the set of functions {$\mathcal{H}_{g, \Lambda}$} of Theorem \ref{T1} will be similar to the construction given in  \cite[Section 4.3 - Lemma 4.9]{R} with some minor changes. \\
The following properties will be very useful for our construction.
\begin{Defi}\label{s(u) beta property}
Let $\Theta$ be a basic set for $\cR$ and $\beta>0$ small, we say that a sub-horseshoe $\widetilde{\Theta}$ of $\Theta$ has the $\beta$-\emph{stable property} if 
$$HD(K_{\widetilde{\Theta}}^{s})> HD(K^{s}_{\Theta})-\beta,$$
where $K_{\widetilde{\Theta}}^{s}$, $K^{s}_{\Theta}$ are the stable Cantor sets associated to $\widetilde{\Theta}$ and $\Theta$ \emph{(see Section \ref{sec EMAH} to the definition of stable and unstable Cantor sets)}.\\
The definition of the $\beta$-\emph{unstable property} is analogous, using unstable Cantor sets instead of stable Cantor sets.
\end{Defi}
\begin{R}\label{MAIN-REMARK}
It is worth noting that the parameter $\beta$ in the definition above is small, \emph{i.e.}, $0<\beta<\max\{HD(K^{s}_{\Theta})/2, HD(K^{u}_{\Theta})/2\}$. In fact, all the lemmas in this section that involve the parameter $\beta$ work in this interval. However, in \emph{Section \ref{ProV}}, we use the basic set $\Delta$ and an interval of the parameters $\beta$, which is related to  $HD(K^{s}_{\Delta})$ and $HD(K^{u}_{\Delta})$, more specifically, since $HD(\Delta)=HD(K^{s}_{\Delta})+ HD(K^{u}_{\Delta})$ $($see \emph{Subsection \ref{sec EMAH}}$)$,  we take $\beta$ such that
$$HD(\Delta)-6\beta=HD(K^{s}_{\Delta})+ HD(K^{u}_{\Delta})-6\beta >1,$$
where $\Delta$ is the basic set given by \emph{Lemma \ref{L7.1}, which satisfies Lemma \ref{l.dimK}}.  \\
\end{R}
\begin{R}\label{MAIN-REMARK-1}
All results of this section will be valid for any basic set $\Theta$ of $\cR$. In \emph{Section \ref{Family}} we use all these results for the specific basic set $\Delta$.
\end{R}
\noindent Given a basic set $\Theta$ for $\cR$, a Markov partition $R$ of $\Theta$,  we define the set 
\begin{equation}\label{E1-Descrip}
{H}_{1}(\cR,\Theta)=\left\{f\in C^{1}(\Sigma\cap R,\re):\#M_{f}(\Theta)=1, \  z\in M_{f}(\Theta), \ D\cR_{z}(e_{z}^{j})\neq 0, j=s,u\right\},
\end{equation}
where $M_{f}(\Theta):=\{z\in \Theta: f(z)\geq f(x)\ \text{for all} \ x\in \Theta\}$, the set of maximum points of $f$ in $\Theta$ and $e_{z}^{j}\in E^{j}_{\Sigma}(z)$ is a unit vector, $j=s,u$. (cf. \cite[Section 3]{MR}).\\

A notion that will be useful in the next sections is the concept of the Markov partition of a horseshoe, to make the text easier to read we put its definition in the Appendix (see Subsection \ref{sec EMAH}).
\begin{Defi}\label{DLmax00} 
We say that  $F\in\mathcal{H}_{\Theta, \beta}^{s}\subset C^2(S^{g_{_{0}}}M,\re)$ if there is a sub-horseshoe $\Theta_F^{s}$ of $\Theta$ with the $\beta$-stable property and there is a Markov partition $R_{F}^{s}$ of $\Theta_{F}^{s}$ such that the function $\max  (F)_{\phi}|_{\Sigma\cap R_{F}^{s}}\in H_1(\cR,\Theta_F^{s})\subset C^{1}(\Sigma\cap R_{F}^{s},\re)$.
The definition for the unstable case $\mathcal{H}_{\Theta, \beta}^{u}$ is analogous $($see \emph{Subsection \ref{sec EMAH}} for the definition of Markov Partition$)$.
\end{Defi}

\begin{Le}\label{set of functions}{\emph{\cite[Lemma 4.9]{R}}} For any basic of $\Theta$ of $\mathcal{R}$ and  $\beta>0$ small, we have that ${\mathcal{H}}_{\Theta,\beta}^{j}$, 
is dense and $C^2$-open set,  $j=s,u$. Consequently, for  $g\in \mathcal{V}$ there is $\tilde{\beta}>0$ such that  $\mathcal{H}_{\Theta_g, \tilde{\beta}}^{j}$  is a dense and $C^2$-open set, $j=s,u$.
\end{Le}
\begin{R}\label{RDLmax00'}
By the construction at \emph{\cite[Section 3]{MR}} and the continuity of the Hausdorff dimension,  it is easy to see that, if $\mathcal{V}$ is small enough, then for $g\in \mathcal{V}$  there is $\tilde{\beta}$ \emph{(}which depends on $g$\emph{)} such that 
$$\mathcal{H}_{\Theta_g, \tilde{\beta}}^{j}=\bigg\{F\circ \mathcal{S}_g: F\in \mathcal{H}_{\Theta, \beta}^{j}\bigg \}, \, \, \text{for} \, \, j=s,u.$$
\end{R}
\begin{R}\label{Remark-principal-0} The proof of \emph{Lemma \ref{set of functions}} also allow us conclude the following: If $\Theta$ and $\Gamma$ are two basic sets of $\cR$, then the set ${\mathcal{H}}_{\Theta, \Gamma, \beta}\subset {\mathcal{H}}_{\Theta,\beta}^{s}\cap {\mathcal{H}}_{\Gamma,\beta}^{u}$, which satisfies
$$\max  (F)_{\phi}|_{\Sigma\cap (R_{F,\Theta}^{s}\cup R_{F,\Gamma}^{u})}\in H_1(\cR,(\Theta_{F}^{s}\cup \Gamma_{F}^{u}))\subset C^{1}(\Sigma\cap (R_{F,\Theta}^{s}\cup R_{F,\Gamma}^{u}),\re)$$ 
is dense and $C^2$-open.\\
Here $R_{F,\Theta}^{s}$, and $R_{F,\Gamma}^{u}$ are the Markov partition of $\Theta_{F}^{s}$ and $\Gamma_{F}^{u}$ given by the definition of ${\mathcal{H}}_{\Theta,\beta}^{s}$ and ${\mathcal{H}}_{\Gamma,\beta}^{u}$, respectively. \\
\indent Analogously, since \emph{Remark \ref{RDLmax00'}}, it is easy to see that, if $\mathcal{V}$ is small enough, then  for any $g\in \mathcal{V}$  there is $\tilde{\beta}$ \emph{(}which depends on $g$\emph{)} such that 
\begin{equation}\label{Eq-aux-0}
{\mathcal{H}}_{\Theta_g, \Gamma_g, \tilde{\beta}}=\{F\circ \mathcal{S}_g: F\in \mathcal{H}_{\Theta, \Gamma, \beta}\},
\end{equation}
and consequently ${\mathcal{H}}_{\Theta_g, \Gamma_g, \tilde{\beta}}$ is a dense and $C^2$-open set, 
where $\Theta_g$ and $\Gamma_g$ is the hyperbolic continuation of $\Theta$ and $\Gamma$ by $\cR_{g}$, respectively $($see  \emph{Subsection \ref{Dim. Red}}$)$.
\end{R}



\subsection{Construction of Typical Functions for Theorem \ref{T2}}\label{Description of function H_2}
The construction of ${\widetilde{\mathcal{H}}_{g,\Lambda}}$ is a little more complicated than the construction of {${\mathcal{H}}_{g,\Lambda}$}, although we will use some tools developed at {\cite[Lemma 4.9]{R}} to make $\text{max}(f\circ \pi)_{\phi}$ a $C^1$-function, we have an additional problem because the function $f\circ \pi$ is constant along of  the fiber of $SM$.\\
 For the construction of $\widetilde{\mathcal{H}}_{g, \Lambda}$ we need some auxiliary sets of functions, which give some differentiability and good properties to $\max (f\circ\pi)_{\phi}$. \\ 
 \indent Let $\Theta$ be a basic set for $\mathcal{R}$, and $\beta>0$ small (as Remark \ref{MAIN-REMARK}), we consider  two  one-parameter families of sets of functions, $\widetilde{\mathcal{H}}_{\Theta, \beta}^{s}, \widetilde{\mathcal{H}}_{\Theta, \beta}^{u} \subset C^2(M, \re)$,  defined as follows 
 
 $$\widetilde{\mathcal{H}}_{\Theta, \beta}^{j}:=\bigg\{f\in C^2(M,\re): f\circ \pi \in \mathcal{H}_{\Theta, \beta}^{j}\bigg\}, \ \, j=s,u.$$
In other words, $f\in \widetilde{\mathcal{H}}_{\Theta, \beta}^{s}$ if there is  a sub-horseshoe $\Theta_{f}^{s}$ of  $\Theta$  and a Markov partition $R_{f,\Theta}^{s}$ of $\Theta_{f}^{s}$ such that $\Theta_{f}^{s}$ has the $\beta$-stable property and $\max  \,(f\circ \pi)_{\phi}|_{\Sigma\cap R_{f,\Theta}^{s}}\in H_1(\cR,\Theta_{f}^{s})\subset C^{1}(\Sigma\cap R_{f,\Theta}^{s},\re)$. The stable case is analogous.
Similar to Remark \ref{RDLmax00'} we have 
\begin{R}\label{RDLmax0001}
It is easy to see that, for all metric $g$, $\pi_g= \pi\circ \mathcal{S}^{g}$. So, if $\mathcal{V}$ is small enough, then the construction at \emph{\cite[Section 3]{MR}} and the continuity of the Hausdorff dimension, provides that for $g\in \mathcal{V}$  there is $\tilde{\beta}$ \emph{(}which depends on $g$\emph{)} such that 
$$\widetilde{\mathcal{H}}_{\Theta_g, \tilde{\beta}}^{j}=\widetilde{\mathcal{H}}_{\Theta, \beta}^{j}, \, \, j=s,u$$
where $\Theta_g$ is the hyperbolic continuation of $\Theta$ by $\cR_{g}$ $($see  \emph{Section \ref{Dim. Red}}$)$.
\end{R}
Although the set $\mathcal{H}_{\Theta, \beta}^{j}$ is dense and $C^2$-open in $C^2(SM,\re)$, it does not easily follow that $\widetilde{\mathcal{H}}_{\Theta, \beta}^{j}$ is dense and $C^2$-open in $C^2(M,\re)$, $j=s,u$. Thus, the remainder of this section is devoted to showing that

\begin{Le}\label{Lmax4} 
For any basic $\Theta$ of $\mathcal{R}$ and $\beta>0$ small, we have that $\widetilde{\mathcal{H}}_{\Theta, \beta}^{j}$
is dense and $C^2$-open set, $j=s,u$. Consequently, for  $g\in \mathcal{V}$,  $\widetilde{\mathcal{H}}_{\Theta_g, \tilde{\beta}}^{j}$ is a dense and $C^2$-open set, $j=s,u$.
\end{Le}

\noindent The proof of this lemma is extensive and we present some auxiliary lemmas to do it.\\

The first important ingredient, from \cite[Lemma 4.6]{R}, for all constructions in this section will be the following lemma. Recall that $\Sigma=\bigsqcup\limits_{i=1}^k \Sigma_i$, then 
\begin{Le}\emph{(Lemma 2.6 - \cite{R})}\label{LIC}\\
Let $\alpha=\{\alpha_i:[0,1]\to \Sigma, i\in\{1,\dots,m\}\}$ be a finite family of $C^{1}$-curves and $\Theta$ a basic set for $\cR$. Then for all $\epsilon>0$ there is sub-horseshoe $\Theta_{\alpha}^{s}$ of $\Theta$ such that $\Theta_{\alpha}^{s} \cap \alpha_i([0,1])=\emptyset$, for any $i\in\{1,\dots,m\}$ and $\Theta_{\alpha}^{s}$ satisfies the $\epsilon$-stable property with respect to $\Theta$, that is,

$$HD(K_{\alpha}^{s})\geq HD(K^s_{\Theta})-\epsilon,$$
\noindent where $K_{\alpha}^{s}$, $K^{s}_{\Theta}$ are the stable regular Cantor sets that of  $\Theta_{\alpha}^{s}$ and $\Theta$, respectively. An analog result to the unstable case, using the $\epsilon$-unstable property.
 \end{Le} 

This lemma will allow us to delete suitable subsets of $\Theta$ without losing too much of the Hausdorff dimension of $\Theta$.

\begin{Defi}\label{DLmax2} 
Let $U$ be an open set of $SM$, we say that $f\in \mathcal{N}_{U,\Theta, \beta}^{s}\subset C^2(M,\re)$ if there is a sub-horseshoe $\Theta_f^{s}$ of $\Theta$ with the $\beta$-stable property and a Markov partition $R_{f, \Theta}^{s}$ of $\Theta_{f}^{s}$ so that for each $(x,v)\in \Theta_f^{s}$ we have 
\begin{equation*}
\#\{t\in \left(0,t_{+}(x,v)\right): \ \phi^{t}(x,v)\in U \, \, \text{and} \, \, t\, \, \text{is a critical point of} \ \ f\circ\pi({\phi}^{t}(x,v))\}<\infty,
\end{equation*}
where  $\#\, A$ denotes the cardinality of the set $A$.\\
The definition for the unstable case $\mathcal{N}_{U,\Theta, \beta}^{u}$ is analogous.
\end{Defi} 
\noindent When $U=SM$, then we denote $ \mathcal{N}_{SM,\Theta, \beta}^{j}:= \mathcal{N}_{\Theta, \beta}^{j}$, $j=s,u.$

\begin{Le}\label{Lmax1}
For any $\beta>0$ small, the set $\mathcal{N}_{\Theta,\beta}^{j}$ is dense and $C^2$- open, $j=s,u$.
\end{Le}
\noindent To prove this lemma we look at  specific open sets $U$ of $SM$ as below:
Let $U$ be an open set in $SM$ such that, 
\begin{enumerate}
\item[(i)] The closure, $\bar{U}$, of $U$ is contained in the image of a parametrization $\tilde{\varphi}\colon  V\times I\to SM$, where $V\subset \re^2$ and $I$ is an interval. For instance, if $\varphi$ is a parametrization of $M$, $\varphi\colon V\to M$,  such that the set $\{\frac{\partial \varphi}{\partial x},\frac{\partial \varphi}{\partial y}\}$ is an orthonormal basis and  $\widetilde{\varphi}(x,y,z)=(\varphi(x,y), \cos z\,\frac{\partial \varphi}{\partial x}+\sin z\,\frac{\partial \varphi}{\partial y})$.
\item[(ii)] $U\cap \ds\bigcup_{(x,v)\in \Theta}\, \, \bigcup_{t\in [0,t_{+}(x,v)]}\phi^{t}(x,v)\neq \emptyset$.
\end{enumerate}


\begin{Le}\label{Lmax2'}
If $U$ satisfies \emph{(i)} and \emph{(ii)}, then the set $\mathcal{N}_{U, \Theta, \beta}^{j}$ is dense and $C^2$-open, $j=s,u$.
\end{Le}
\begin{proof} We prove the stable case, since the unstable case is analog.
The openness is a consequence of the definition of $\mathcal{N}_{U, \Theta, \beta}^{s}$. Our task is to prove the density.
Let $f\in C^{\infty}(M,\re)$ and put  $F=f\circ \pi$. Then using the  local coordinates given by (i) we can write 
$F(x,y,z)=f\circ \pi \circ \widetilde{\varphi}(x,y,z)=f(\varphi(x,y))$ and the vector field $\phi$ we write as  $\phi(x,y,z)=(X_{1}(x,y,z),X_{2}(x,y,z),X_{3}(x,y,z))$.

\noindent Consider now the set 
\begin{eqnarray*}
S&=& \left\{(x,y,z):\left\langle \nabla F(x,y,z), \phi(x,y,z))\right\rangle=0\right\}\\
&  & \\
&=&\left\{(x,y,z):\frac{\partial f}{\partial x}X_{1}(x,y,z)+\frac{\partial f}{\partial y}X_{2}(x,y,z)=0\right\} \ \ \ \text{and}
\end{eqnarray*}
\begin{eqnarray*}
H&=&\left\{(x,y,z):\left\langle \text{Hess}F(x,y,z)\phi(x,y,z),\phi(x,y,z)\right\rangle=0\right\},
\end{eqnarray*}
where $\text{Hess}F$ is the Hessian matrix of $F$, given by $$\text{Hess}F(x,y,z)=\ds \left( \begin{array}{ccc}
\frac{\partial^{2}{f}}{\partial{x}^{2}} & \frac{\partial^{2}{f}}{\partial{x}\partial y} & 0 \\
\frac{\partial^{2}{f}}{\partial{x}\partial y}& \frac{\partial^{2}{f}}{\partial{y}^{2}} & 0\\
0 & 0 & 0\end{array} \right).$$\\
We have $$H=\left\{(x,y,z):\frac{\partial^{2}{f}}{\partial{x}^{2}}X_{1}^{2}+2\frac{\partial^{2}{f}}{\partial{x}\partial y}X_1X_2+\frac{\partial^{2}{f}}{\partial{y}^{2}}X_{2}^{2}=0\right\}.$$ Now we would like to perturb $f$ so that the sets $ S $ and $ H $ have a transversal intersection. In fact:\\
Since the vector field $\phi$ is transverse to fiber (see Appendix \ref{sec Geo SM}), then  $X_1\neq 0$ or $X_2\neq 0$. Suppose that $X_1\neq 0$, then put $$f_{\delta}(x,y)=f(x,y)- \ds\frac{\delta x^{2}}{2}-cx$$ (later the constant $c$ will be chosen accordingly). Taking $\delta$ a small  regular value of 
$$L(x,y,z):=\frac{\frac{\partial^{2}{f}}{\partial{x}^{2}}X_{1}^{2}+2\frac{\partial^{2}{f}}{\partial{x}\partial y}X_1X_2+\frac{\partial^{2}{f}}{\partial{y}^{2}}X_{2}^{2}}{X_1^2}.$$
Then, the set $H_{\delta}:=\left\{L(x,y,z)=\delta\right\}$ is a regular surface. In addition, for any choice of the parameter $c$, we also have that  
$$H_{\delta}=\left\{(x,y,z):\frac{\partial^{2}{f_{\delta}}}{\partial{x}^{2}}X_{1}^{2}+2\frac{\partial^{2}{f_{\delta}}}{\partial{x}\partial y}X_1X_2+\frac{\partial^{2}{f_{\delta}}}{\partial{y}^{2}}X_{2}^{2}=0\right\}.$$

\noindent Consider now the function $$G(x,y,z):=\ds\frac{\left(\frac{\partial f}{\partial x}-\delta x\right)X_{1}+\frac{\partial f}{\partial y}X_{2}}{X_1}$$ and
and fix a very small regular value $c$ (close enough to
0) of the restriction of $G(x,y,z)$ to $H_{\delta}$, $G(x,y,z)|_{_{H_{\delta}}}$. Thus, the regular surface
$$S_{\delta}:=\left\{G(x,y,z)=c\right\}=\left\{(x,y,z):\frac{\partial f_{\delta}}{\partial x}X_{1}+\frac{\partial f_{\delta}}{\partial y}X_{2}=0\right\}$$
is transverse to 
 $H_{\delta}$  and its intersection, $J_{\delta}:=S_{\delta}\pitchfork H_{\delta}$, form a finite family  of $C^1$-curves. Let $\alpha_{_{\delta}}$ be the projections of the curves $J_{\delta}$ along  the flow on $\Sigma$. Then, by Lemma \ref{LIC},  there are sub-horseshoes $\Theta_{\delta}^{s}$ of $\Theta$ such that $\Theta^{s}_{\delta}\cap \alpha_{\delta}=\emptyset$ and
$$HD(K^{s}_{\delta})\geq HD(K^{s}_{\Theta})-\beta,$$
where $K_{\delta}^{s}$, $K^{s}_{\Theta}$ are the stable Cantor sets associated to $\Theta_{\delta}^{s}$ and $\Theta$.
Moreover, let $R_{f_{_{\delta}}}^{s}$ be a  Markov partition of $\Theta^{s}_{\delta}$ such that $R_{f_{_{\delta}}}^{s}\cap \alpha_{_{\delta}}=\emptyset$. Then, if $(x,v)\in R_{f_{_{\delta}}}^{s}$, then  the critical points of  $(f_{\delta} \circ \pi)(\phi^{t}(x,v))|_{\{t:\phi^{t}(x,v)\in U\}}$ are non-degenerates and therefore  finite. Thus, $f_{\delta}\in \mathcal{N}_{U, \Theta, \beta}^{s}$ as we wish.
$f_{\delta}\in \mathcal{N}_{U, \Theta, \beta}^{s}$ as we wish. 
\end{proof}
\begin{proof}[\emph{\textbf{Proof of Lemma \ref{Lmax1}}}] The set $\mathcal{N}_{\Theta, \beta}^{s}$ is clearly open. \\
Since $\ds \bigcup_{(x,v)\in \Theta}\, \, \bigcup_{t\in [0,t_{+}(x,v)]}\phi^{t}(x,v)$ is a compact set, then there are a finite number of
open set $U_1,\dots, U_n \subset SM$, which satisfies (i) and (ii) such that 
\begin{equation}\label{EQ1-Lmax2}
\bigcup_{(x,v)\in \Theta}\, \, \bigcup_{t\in [0,t_{+}(x,v)]}\phi^{t}(x,v)\subset \bigcup_{i=1}^{n}U_i.
\end{equation}

\noindent From Lemma \ref{Lmax2'}, for each $i$,  $\mathcal{N}^{s(u)}_{U_i,\Theta, \beta}\subset C^{2}({M,\re})$ is dense and $C^2$-open. \\
\textbf{Claim:} $\ds\bigcup_{i=1}^{n}\mathcal{N}^{s}_{U_i,\Theta, \beta}\subset \mathcal{N}_{\Theta, \beta}^{s}$.
\begin{proof}[\emph{\textbf{Proof of Claim}}]
Let $f\in \ds\bigcup_{i=1}^{n}\mathcal{N}^{s}_{U_i,\Theta, \beta}$, then for each $i$ there is a sub-horseshoe $\Theta_{f,i}^{s(u)}$ of $\Theta$ which satisfies Definition \ref{DLmax2}. Therefore, taking $\Theta_{f}^{s}:=\ds\bigcup_{i=1}^{n}\Theta_{f,i}^{s}$ we see that $f\in \mathcal{N}_{\Theta, \beta}^{s}$. 
\end{proof}

\noindent Since $\ds\bigcup_{i=1}^{n}\mathcal{N}^{s}_{U_i,\Theta, \beta}$ is $C^2$-dense, then the claim completes the proof. \\
The unstable case is analogous.
\end{proof}
\begin{R}\label{RLmax1'}
Once again we cover $\Lambda$ with a finite number of tubular neighborhoods $U_r$, $1\leq r\leq m$ whose boundaries are the good cross-sections $\Sigma=\bigsqcup\limits_{i=1}^k\Sigma_i$ mentioned in Section \emph{\ref{Dim. Red}}. For each $r$, let us fix coordinates $(x_1(w),x_2(w),x_3(w))$ on $U_r$ such that $x_3(w)$ is the flow direction and $U_r\cap\Sigma=\{x_3(w)=0\}\cup\{x_3(w)=1\}$.
 
\noindent  Consider $f\in \mathcal{N}_{\Theta,\beta}^{j}$, and $\Theta_{f}^{j}$ the sub-horseshoe given by \emph{Definition \ref{DLmax2}},  $j=s,u$.\newline 
Then, the proof of \emph{Lemma \ref{Lmax2'}} implies that there is a Markov partition $R_{f,\Theta}^{j}$ of \, $\Theta_{f}^{j}$, such  that the value $\max (f\circ \pi)_{\phi_{g_{_{0}}}}(z)$, for $z\in R_f^{j}$, is described by the values of a finite collection of graphs transverse to the flow direction.
In other words, for each $(x_1,x_2,0)\in U_r\cap R_f^{j}$ there is neighborhood \ $V$ of $(x_1,x_2,0)$ and a finite collection of disjoints graphs $\{(x,y, \psi_l(x,y)): (x,y,0)\in V\}$, $1\leq l\leq n$, such that  if $F(x,y,t)=\max (f\circ \pi)_{\phi}(x,y,t)$ with $(x,y,0)\in V$, then $t=\psi_l(x,y)$ for some $l$. 
\end{R}
\begin{R}\label{RLmax1}
A sub-horseshoe $\Theta$ of $\mathcal{R}$ has finite many fixed points. Thus, removing the fixed points of $\Theta$, we obtain a sub-horseshoe \emph{(}that we still call $\Theta$\emph{)} with almost the same Hausdorff dimension of $\Theta$. Hence, from now on, we assume that $\cR$ has no fixed points on $\Theta$.
\end{R}

\begin{Defi}\label{DLmax3} 
We say that  $f\in \mathcal{M}_{\Theta,\beta}^{s}\subset C^2(M,\re)$ if there is a sub-horseshoe $\tilde{\Theta}_f^{s}$ of $\Theta$ with the $\beta$-stable property and a Markov partition $R_{f, \tilde{\Theta}}^{s}$ of $\tilde{\Theta}_{f}^{s}$ so that for each $(x,v)\in \tilde{\Theta}_f^{s}$ we have that $\pi({\phi}^{t_0}(x,v))\neq \pi({\phi^{t_1}}(x,v))$, whenever 
\begin{equation*}
t_0, t_1\in \{t\in(0,t_{+}(x,v)): \ t \ \text{is a critical point of} \ \ f\circ\pi({\phi^{t}}(x,v))\}.
\end{equation*}

\noindent The definition for the unstable case $\mathcal{M}_{\Theta,\beta}^{u}$ is analogous.
\end{Defi} 


\begin{Le}\label{Lmax3'}
For any $\beta>0$ small, the set $\mathcal{M}_{\Theta,\beta}^{j}$ is dense and $C^2$-open, $j=s,u$.
\end{Le}
\begin{proof}
It is clear that $\mathcal{M}_{\Theta, \beta}^{s}$ is a $C^2$-open set. \\
Consider the set $\mathcal{M}$ of Morse's functions of $M$, which is a dense and $C^2$-open set, then we will prove that the set $\mathcal{M}\cap \ \mathcal{N}_{\Theta,\beta/2}^{s}\subset\mathcal{M}_{\Theta,\beta}^{s}$ and therefore, from Lemma \ref{Lmax1}, $\mathcal{M}_{\Theta,\beta}^{s}$ is a dense set. 
Let $f\in \mathcal{M}\cap \ \mathcal{N}_{\phi,\beta/2}^{s}$, then  
the set of critical points of $f$ in $\pi(\bigcup_{x\in \Sigma}\bigcup_{x\in [0,t_{+}(x)]}\phi^{t}(x))$ is finite, which we denoted by $x_{1}^{f},\dots,x_{k}^{f}$. 
Consider the finite family of $C^1$-curves $\alpha_{f}$ given by the projections of the fibers  $\pi^{-1}({x_{i}^{f}})$ on $\Sigma$ along of the flow. Thus, applying Lemma \ref{LIC} to the family of curves $\alpha_{f}$ and the horseshoe $\Theta^{s}_{f}$ (given by  Lemma \ref{Lmax1} with  $\beta/2$), we obtain 
two sub-horseshoes $_{*}\Theta^{s}_{f}$ of $\Theta_{f}^{s}$, such that $_{*}\Theta^{s}_{f}\cap \alpha_{f}=\emptyset$, and 
$$HD(_{*}K^{s}_{f, \Theta})\geq HD(K^{s}_{f,\Theta})-\beta/2 \ \ \text{and} \ \ HD(_{*}K^{u}_{f,\Theta})\geq HD(K^{u}_{f,\Theta})-\beta/2,$$
where $_{*}K_{f,\Theta}^{s}$ is the stable  Cantor sets associated to $_{*}\Theta_{f}^{s}$.\\
Therefore, as $\Theta_{f}^{s}$ has the $\frac{\beta}{2}$ stable property, then 
$$HD(_{*}K^{s}_{f,\Theta})\geq HD(K^{s}_{\Theta})-\beta \ \ \text{and} \ \ HD(_{*}K^{u}_{f,\Theta})\geq HD(K^{u}_{\Theta})-\beta.$$ 
Put $\tilde{\Theta}_{f}^{s}:=\,_{*}\Theta_{f}^{s}$ and consider a Markov partition $_{*}R_{f, \tilde{\Theta}}^{s}$ of $_{*}\Theta^{s}_{f}$ such that $_{*}R^{s}_{f,\tilde{\Theta}}\cap \alpha_{f}=\emptyset$. 
We stated that $f\in\mathcal{M}_{\Theta,\beta}^{s}$, in fact: by contradiction, assume that there is $(x,v) \in {_{*}}R_{f, \tilde{\Theta}}^{s}$ and $t_0, t_1\in (0,t_{+}(x,v))$  critical points of $g(t):=f\circ \pi(\phi^{t}(x,v))$ such that $\pi({\phi}^{t_0}(x,v))= \pi({\phi}^{t_1}(x,v))$,  \emph{i.e.}, $g'(t_0)=g'(t_1)=0$. Moreover, since $g'(t)=\left\langle \nabla f(\gamma_{v}(t)),\gamma'_{v}(t)\right\rangle$, where $\gamma_{v}(t)=\pi(\phi^{t}(x,v))$, the construction of $_{*}\Theta_{f}^{s}$ provides that $\nabla f(\gamma_{v}(t))\neq 0$ for all $t\in(0,t_{+}(x,v))$. Thus, as $\gamma_{v}(t_0)=\gamma_{v}(t_1)$ must be $\gamma'_{v}(t_0)=-\gamma'_{v}(t_1)$ or $\gamma_{v}'(t_0)=\gamma_{v}'(t_1)$. In the first case, we have a contradiction with the uniqueness of the geodesics with respect to the initial conditions. In the second case, we have that $\gamma_{v}(t)$ is a closed geodesic, and this is a contradiction, since  $t_0,t_1\in \left(0,t_{+}(x,v)\right)$. Thus, we complete the proof of the lemma. The unstable case is analogous. 
\end{proof}
\begin{Defi}\label{DLmax001} 
We say that $f\in \mathcal{C}_{\Theta, \beta}^{s}\subset C^2(M,\re)$ if there is a sub-horseshoe $\Theta_f^{s}$ of $\Theta$ with the $\beta$-stable property and there is a Markov partition $R_{f, \Theta}^{s}$ of $\Theta_{f}^{s}$, such that the function $\max  (f\circ \pi)_{\phi}|_{\Sigma\cap R_{f, \Theta}^{s}}\in  C^{1}(\Sigma\cap R_{f, \Theta}^{s},\re)$.
The definition for the unstable case $\mathcal{C}_{\Theta, \beta}^{u}$ is analogous.
\end{Defi}

\begin{Le}\label{Lmax0}
For any $\beta>0$ small, the set $\mathcal{C}_{\Theta,\beta}^{j}$ is dense and $C^2$-open, $j=s,u$.
\end{Le}
\begin{proof}
The set $\mathcal{C}_{\Theta,\beta}^{s}$ is clearly a $C^2$-open set. We will just prove the density for the stable case. The proof for the unstable case is analogous. 

By Lemma \ref{Lmax3'}, it is sufficient to prove that $\mathcal{C}_{\Theta,\beta}^{s}$ is dense $\mathcal{M}_{\Theta,\beta/2}^{s}$.
Let $f\in \mathcal{M}_{\Theta,\beta/2}^{s}$, then using coordinates $(x_1(w),x_2(w),x_3(w))$ given in Remark \ref{RLmax1'} our task is reduced to perturb $f$ in such a way that $\max (f\circ \pi)_{\phi}(x_1,x_2,t)$ is given by the values of $\max(f\circ \pi)_{\phi}$ on a unique graph $(x,y,\psi_{i}(x,y))$, $(x,y)\in V$,  $1\leq i \leq n$ (see Remark \ref{RLmax1'}).
Along this lines,  we employ the argument from \cite[Lemma  4.7]{R}. We begin noting that by Lemma \ref{Lmax3'}
\begin{equation}\label{E2Lmax3'}
\pi\{\text{graph} \, \psi_i\}\cap \pi\{\text{graph}\, \psi_j\}=\emptyset, \, \,  i\neq j. 
\end{equation}
Let $O_i, \tilde{O}_i$ be small neighborhoods of $\pi(\{z=\psi_{i}(x,y)\})$ such  that $O_i\subset \tilde{O}_i$ and $\tilde{O}_i\cap \tilde{O}_{j}=\emptyset$, $i\neq j$. Let $g_{1j}=f\circ \pi (x,y,\psi_{1}(x,y))-f\circ \pi (x,y,\psi_{j} (x,y))$ for $j\neq 1$. Consider $\gamma_1>0$ a small regular value of $g_{1j}$ for all $j\neq 1$. Take $\xi_1$ a $C^{\infty}$- bump function, equal to $-\gamma_1$ in $O_1$  and $0$ outside of $\tilde{O}_1$. So, the function $(f+\xi_1)$ is $C^2$-close to $f$. We define the function $$g_{1j}^{\gamma_1}(x,y)=((f+\xi_1)\circ \pi)(x,y,\psi_{1}(x,y))-((f+\xi_1)\circ \pi )(x,y,\psi_{j}(x,y))=g_{1j}(x,y)-\gamma_1.$$ 
Put $f_{{1}}:=f+\xi_{1}$ and define $g_{2j}(x,y)=(f_1\circ\pi)(x,y,\psi_2(x,y))-(f_{1}\circ \pi)(x,y,\psi_{j}(x,y))$ for $j\neq 2$. Consider  $\gamma_2>0$ a small regular value of $g_{2j}$ for all $j\neq2$. Take $\xi_2$ a $C^{\infty}$-bump function, equal to $-\gamma_2$ in $O_2$ and $0$ outside $\tilde{O}_2$.  So, the function $f_{1}+\xi_2$ is $C^2$-close to $f_1$ and define the function $$g_{2j}^{\gamma_2}(x,y)=((f_1+\xi_2)\circ\pi))(x,y,\psi_{2}(x,y))-((f_1+\xi_2)\circ\pi))(x,y,\psi_{j}(x,y))=g_{2j}(x,y)-\gamma_2.$$

The equation (\ref{E2Lmax3'}) implies that the functions $\xi_{1}\circ \pi$ and $\xi_{2}\circ \pi$ have disjoint support and therefore the perturbation of $f$, $f_1$, does not affect the perturbation $f_2$. Thus, as it is explained in \cite[Lemma 4.7]{R}, we can perform a small perturbation, $\tilde{f}$, of $f$ on $V$ in such a way that $0$ is a simultaneous regular value of the functions $(x,y)\mapsto g_{ij}(x,y):= \tilde{f}\circ \pi (x_1,x_2,\psi_i(x,y))-\tilde{f}\circ\pi  (x,y,\psi_j(x,y))$ for all choices of $1\leq i<j\leq n$. In this situation, $L_n=\bigcup\limits_{1\leq i<j\leq n} g_{ij}^{-1}(0)$ is a finite collection of $C^1$-curves such that, for each $y\in V\setminus L_n$, the values of $\max (\tilde{f}\circ \pi)_{\phi}$ near $y$ are described by the values of $\tilde{f}\circ \pi$ at a unique graph. Hence, for each $y\in V \setminus L_n$, one has that $\max (\tilde{f}\circ \pi)_{\phi}(y)=(\tilde{f}\circ \pi)(\phi^{t(y)}(y))$ for a unique $0\leq t(y)\leq t_+(y)$ that depends $C^1$ on $y$.

Since $f\in \mathcal{M}_{\Theta, \beta/2}^{s}$, then $\tilde{f}\in \mathcal{M}_{\Theta, \beta/2}^{s}$ and let $\Theta^s_{\tilde{f}}$ the sub-horseshoe of definition \ref{DLmax3} which satisfies  
$$HD(K^{s}_{\tilde{f},\Theta})> HD(K^{s}_{\Theta})-\beta/2,$$
where $K_{\tilde{f},\Theta}^{s}$ is the stable Cantor sets associated to $\Theta_{\tilde{f}}^{s}$.
Moreover, for $\beta/2$, the  Lemma \ref{LIC}  provides a sub-horseshoe $_{*}\Theta^{s}_{\tilde{f}}$ of $\Theta_{\tilde{f}}^{s}$ a Markov partition $_{*}R^{s}_{\tilde{f},\Theta}$ of $_{*}\Theta^{s}_{\tilde{f}}$ such that $_{*}R^{s}_{\tilde{f},\Theta}\cap L_n=\emptyset$ and 
$$HD(_{*}K^{s}_{\tilde{f},\Theta})> HD(K^{s}_{\tilde{f},\Theta})-\beta/2> HD(K^{s}_{\Theta})-\beta,$$
where $_{*}K_{\tilde{f},\Theta}^{s}$ is the stable (unstable) Cantor sets associated to $_{*}\Theta_{\tilde{f}}^{s}$.\\
The above discussion allows us to conclude that $\tilde{f}\in \mathcal{C}_{\Theta,\beta}^{s}$, as we wish. 
\end{proof}
\begin{R}\label{RLmax2}
For $j=s,u$, the construction of the sub-horseshoe $\Theta_{f}^{j}$ and the Markov Partition $R_{f, \Theta}^{j}$ at \emph{Lemma $\ref{Lmax0}$} satisfies  
\begin{itemize}
\item[\emph{(i)}] $\lbrace \pi(\phi^{t}(z)):z\in \Theta_{f}^{j}, \, 0\leq t\leq t_{+}(z)\rbrace \cap \{ \text{Critical points of} \  f \ \}=\emptyset$,
\item[\emph{(ii)}] For any $z\in R_{f, \Theta}^{s}$ we have that 
$$\#\{t\in [0,t^{+}(z)]: \max (f\circ \pi)(z)=f\circ\pi(\phi^t(z))\}=1,$$
\item[\emph{(iii)}] Note that, for any sub-horseshoe $\Theta$, the properties \emph{(i)} and \emph{(ii)} are persistently for small perturbations of $\Theta$. So, if $f_1,f_2\in \mathcal{C}_{\Theta,\beta}^{j}$ are close enough, then the sub-horseshoe $ \Theta_{f_1}^{j}$ can choose  close to  $\Theta_{f_2}^{j}$ in the Hausdorff distance for compact sets, $j=s,u$.
\end{itemize}

\end{R}

The rest of this section will be devoted to proving  Lemma \ref{Lmax4}. The set of functions $\mathcal{C}_{\Theta, \beta}^{j}$ provides the $C^1$ property for $\max(f\circ\pi)_{\phi}|\Sigma \cap R_{f, \Theta}^{j}$, $j=s,u$. To finally prove  Lemma \ref{Lmax4} we need to find the conditions of (\ref{E1-Descrip}).


 \subsubsection{The set of geodesics with transversal self-intersection}\label{sec GCSI}
Given any Riemannian metric $g$ on $M$, we define the set, $\mathcal{SI}_{g}$, of self-intersection of geodesics  by: 
   $$\mathcal{SI}_{g}=\left\{(x,v)\in S^{g}M: \exists \ \ t(x,v) \ \ \text{such that} \ \gamma_{v}(t(x,v))=x \ \ \text{and} \ 
\{v,\gamma_{v}'(t(x,v))\} \, \, 
 \text{are} \,\, \ \text{L.I.} \ \right\},$$
where L.I. stands for linearly independent and $\gamma_{v}(t)$ is the geodesic in the metric $g$ with $\gamma_{v}(0)=0$ and $\gamma'_{v}(0)=v$.\\

\begin{R} 
For any Riemannian metric $g$ of finite volume, the \emph{Liouville measure} is finite and  invariant by the geodesic flow, then  $\mathcal{SI}_{g}\neq \emptyset$.
\end{R}
\begin{Pro} \label{PauxLmax} The set $\mathcal{SI}_{g}$ is a cross-section for the geodesic flow of the metric $g$, $\phi_{g}^{t}$.
\end{Pro}


\begin{proof}
\noindent Let $(x_0,v_0) \in \mathcal{SI}_{g}$ and $t_0$ such that $\gamma_{v_0}(t_0)=x_0$ with $\{v_0,\gamma_{v_0}'(t_0)\}$. Consider $\mathcal{L}$ a transverse section to the flow $\phi_{g}^{t}$ and the fiber $\pi^{-1}_{g}(x_0)$ and define the following function 
\begin{eqnarray*}
h : &\mathcal{L}\times \re& \longrightarrow M\\
&((x,v),t)&\longmapsto \pi_{g}(\phi^{t}(x,v))
\end{eqnarray*}

\noindent with $h((x_0,v_0),0)=x_0$ and $h((x_0,v_0),t_0)=x_0$. Put $h_{0}:=h|_{\mathcal{L}\times I_0}$ and
$h_{t_0}:=h|_{\mathcal{L}\times I_{t_0}}$, where $I_0$ and $I_{t_{0}}$ are small intervals containing $0$ and $t_0$, respectively.

Let $\varphi\colon W_0\subset T_{x_0}M\to U_{x_0}$ be normal coordinates in $x_0$, where $U_{x_0}$ is a neighborhood of $x_0$, \emph{i.e.}, 
let $\{e_1,e_2\}$ be orthonormal basis of $T_{x_0}M$ and $\varphi(x_1,x_2)=\varphi(x_1e_1+x_2e_2)$\ $=\exp_{_{x_0}}^{g}(x_1e_1+x_2e_2)$, where $\exp_{_{x_0}}^{g}$ is the exponential map with the metric $g$. \\
\noindent We define 
$H\colon\mathcal{L}\times I_{t_0}\times I_{0}\longrightarrow V_{0}\subset T_{x_0}M$
\noindent by $$\ds H\left((x,v),t,s\right)=\left(\varphi^{-1}\circ h_{t_0}\right)((x,v),t)-\left(\varphi^{-1}\circ h_{0}\right)((x,v),s)$$ with  
 $H((x_0,v_0),t_0,0)=0$.

\noindent Moreover, 
\begin{eqnarray*}
\ds \frac{\partial H}{\partial t}{\left((x_0,v_0),t_0,0\right)}&=&(D\varphi^{-1})_{h_{t_{_{0}}}((x_0,v_0),t_0)}\left[\frac{\partial h_{t_0}}{\partial t}((x_0,v_0),t_0)\right]\\
&=&(D\exp_{x_0}^{-1})_{x_0}(\gamma_{v_0}'(t_0))=\gamma_{v_0}'(t_0), 
\end{eqnarray*}
since $D((\exp_{x_0}^{g})^{-1})_{x_0}=Id\colon T_{x_0}M\to T_{x_0}M$, the  identity of $T_{x_{0}}M$.
Also,
\begin{eqnarray*}
\ds \frac{\partial H}{\partial s}{\left((x_0,v_0),t_0,0\right)}&=&-(D\varphi^{-1})_{h_{_{0}}((x_0,v_0),0)}\left[\frac{\partial h_{0}}{\partial s}((x_0,v_0),0)\right]\\
&=&-D((\exp_{x_0}^{g})^{-1})_{x_0}(\gamma_{v_0}'(0))=-D((\exp_{x_0}^{g})^{-1})_{x_0}(v_0)\\
&=&-{v_0}.
\end{eqnarray*}
Since $\{-v_0,\gamma_{v_0}'(t_0)\}$ are linearly independent, then, $\frac{\partial H}{\partial (t,s)}$ is an isomorphism. Therefore, by the 
Implicit Function Theorem, there is neighborhood $U_{\mathcal{L}}\subset \mathcal{L}$ of $(x_0,v_0)$ and a diffeomorphism $\xi\colon U_{\mathcal{L}} \to V_{(t_0,0)}$, where $V_{(t_0,0)}$ is an open set containing $(t_0,0)\in \re \times \re$ and 
$H\left((y,w),\xi(y,w)\right)=0$.

\noindent Without loss of generality we can assume that $V_{(t_0,0)}=\widetilde{I}_{t_0}\times \widetilde{I}_{0}$ and
$$\xi(y,w)=(\xi_1(y,w),\xi_2(y,w)),$$
with $\xi_1$ close to $t_0$ and $\xi_2$ close to $0$. Then, by definition of $H$ we have that 
$$\exp_{x_0}^{-1}(\pi_{g}(\phi^{\xi_1(y,w)}(y,w)))=\exp_{x_0}^{-1}(\pi_{g}(\phi^{\xi_2(y,w)}(y,w))).$$ 
Therefore $\pi_{g}(\phi^{\xi_1(y,w)}(y,w))=\pi_{g}(\phi^{\xi_2(y,w)}(y,w))$. Equivalently, $\gamma_{w}(\xi_1(y,w))=\gamma_{w}(\xi_2(y,w))$, where $\pi_{g}(\phi^{t}(y,w))=\gamma_{w}(t)$ for any $(y,w)\in U_{\mathcal{L}}$.
Consider the new cross-section  
$$\widetilde{U}_{\mathcal{L}}=\left\{\phi^{\xi_2(y,w)}(y,w):(y,w)\in U_{\mathcal{L}}\right\}.$$
Note that $\xi_{1}(x_0,v_0)=t_0$ and $\xi_{2}(x_0,v_0)=0$, so $(x_0,v_0)\in \widetilde{U}_{\mathcal{L}}$.

\noindent  Let $(x,v)\in \widetilde{U}_{\mathcal{L}}$, then there is a unique $(y,w)\in U_{\mathcal{L}}$ such that $(x,v)=\phi^{\xi_2(y,w)}(y,w)$ and 
\begin{eqnarray*}
x=\pi_{g}(x,v)&=&\pi_{g}(\phi^{\xi_2(y,w)}(y,w))=\pi_{g}(\phi^{\xi_1(y,w)}(y,w))\\
&=&\pi_{g}\left(\phi^{\xi_1(y,w)-\xi_2(y,w)}\left(\phi^{\xi_2(y,w)}(y,w)\right)\right)\\
&=&\pi_{g}\left(\phi^{\eta(y,w)}(x,v)\right),
\end{eqnarray*}
where $\eta(y,w)=\xi_1(y,w)-\xi_2(y,w)$ is close to $t_0$. This implies that for any $(x,v)\in \widetilde{U}_{\mathcal{L}}$ there is 
$\eta(y,w)$ such that $\phi^{\eta(y,w)}(x,v)\in \pi^{-1}_{g}(x)$ and $\{v,\gamma'_{v}(\eta(y,w))\}$ are linearly independent and consequently $\widetilde{U}_{\mathcal{L}}\subset \mathcal{SI}_{g}$.

Thus,  the above discussion leads us to see that $\mathcal{SI}_{g}$ is a cross-section.
\end{proof}

\begin{R}\label{R1-Paux}
If $\mathcal{SI}_{g}^{n}:=\{(x,v)\in \mathcal{SI}_{g}: \ \left|t(x,v)\right|<n \}$.
Clearly, $\mathcal{SI}_{g}^{n}\subset\mathcal{SI}_{g}^{n+1}$. Moreover, given $(x,v)\in \mathcal{SI}_{g}^{n}$, from  \emph{Proposition \ref{PauxLmax}}, there is a
neighborhood $U$ of $(x,v)$ in $\mathcal{SI}_{g}$ such that $U\subset\mathcal{SI}_{g}^{n+1}$. Therefore, we can consider that $\mathcal{SI}_{g}^{n}$ as a submanifold of $S^{g}M$ of dimension $2$. Moreover,  since the transversality condition is open and dense, then, from now on, we assume that $\Sigma_{g}$ is  transverse to the  surface $\mathcal{SI}_{g}$ and $\mathcal{SI}_{g}^{n}$. 
\end{R}
\begin{R}\label{R2-Paux}
Under the conditions of the above remark, $\Sigma\pitchfork \mathcal{SI}_{g_{_{0}}}$ is a finite family of smooth curves $\alpha$. So, by \emph{Lemma \ref{LIC}}, given $\beta>0$ we have that there is a sub-horseshoe, $\Theta_{\alpha}^{s}$, of $\Theta$ such that $\Theta_{\alpha}^{j} \cap \Theta=\emptyset$ satisfying the $\beta$-stable property. Analogously, there is a sub-horseshoe $\Theta_{\alpha}^{u}$ for the unstable case. 
So, in order to avoid any confusion in the notation, we put $\Theta_1:=\Theta_{\alpha}^{s}$ and $\Theta_2:=\Theta_{\alpha}^{u}$.

\end{R}

Let $\beta>0$ be small enough, let $f\in \mathcal{C}_{\Theta_1,\beta}^{s}$, then there is a sub-horseshoe  $\Theta_{1,f}^{s}$  of $\Theta_1$ which satisfies the \text{Definition} \ref{DLmax001}. Analogously, if  $f\in \mathcal{C}_{\Theta_2,\beta}^{u}$ there is a sub-horseshoe $\Theta_{2,f}^{u}$  of $\Theta_2$ satisfying of Definition \ref{DLmax001}.

\begin{Defi}\label{DLmax5} We say that $f\in\mathcal{A}_{\Theta_1,\beta}^{s} \subset \mathcal{C}_{\Theta_1,\beta}^{s}$ if there is $z_f\in M_{\max (f\circ \pi)_{\phi}}(\Theta_{1,f}^{s})$ \emph{(}the set of maximum point of $\max(f\circ \pi)_{\phi}$  on  $\Theta_{1,f}^{s}$\emph{)} such that 
\begin{equation}\label{e1Lmax4}
\max (f\circ \pi)_{\phi}(z_f)>\max (f\circ \pi_g)_{\phi}(z), \, z \in \pi^{-1}(\pi(z_f))\cap \Theta_{1,f}^{s}.
\end{equation}
Changing $s$ by $u$ and $\Theta_1$ by $\Theta_2$, we have the definition of $\mathcal{A}_{\Theta_{2},\beta}^{u}$.
\end{Defi}
\begin{Le}\label{Lmax5}
The sets $\mathcal{A}_{\Theta_1,\beta}^{s}$ and $\mathcal{A}_{\Theta_2,\beta}^{u}$ are dense and $C^2$-open.
\end{Le}
\begin{proof}
We prove the lemma for $\mathcal{A}_{\Theta_1, \beta}^{s}$  since the proof for $\mathcal{A}_{\Theta_2, \beta}^{u}$ is analog. \\
By definition, the set $\mathcal{A}_{\Theta_1, \beta}^{s}$ is $C^2$-open. To prove the density, we have to prove simply that $\mathcal{A}_{\Theta_1, \beta}^{s}$ is dense in $\mathcal{C}_{\Theta_1, \beta}^{s}$. In fact: 
\noindent
Let $f\in\mathcal{C}_{\Theta_1,\beta}^{s}$, and $z_{f}=(x,v)$ a maximum point of $\max (f\circ \pi)_{\phi}$ in
$\Theta_1$. We can assume, without loss of generality, that  $\pi^{-1}(\pi(z_f)) \pitchfork \Sigma$ and therefore
$\#(\pi^{-1}(\pi(z_f))\cap \Sigma)<\infty$. If $f$
satisfies the condition of Definition \ref{DLmax5} the result is proven. Otherwise, we can suppose that there is 
$z=(x,w)\in\pi^{-1}(\pi(z_{f}))\cap \Theta_{1,f}^{s}$ such that 
$\max (f\circ \pi)_{\phi}(z_{f})=\max (f\circ \pi)_{\phi}(z)$, put $\tilde{z}_{f}:=\phi^{t(z_{f})}(z_f)$ and
$\tilde{z}:=\phi^{t(z)}(z)$ such that $(f\circ\pi)(\tilde{z}_{f})=(f\circ\pi)(\phi^{t(z_{f})}(z_f))=\max (f\circ \pi)_{\phi}(z_{f})=\max (f\circ \pi)_{\phi}(z)=f\circ\pi(\tilde{z})=f\circ\pi(\phi^{t(z)}(z))$.\\ 
\ \\
\textbf{Claim:} The points $\tilde{z}_{f}$ and $\tilde{z}$ satisfies 
\begin{enumerate}
	\item[1)] $\pi(\tilde{z}_{f})\neq\pi(\tilde{z})$,
	\item[2)] $\pi^{-1}(\pi(\tilde{z}_f))\cap \{\phi^{t}(z):t\in[t_{-}(z),t_{+}(z)]\}=\emptyset$. 
\end{enumerate}

\begin{proof}[\textbf{\emph{Proof of Claim}}]First we prove item 1)
By contradiction, assume that $\pi(\tilde{z}_{f})=\pi(\tilde{z})$, then as $df_{\pi(\tilde{z}_{f})}\neq 0$ (see  Remark \ref{RLmax2}) and $0=\frac{d}{dt}(f\circ\pi(\phi^{t}(z_f))|_{t=t(z_f)}=df_{\pi(\tilde{z}_{f})}\gamma_{v}'(t(z_f))$ and 
$0=\frac{d}{dt}(f\circ\pi(\phi^{t}(z))|_{t=t(z)}=df_{\pi(\tilde{z})}\gamma_{w}'(t(z))$, we have that $\gamma_{v}'(t(z_f))$ and $\gamma_{w}'(t(z))$ are parallel. Thus, due the uniqueness of geodesics, we must have $\gamma_{v}'(t(z_f))=-\gamma_{w}'(t(z))$. 
Note that $\Theta$ does not contain fixed points (see Remark \ref{RLmax1}),
then $v$ and $w$ are Linearly Independence and  therefore $z_f\in \mathcal{SI}_{g_{_{0}}}$, which is a contradiction by Remark \ref{R2-Paux} and the proof of item 1) is complete.\\
\indent Let us now prove item 2). Suppose by contradiction that $\phi^{t_{1}}(z)\in \pi^{-1}(\pi(\tilde{z}_f))$. Then 
$f\circ\pi(\phi^{t_1}(z))=f\circ\pi(\tilde{z}_{f})=f\circ\pi(\tilde{z})=f\circ\pi(\phi^{t(z)}(z))$. By  uniqueness of the maximum 
points of function $t \to f\circ\pi(\phi^{t}(z))$ (see Remark \ref{RLmax2}), we have that $t_1=t(z)$, \emph{i.e.},
$\tilde{z}=\phi^{t(z)}(z)=\phi^{t_1}(z)\in \pi^{-1}(\pi(\tilde{z}_{f}))$, which is not possible by item 1) and the proof of item 2) is complete .
\end{proof}


Following the proof of the lemma, note that the second property of claim provides  that for 
$z\in (\pi^{-1}(\pi(z_{f}))\setminus\{z_{f}\})\cap \Theta^{s}_{1,f}$ with $\max (f\circ \pi)_{\phi}(z_{f})=\max (f\circ \pi)_{\phi}(z)$, there are
neighborhoods $V_{z}$ of $z$ in $\Sigma$ and $U_{\pi(\tilde{z}_{f})}$ of $\pi(\tilde{z}_{f})$  such that 
\begin{equation}\label{ELmax3'}
\left( \bigcup_{y\in V_z}\{\phi^{t}_{g}(y):t\in [t_{-}(y),t_{+}(y)]\}\right) \cap \pi^{-1}(U_{\pi(\tilde{z}_{f})})=\emptyset.
\end{equation}
Let $R^{s}_{1,f}$ be Markov partition from $\Theta^{s}_{1,f}$ of Definition \ref{DLmax001}. Given $\epsilon>0$ sufficiently small, let $0\leq \varphi\leq \epsilon$  a $C^{\infty}$-bump function with a unique maximum point equal to $\epsilon$ in $\pi(\tilde{z}_{f})$ and $0$ outside of $U_{\pi(\tilde{z}_{f})}$.  Then, $z_f$ is a maximum point of $\max \,(f+\varphi)_{\phi}|_{\Sigma\cap {R_{1,f}^{s}}}$. \\
\noindent Moreover, if $z\in (\pi^{-1}(\pi(z_{f}))\setminus\{z_f\})\cap\Theta_{1,f}^{s}$, then by equation (\ref{ELmax3'}), we have that $\phi^{t}(z)\notin \pi^{-1}(U_{\pi(\tilde{z}_{f})})$, thus 
\begin{eqnarray*}
(f+\varphi)\circ\pi(\phi^{t}(z))&=&f(\pi(\phi^{t}(z)))+\varphi(\pi(\phi^{t}(z)))\\
&<& f(\pi({\tilde{z}_{f}}))+\epsilon=f(\pi({\tilde{z}_{f}}))+\varphi(\pi(\tilde{z}_{f}))=\max(f+\varphi)_{\phi}(z_{f}),
\end{eqnarray*}
\emph{i.e.}, $\max(f+\varphi)_{\phi}(z)<\max(f+\varphi)_{\phi}(z_f)$.
 
The previous property together with the fact that $\pi^{-1}(\pi(z_f))\pitchfork \Sigma$ implies that there is a neighborhood 
$U_{z_{_{f}}}$ of $z_f$ in $\Sigma$ such that $$\max(f+\varphi)_{\phi}(\tilde{z})<\max(f+\varphi)_{\phi}(x)$$
$\tilde{z}\in (\pi^{-1}(\pi(x))\setminus\{x\})\cap \Theta_{1,f}^{s} \ \text{and} \ x\in U_{z_f}.$
Finally, note that for $\epsilon$ small enough, $f+\varphi \in \mathcal{C}_{\Theta_1, \beta}^{s}$, and $\Theta_{1,f+\varphi}^{s}$ is close to $\Theta_{1,f}^{s}$ (see Remark \ref{RLmax2}). Therefore, the last inequality holds for $f+\varphi$ and $\Theta^{s}_{1,f+\varphi}$, and then $f+\varphi\in \mathcal{A}_{\Theta_1,\beta}^s$, as we wished. 
\end{proof}

\begin{C}\label{CLmax3'}
With the notation of \emph{Lemma \ref{Lmax5}}, if $f\in \mathcal{A}_{\Theta_1, \beta}^{s}$ there is a neighborhood \ $U_{z_{_{f}}}\subset \Sigma$ of $z_f$ such that 
$$\max (f\circ \pi)_{\phi}(x)>\max (f\circ \pi)_{\phi}(\tilde{x}) $$
for all $\ x\in U_{z_f}$ {and} $\tilde{x}\in \left(\pi^{-1}(\pi(x))\setminus\{x\}\right)\cap \Theta_{1,f}^{s}.$
\end{C}
\begin{proof} 
Without loss of generality, we can assume that $\pi^{-1}(\pi(z_f))\pitchfork \Sigma$ and  there is a neighborhood $V$ of $z_f$ such that 
$$\pi^{-1}(\pi(z))\pitchfork \Sigma, \, \, \text{for all}\, \, z\in V.$$

By contradiction, suppose that there are $z_n\to z_f$ and 
$\tilde{z}_{n}\in \pi^{-1}(\pi(z_n))\cap \Theta_{1,f}^{s}$ such that $\max (f\circ \pi)_{\phi}(z_n)\leq \max (f\circ \pi)_{\phi}(\tilde{z}_n)$. By compactness and transversality, we assume that $\tilde{z}_{n}\to z\in \pi^{-1}(\pi(z_f))\setminus\{z_f\}$. Therefore, since $z_f$ is a maximum point of $\max (f\circ \pi)_{\phi}$, then $\max (f\circ \pi)_{\phi}(z_f)=\max (f\circ \pi)_{\phi}(z)$, so $f\notin \mathcal{A}_{\Theta_1,\beta}^{s}$ and we have a contradiction. 
\end{proof}
\begin{proof}[\emph{\textbf{Proof of Lemma \ref{Lmax4}}}] We simply prove the lemma for $\widetilde{\mathcal{H}}_{\Theta,\beta}^{s}$ since the proof for $\widetilde{\mathcal{H}}_{\Theta,\beta}^{u}$ is analog. 
The openness $\widetilde{\mathcal{H}}_{\Theta,\beta}^{s}$  is a consequence of its  definition.
By density of Lemma \ref{Lmax5}, it suffices to prove that $\widetilde{\mathcal{H}}_{\Theta,\beta}^{s}$ is dense in $\mathcal{A}_{\Theta_1, \beta/2}^{s}$.\\ By definition, if  $f\in \mathcal{A}_{\Theta_1, \beta/2}^{s}$, there is $z_{f}=(x,v)\in \Theta_{1,f}^{s}$ a maximum point of $\max (f\circ \pi)_{\phi}$. Then by Corollary \ref{CLmax3'} there is a neighborhood \ $U_{z_f}$ of $z_f$ such that 
\begin{equation}\label{ELmax3''}
\max (f\circ \pi)_{\phi}(x)>\max (f\circ \pi)_{\phi}(\tilde{x}) \ 
\end{equation}
for all $ x\in U_{z_f} \ \text{and} \ \tilde{x}\in \left(\pi^{-1}_{g}(\pi_{g}(x))\setminus\{x\}\right)\cap \Theta_{1,f}^{s}$.
It is easy to see that there is  $h\in C^{2}(\Sigma\cap R_{1,f}^{s},\re)$, $C^2$-close to the null function, $h=0$ outside of $U_{z_{f}}$, and such that the function  $\max (f\circ \pi_g)_{\phi_g}+h\in{H}_{1}(\cR,\Theta_{1,f}^{s})$ with the maximum point $\tilde{z}_{f}$ in $U_{z_{f}}$, \emph{i.e.}, $\tilde{z}_{f} \in M_{\max (f\circ \pi)_{\phi}+h}(\Theta_{1,f}^{s})\cap U_{z_{f}}$ (the set of maximum points of $f$ in $\Theta_{1,f}^{s}\cap U_{z_{f}}$).  Without loss of generality, we can assume that 
$\pi|_{U_{z_{f}}}:U_{z_{f}}\to \pi(U_{z_{f}})$ is a diffeomorphism. Put $\tilde{h}:=h\circ (\pi|_{U_{z_{f}}})^{-1}\colon \pi(U_{z_{f}})\to \re$ and put $j$ equal to $0$ outside of the neighborhood of $\pi(U_{z_{f}})$. 
\ \\
\textbf{Claim:} The function $\max (f\circ \pi)_{\phi}+\tilde{h}\circ \pi\in{H}_{1}(\cR,\Theta_{1,f}^{s})$. 
\begin{proof}[\emph{\textbf{Proof of Claim}}]
We have to prove simply that $\tilde{z}_{f}$ is the only maximum point of \break $\max (f\circ \pi)_{\phi}+\tilde{h}\circ \pi$ in $\Theta_{1,f}^{s}$. Since 
$\tilde{h}\circ \pi=h$ on $U_{z_{f}}$, then $\tilde{z}_{f} \in M_{\max (f\circ \pi)_{\phi}+\tilde{h}\circ \pi}(U_{z_{f}})\cap U_{z_f}$. Now, if 
$z\in (\pi^{-1}(\pi(U_{z_{f}}))\setminus U_{z_{f}})\cap \Theta_{1,f}^{s}$, then there is $x\in U_{z_{f}}$ such that
$z\in \pi^{-1}(\pi(x))$. Thus $\tilde{h}\circ\pi(z)=\tilde{h}\circ\pi(x)$. Moreover,  by inequality (\ref{ELmax3''}) we have 
$\max (f\circ \pi)_{\phi}(x)>\max (f\circ \pi)_{\phi}(z)$. This implies that 
$$\max (f\circ \pi)_{\phi}(\tilde{z}_{f})+\tilde{h}\circ\pi(\tilde{z}_{f})>\max (f\circ \pi)_{\phi}(x)+\tilde{h}\circ \pi(x)>\max (f\circ \pi)_{\phi}(z)+\tilde{h}\circ \pi(z).$$
Again, since $\tilde{h}\circ\pi=0$ on $\Sigma\setminus \pi^{-1}(\pi(U_{z_{f}}))$, then the last inequality  holds for \\
$z\in\Theta_{1,f}^{s}\setminus \pi^{-1}(\pi(U_{z_{f}}))$, and therefore $\tilde{z}_{f}$ is the unique maximum point of $\max (f\circ \pi)_{\phi}+\tilde{h}\circ \pi$ on $\Theta_{1,f}^{s}$ and consequently, this proves the claim. 
\end{proof}
\noindent To finish the proof of lemma, note that the function
\[
f_1 = 
  \begin{cases}
      f+\tilde{h}, & \, \text{on} \,\,\pi(U_{z_f}), \\
      f & \text{otherwise}
  \end{cases}
\]
is $C^2$-close to $f$ with 
$$\max (f_1\circ \pi)_{\phi}=\max (f\circ \pi)_{\phi}+\tilde{h}\circ \pi\in H_{1}(\cR,\Theta_{1,f}^{s}).$$
Moreover,  as $\mathcal{A}_{\Theta_1, \beta/2}^{s}\subset\mathcal{C}_{\Theta_1,\beta/2}^{s}$, then by definition of $\mathcal{C}_{\Theta_1,\beta/2}^{s}$ (see Definition \ref{DLmax001}) and Remark \ref{R2-Paux} we have that $f_1 \in\widetilde{\mathcal{H}}_{\Theta,\beta}^{s}$ with the sub-horseshoe $\Theta_{1,f}^{s}$ and the Markov partition $R_{1,f}^{s}$, which concludes the  proof. 

\end{proof}
\begin{R}\label{Remark-principal-1} Note that the proof of \emph{Lemma \ref{Lmax4}}, also allows us to conclude that: If $\Theta$ and $\Gamma$ are two basic sets of $\cR$, then the set $\widetilde{\mathcal{H}}_{\Theta, \Gamma, \beta}\subset \widetilde{\mathcal{H}}_{\Theta,\beta}^{s}\cap \widetilde{\mathcal{H}}_{\Gamma,\beta}^{u}$, which satisfies
$$\max  (f\circ \pi)_{\phi}|_{\Sigma\cap (R_{f,\Theta}^{s}\cup R_{f,\Theta}^{u})}\in H_1(\cR,(\Theta_{f}^{s}\cup \Gamma_{f}^{u}))\subset C^{1}(\Sigma\cap (R_{f,\Theta}^{s}\cup R_{f,\Gamma}^{u}),\re)$$ 
is dense and $C^2$-open.\\
Here $R_{f,\Theta}^{s}$ and $R_{f,\Gamma}^{u}$ are the Markov partition of $\Theta_{f}^{s}$ and $\Gamma_{f}^{u}$ given by the definition of $\widetilde{\mathcal{H}}_{\Theta,\beta}^{s}$ and $ \widetilde{\mathcal{H}}_{\Gamma,\beta}^{u}$, respectively. \\
\indent Analogously, since \emph{Remark \ref{RDLmax0001}} $($compare with \emph{Remark \ref{Remark-principal-0}}$)$, it is easy to see that, if $\mathcal{V}$ is small enough, then  for any $g\in \mathcal{V}$  there is $\tilde{\beta}$ \emph{(}which depends on $g$\emph{)} such that 
\begin{equation}\label{Eq-aux}
\widetilde{\mathcal{H}}_{\Theta_g, \Gamma, \tilde{\beta}}=\widetilde{\mathcal{H}}_{\Theta, \Gamma, \beta},
\end{equation}
and consequently $\widetilde{\mathcal{H}}_{\Theta_g, \Gamma, \tilde{\beta}}$ is a dense and $C^2$-open set, where $\Theta_g$  is the hyperbolic continuation of $\Theta$  by $\cR^{g}$, respectively $($see  \emph{Section \ref{Dim. Red}}$)$.
\end{R}

\section{The Family of Perturbations of Metric}\label{Family of Pert}\label{Family}

\noindent Recall that in  Section \ref{HD>1} it was proven that there are a finite number of smooth-GCS, $\Sigma_{i}$ pairwise disjoint and  such that
the Poincar\'e map $\mathcal{R}$ (map of first return) of $\Sigma:=\bigcup^{l}_{i=1}\Sigma_{i}$ $$\cR \colon \Sigma\to \Sigma$$
satisfies:
\begin{itemize}
 \item $\bigcap_{n\in \mathbb{Z}}\mathcal{R}^{-n}(\Sigma):=\Delta$ is hyperbolic set for $\cR$.
\item $\ds HD\left(\Delta\right)\sim 2$.
\end{itemize}
\subsection{Nonzero Birkhoff Invariant and Property V}
 The family of perturbations of our metric will be special and involve some properties of horseshoes, as well of  regular Cantor sets. Since the geodesic flow is conservative, then Poincar\'e map $\cR$ is a conservative diffeomorphism (see \cite[Lemma 5.1]{R}), then to use some techniques at \cite{MY} and \cite{MR}, we need an appropriated family of perturbations of $\cR$. Thus, we perform a first conservative perturbation of $\cR$, such that the remain Poincar\'e map has  non-zero  Birkhoff invariant  in some periodic point  (cf. Appendix \ref{BI} and \cite[Section 4.3]{MY1} for more details). 

\begin{Le}\label{Le BI}
There is a small perturbation of the metric $g_0$ such that the Birkhoff invariant, for the new Poincar\' e map $($still denoted by $\cR$\emph{)}, is non-zero for some periodic orbit.
\end{Le}
\begin{proof}
We can assume that $\cR$ has a fixed point such that the Birkhoff invariant is zero. Then Lemma \ref{Birkhoff Invariant} implies that
the set $Q \subset J^{3}(0)$ of mappings having  non-zero Birkhoff invariant is open, dense, and invariant. Thus, the Klingenberg and Takens
theorem (cf. Appendix \ref{BI}) allows us to conclude the proof.
\end{proof}
\begin{R}\label{RBI}
As the perturbation of $g_0$ of \emph{Lemma \ref{Le BI}} is small enough, we can assume such perturbation is an element of $\mathcal{V}\subset$ {\boldmath$\mathcal{G}$}$^{3}(M)$. Moreover, from \emph{Lemma \ref{Le BI}} we can assume that, from now on,  the Poincar\'e map $\cR$  has the 
property that the Birkhoff invariant is non-zero for some periodic orbit. 
\end{R}
\subsubsection{Property V}\label{PropV}
In this section, we will try to expose a second and important property that we need to achieve a good perturbation of $\cR$. \ \\
\begin{Defi}
Let $K$, $K'$ be two regular Cantor sets, we say that $K$ and $K'$ have a stable intersection if there is a neighborhood $U$ of $(K,K')$ in the set of pairs in the $C^{1+ }$-regular Cantor sets such that 
$(\tilde{K}, \tilde{K'})\in V \Rightarrow \tilde{K}\cap \tilde{K'}\neq \emptyset$.
\end{Defi}
For more details on the definition of $C^{1+}$-topology of pairs of regular Cantor sets, see \cite[Section 2]{MY} or \cite[Chapter 2]{PT}.\\

\noindent Given a horseshoe $\Gamma$ associated to a $C^2$ bi-dimensional diffeomorphism $\psi$ , we consider $K^s$ and $K^u$ the stable and unstable Cantor sets associated to $\Gamma$ (see Section \ref{sec EMAH}).

\begin{Defi}
We say that a pair $(\psi, \Gamma)$, where $\Gamma$ is a horseshoe for $\psi$, has the property V if $K^s$ and $K^u$ have a stable intersection. 
\end{Defi}

\noindent For more details on property V, we recommend in \cite[Sections 1 and 2 ]{MY}.\\

When a horseshoe has Hausdorff dimension bigger than 1, then we get the property V, in fact:\begin{T}[\text{\cite[Moreira-Yoccoz]{MY1}}]Let $\varphi$ be a $C^\infty$ diffeomorphism on a surface $M$ with a horseshoe $\Gamma$ of $HD(\Gamma)>1$. If $\mathcal{U}$ is a sufficiently small neighborhood of $\varphi$ in $Diff^{\infty}(M)$,
there is an open and dense set $\mathcal{U}^* \subset \mathcal{U}$ such that, for every $\psi\in \mathcal{U}^*$ the pair $(\psi, \Gamma_\psi)$
has the property V, where $\Gamma_\psi$ is the hyperbolic perturbation of $\Gamma$ by $\psi$.
\end{T}
\ \\
\indent We recall  that $\Delta$ is a horseshoe for $\cR$ with $HD(\Delta)>1$, so the above theorem implies that we get the property V with small perturbations of $\cR$, but note that perturbations of $\cR$ non-necessarily are Poincar\'e maps of small perturbation of the metric $g_0$, in other words, we can not use directly the above theorem. So, to get the property V on the pair $(\cR^g, \Delta_g)$, for a small perturbation of $g_0$, we need to use the  construction of the proof of the above theorem. Thus, in the next sections, we will explain such a construction.
\subsection{Good Perturbations}\label{sec ISGF}
In this section, we will prove  Theorems \ref{Theorem 1} and \ref{Theorem 2}. 
The main difficulty for this goal  is to produce perturbations of  the Riemannian metric so that, the perturbation obtained for the Poincar\'e map $\cR$ has  independence, \emph{i.e.}, allows to obtain property $V$  (see Subsection \ref{PropV}).  
It is worth pointing out that, small perturbations of the Riemannian metric produce great perturbations on the geodesic flow. 
For example, if we want to perturb $\cR$ to obtain property  $V$ and still be an application of the first return of the geodesic flow for a Riemannian metric near the initial Riemannian metric, we must keep in mind that perturbing a metric in a small neighborhood $U\subset M$  always affect the Sasaki's metric on $\pi^{-1}(U)\subset SM$, where  $\pi\colon SM \to M$ is the canonical projection. In other words, local perturbations of the metric, in general, do not produce local perturbations of $\cR$.  Thus, the more complicated part of this section is to produce local perturbations of $(\cR, \Delta)$ with the property V perturbing the metric. 
\subsubsection{Independent Perturbations of the Metric}\label{sec PM}

In the sequel, we present the necessary tools to obtain the independence of the perturbations of the metric.  We start with the following lemma.

\begin{Le}\label{Lemma 2 ad}
For any fixed $\theta\in \Delta$, put $\tilde{\theta}=\phi^{\frac{t_{+}(\theta)}{2}}(\theta)$, then the surface $$\mathcal{W}_{\tilde{\theta}}:=\bigcup_{t\in (-t_{+}(\theta),t_{+}(\theta))}\phi^{t}(\pi^{-1}(\pi(\tilde{\theta})))$$ intersects transversely $\Sigma$, therefore $\mathcal{W}_{\tilde{\theta}}\cap \Sigma=\{\beta_1,\dots, \beta_l\}$ is a finite family of curves. Moreover, if $z\in\beta_{i}\cap \Delta\cap \Sigma$ for some $i\in \{1,\dots,l\}$, then 
$$\beta_i \pitchfork W^{s}(z,\Sigma) \ \ \text{and} \ \ \beta_i \pitchfork W^{u}(z,\Sigma).$$
\end{Le}
\begin{proof}
For $(p,v)\in \pi^{-1}(\pi(\tilde{\theta}))$ and a non-zero vector $\xi\in T_{(p,v)}\pi^{-1}(\pi(\tilde{\theta}))$ (the tangent space of 
$\pi^{-1}(\pi(\tilde{\theta}))$ in $(p,v)$), then the tangent space to $\mathcal{W}_{\tilde{\theta}}$ in $\phi^{t}(\xi)$ is 
$$\ds T_{\phi^{t}((p,v))}\mathcal{W}_{\tilde{\theta}}=\text{span}\{d\phi^{t}_{(p,v)}(\xi), d\phi^{t}_{(p,v)}(\phi((p,v)))\}=\text{span}\{d\phi^{t}_{(p,v)}(\xi), \phi(\phi^{t}((p,v)))\}.$$
Therefore, since $\phi\pitchfork \Sigma$, then $\mathcal{W}_{\tilde{\theta}}\pitchfork\ \Sigma: =\{\beta_1,\dots, \beta_l\}$ a finite set of $C^1$-curves.

Let $z\in\beta_{i}\cap \Delta\cap \Sigma$ for some $i\in\{1, \dots, l\}$, then there is $\omega\in \pi^{-1}(\pi(\tilde{\theta}))$ such that $z=\phi^{t}(\omega)$, so
$$T_{z}\beta_{i}=T_{z}\mathcal{W}_{\tilde{\theta}}\pitchfork T_{z}\Sigma=\text{span}\{d\phi^{t}_{\omega}(\eta), \phi(\phi^{t}(\omega))\}\pitchfork T_{z}\Sigma,$$
where $\eta\in T_{\omega}\pi^{-1}(\pi(\tilde{\theta}))$ is a non-zero vector. \\
Remember that $T_{z}W^{i}(z,\Sigma)=(E^{i}(z)\oplus\text{span}(\phi(z)))\cap T_{z}\Sigma$ for $i=s,u$, then since $E^{i}$ is an invariant sub-bundle
for $d\phi^{t}$, $i=s,u$, it follows that $T_{z}W^{i}(z,\Sigma)=(d\phi^{t}_{\omega}(E^{i}(\omega))\oplus \text{span}\{\phi(\phi^{t}(\omega))\})\pitchfork T_{z}\Sigma$. Therefore, as $\eta$ is a vertical vector, Lemma \ref{Lemma 1 ad} implies that  $d\phi^{t}_{\omega}(\eta)\notin (d\phi^{t}_{\omega}(E^{i}(\omega))\oplus \text{span}\{\phi(\phi^{t}(\omega))\}$, $i=s,u$, which completes the proof of the lemma. 
\end{proof}
\begin{Defi}
Let $\{\theta, \tilde{\theta}\}$ two words and $\mathfrak{R}=\{R_1, R_2,\dots, R_k\} $ a set of words of the same alphabet $\mathcal{B}$. We say that the word $\theta\tilde{\theta}$ is \emph{prohibited} by $\mathfrak{R}$ if there is $R\in \mathfrak{R}$ inside $($as a factor of$)$ $\theta\tilde{\theta}$.
$$\overbrace{-----}^{\alpha}\overbrace{-----------}^{R}\overbrace{-----}^{\beta}$$
$$\underbrace{----------}_{\theta} \underbrace{-----------}_{\tilde{\theta}}.$$
\end{Defi}

\begin{Defi}\label{def no dist}
Let $\Theta_1$, $\Theta_2$ be two disjoint sub-horseshoes of $\Delta$ and a Markov Partition $\mathfrak{R}_i$ of $\Theta_i$, $i=1,2$. We say that $R_1\in \mathfrak{R}_1$ disturbs $R_2\in \mathfrak{R}_2$ if 
$$T_{R_2}\cap \tau_{R^{1/2}_1}\neq\emptyset,$$
where $T_{R_2}=\{\phi^{t}(x):x\in R_2,\ \, 0\leq t \leq t_{+}(x)\}$, $\tau_{R^{1/2}_1}=\pi^{-1}(\pi(R^{1/2}_1))$ and
	$R^{1/2}_1=\phi^{{\bar{t}_1}/2}(R_1)$ with $\bar{t}_1=\ds \sup_{x\in R_1\cap \Delta}t_{+}(x)$. 
We say that $\Theta_1$ has no interference on $\Theta_2$ if no element of $\mathfrak{R}_1$ disturbs any element of $\mathfrak{R}_2$.
\end{Defi}

The main goal of this section is to prove the following lemma, which will give us the region where we can perturb the metric independently.

\begin{Le}\label{main-lemma}
There are two disjoint sub-horseshoe $\Delta_2$ and $\Delta_3$ of $\Delta$ such that $\Delta_2$ has {no interference} on $\Delta_3$ and $HD(K_3^s)+HD(K_2^u)>1$, where $K_3^s$ is the stable Cantor set of $\Delta_3$  and $K_2^u$ is the unstable Cantor set of $\Delta_2$.
\end{Le}

The remainder of this section is dedicated to proving Lemma \ref{main-lemma}. For this sake, we  introduce some auxiliary lemmas which will
help reach this goal. 
\begin{Le}\label{ind. of horseshoe}
Let $\Delta_{1}\subset \Delta$ a sub-horseshoe with $0<HD(\Delta_1)=:\lambda<\frac{1}{2}$, then there exists another sub-horseshoe $\Delta_2\subset\Delta$ with the following properties:

\begin{enumerate}
    \item[\emph{(1)}]  $HD(K_2^{u})$ is sufficiently  close $HD(K^{u})$, where $K_{2}^{u}$ and $K^{u}$ are the unstable Cantor sets associated with $\Delta_{2}$  and $\Delta$, respectively; 
	\item[\emph{(2)}]  $\Delta_{1}\cap\Delta_2=\emptyset$;

	\item[\emph{(3)}]  $\Delta_2$ has no interference on $\Delta_1$. 

\end{enumerate}
\end{Le}
\begin{proof}
Consider a Markov partition $\mathfrak{R}$ of $\Delta$ into squares of size $\epsilon$. Then $\mathfrak{R}$ has approximately
$\epsilon^{-d}$\ squares, where $d=HD(\Delta)$. Consider, analogously, a set $\widetilde{\mathfrak{R}}_{\Delta_1}\subset \mathfrak{R}$ of the order of $\epsilon^{-\lambda}$ squares of size $\epsilon$ forming a Markov partition of $\Delta_1$. 
Observe also that, given $p\in \Delta_1$ belonging to a cross-section $\Sigma_i$, the projection by the flow $\phi^t$ of the fiber $\pi^{-1}(\pi(\phi^{t_{+}(p)/2}(p)))$ is a curve (or a finite union of curves), and so, as in the proof of Lemma \ref{LIC} (\cite[Lemma 4.6]{R}), each square of $\widetilde{\mathfrak{R}}_{\Delta_1}$ has interference on at most of the order of $\epsilon^{-d/2}$ (which is much smaller than  $\epsilon^{-1}$) squares of $\mathfrak{R}$. Thus, the squares of $\widetilde{\mathfrak{R}}_{\Delta_1}$ have interference on at most $\epsilon^{-\lambda}\cdot\epsilon^{-1}\leq\epsilon^{-3/2}$ squares of $\mathfrak{R}$. We call $\mathfrak{X}\supset \widetilde{\mathfrak{R}}_{\Delta_1}$ the set of squares that suffer interference of some square of $\widetilde{\mathfrak{R}}_{\Delta_1}$. 
Therefore, we have $\tilde N:=\#\{\mathfrak{R}\setminus \mathfrak{X}\} \geq \epsilon^{-d}-\epsilon^{-3/2}$ squares of $\mathfrak{R}$ which do not suffer interference of any square of $\widetilde{\mathfrak{R}}_{\Delta_1}$.
\ \\ 
\indent The maximal invariant set by $\cR$ of the union of these $\tilde N$ remaining squares in $\mathfrak{R}\setminus \mathfrak{X}$,
which will be the sub-horseshoe $\Delta_2'\subset \Delta$. By the above construction, this sub-horseshoe $\Delta_2'$ satisfies
conditions $2$ and $3$ of the lemma. In what follows, we will estimate the size of $\Delta_2'$.\\
\indent Let $\{\tilde{\theta}_1,\dots,\tilde{\theta}_{\tilde N}\}$ be the words associated with the remaining squares. They
generate intervals of the length of the order of $\epsilon^2$ in $W^{u}(\Delta)$ (of the construction of the unstable  regular Cantor set).
Without loss of generality, we can assume, {as in \cite[Remark 13]{R}}, that the transitions $\tilde{\theta}_i\tilde{\theta}_j$ are
admissible for all $i,j\in \{1,\dots,{\tilde N}\}$.\\
\ \\
\textbf{Claim:} $\#\{\tilde{\theta}_i\tilde{\theta}_j \, \, \, \text{prohibited by}\, \,  \mathfrak{X}\}=O(\tilde{N}^2)$.
\begin{proof}[\emph{\textbf{Proof of Claim}}]

\noindent Since the product of the lengths of the intervals in $W^{u}(\Delta)$ generated by the words $\alpha$ and $\beta$ is of the order
of $\epsilon^4/\epsilon^2=\epsilon^2$, the number of possibilities for the word $\#\{\alpha\beta\}$ is of the order of $\tilde N$. On the other hand, the size of each word  $\tilde{\theta}_i$ (which gives an upper bound for the number of positions where the word $R$ begins) is of the order of $\log\epsilon^{-d}$, which is of the order of $\log {\tilde N}$. Then
each word $R$ corresponding to a square in $\mathfrak{X}$ prohibits $O(\tilde N\log {\tilde N})$ transitions. So we have in total 
$O(\epsilon^{-3/2}\tilde N\log {\tilde N})$ prohibited transitions $\tilde{\theta}_i\tilde{\theta}_j$. Since $d>3/2$ and $\tilde N$ is of
the order of $\epsilon^{-d}$, we have $\epsilon^{-3/2}\tilde N\log {\tilde N}=o({\tilde N}^2)$.
\end{proof}
This shows that the number of prohibited transitions 
is much smaller than the total number of transitions. So, consider the following matrix $A$ for $i,j\in\{1,\dots,\tilde N\}$
\[ a_{ij} = \left\{ \begin{array}{lll}
         1 & \mbox{if $\tilde{\theta}_{i}\tilde{\theta}_{j} \ \text{is not prohibited}$};\\
          & \\
       0 & \mbox{if $\tilde{\theta}_i\tilde{\theta}_j \ \text{is prohibited by} \ 
        \mathfrak{X}. $}\end{array} \right. \] 
The above claim states that $\#\{a_{ij}:a_{ij}=1\}\geq\frac{99}{100}{\tilde N}^{2}$, so from \cite[Remark 13 ]{R}, there is a sub-horseshoe $\Delta_2$ of $\Delta_2'$ that satisfies condition (1). Also, as $\Delta_2'\subset \Delta_2$, then $\Delta_2'$ also satisfies properties (2) and (3).


\end{proof}

 Notice that  the sub-horseshoe $\Delta_1\subset\Delta$ in Lemma \ref{ind. of horseshoe} with $0<HD(\Delta_1)<\frac{1}{2}$ can be taken such that 
\begin{equation}\label{E4'SGF}
HD(K_1^{s})\sim 1/4,
\end{equation} where $K_1^{s}$ is the regular stable Cantor set associated to $\Delta_1$.
\noindent Now we construct a family of sufficiently independent perturbations of $\cR$ in a neighborhood of a suitable sub-horseshoe of
$\Delta_1$.
Fix $n\in\mathbb{N}$ large and let $\mathcal{SI}^{n}_{g_0}$ be as in Subsection \ref{sec GCSI} (see Remark \ref{R1-Paux}). Since the
transversality condition is open and dense, then we can suppose that the $\Sigma\pitchfork \mathcal{SI}^{n}_{g_0}$ So,  $\alpha_{n}:=\Sigma\pitchfork \mathcal{SI}^{n}_{g_0}$ is a finite family of smooth curves and from Lemma \ref{LIC} applied to the family of curves $\alpha_{n}$ and the sub-horseshoe $\Delta_1$ (compare with Remark \ref{R1-Paux}), we have that given $\tilde{\delta}>0$ there is a sub-horseshoe $\Delta_{0}$ of $\Delta_1$ such that $\Delta_{0}\cap \alpha=\emptyset$ and satisfies that $\tilde{\delta}$-stable property (see Definition \ref{s(u) beta property}), that is,

\begin{equation}\label{E4SGF}
HD(K_{0}^{s})\geq HD(K_{1}^s)-\tilde{\delta},
\end{equation}

\noindent where $K_{0}^{s}$ and  $K_{1}^{s}$ are the stable   Cantor sets associated to $\Delta_{0}$ and $\Delta_1$,  respectively. \\
\indent Moreover, as the set of periodic points of $\mathcal{R}$ of period smaller than $n$ in $\Delta_1$ is finite, we may also assume that $\Delta_0$ does not contain any periodic point of period smaller than $n$. In short, $t_{+}(\theta)\geq n$, $\theta\in \Delta_0$. 

\begin{R}\label{R2SGF}
Given a positive integer $n$, large enough, 
choose a Markov partition $\mathfrak{R}_{0}$ of $\Delta_0$ such that for each $R_{a}\in \mathfrak{R}_{0}$ there is a neighborhood $U_a$ of $R_a$ with the property 
{$\phi^t(U_{a})\cap (\tau_{U^{1/2}_{a}}\setminus U_a^{1/2})=\emptyset$} for 
$$\inf_{x\in\overline{U}_{a}}t_0(x)<t\leq \sup_{x\in\overline{U}_{a}}\sum_{i=0}^n t_i(x) \,\,\, \text{or} \,\, 
-\sup_{\theta\in\overline{U}_{a}}\sum_{i=0}^n t_i(\theta)\leq t<0,$$ where $t_{i}(\theta)=t_{+}(\cR^{i}(\theta))$ and 
$\tau_{U^{1/2}_a}=\pi^{-1}(\pi(U^{1/2}_a)))$. In particular, $\cR^{r}(U_{a})\cap \tau_{U^{1/2}_{a}}=\emptyset$ for $0<r\leq n$. 
\end{R}
The above Remark will be essential to count the number of prohibited transitions of a suitable Markov partition of $\Delta_0$ (see proof of Lemma \ref{L3SGF}).\\
\ \\
Let $\mathfrak{R}_{0}=\{R_1,\dots, R_N\}$ be a Markov partition by squares of size $\epsilon^{1/2}$ of $\Delta_{0}$ as in Remark
\ref{R2SGF}. Note that since $\cR$ is conservative, then each square of this Markov partition $\mathfrak{R}_{0}$, corresponds to an interval 
of the size of the order of $\epsilon$ in $W^{s}(\Delta_0)$ (of the constructions of the stabler Cantor set of $\Delta_0$): there is an iterate
of the square which is a strip in the unstable direction, whose basis is this interval. We call $X:=\{\theta_1,\dots,\theta_{N}\}$ the set 
of words associated with the intervals corresponding to the squares of $\mathfrak{R}_{0}$ in $W^{s}(\Delta_0)$. Without loss of generality
(by considering, if necessary, a suitable sub-horseshoe with almost the same dimension), we can assume, as in Remark 13 at \cite{R}, that 
the transitions $\theta_i\theta_j$ are all admissible. Remember that  $\theta_i$ \emph{disturbs} in $\theta_j$ if $i\ne j$ and 
$T_{R(\theta_j)}\cap\tau_{R(\theta_{i})^{1/2}}\neq \emptyset$, where $R(\theta_{i})$ is the square associated to the word $\theta_i$. 
Consider $P_{\theta_{i}}:=\{\theta_{j}:\theta_i \ \text{ disturbs } \ \theta_j\}=\{\theta_{r_{1}(i)},\cdots,\theta_{r_{p_{i}}(i)}\}$, and  similar to the proof of Lemma \ref{ind. of horseshoe}, $|P_{\theta_i}|=O(N^{1/2})$. 




\begin{Le}\label{L3SGF}
Given a constant $\delta \in (0,1)$, there is a positive integer $n$ such that, if $\Delta_0$ is a sub-horseshoe of $\Delta_1$ as in 
\emph{Remark \ref{R2SGF}} then, for any $i\le N$, 
$$\#\{\theta_i\theta_j\, \, \text{or}\, \, \theta_{j}\theta_i \, \, \text{prohibited by}\,\,  P_{\theta_i}, \, \,  j\leq N \}\leq \delta N.$$
\end{Le}
\begin{proof}
Let us  consider transitions of the type $\theta_i\theta_j$. Write $\theta_i=\alpha \beta \gamma$ such that $\alpha$ is associated to
an interval of the size of the order $\epsilon^{1/2}$ and, $\beta$ and $\gamma$ are associated to intervals of sizes of order
$\epsilon^{1/4}$ in $W^s(\Delta_0)$. Also, let $\alpha=s_1 s_2 \dots s_r$, $\beta=s_1' s_2' \dots s_t'$, $\gamma=s_1''s_2'',\dots s_u''$, and
$P_{\theta_{i}}=\{\theta_{r_{1}(i)},\cdots,\theta_{r_{p_{i}}(i)}\}$. If $\theta_{i}$ prohibits the transition $\theta_{i}\theta_j$ then there exists a word $\theta_{r_{l}(i)}\in P_{\theta_{i}}$ inside (as a factor of) $\theta_i\theta_j$. 

Let us first show that the word $\theta_{r_{l}(i)}$ cannot begin too close from the beginning of $\theta_i$ itself. More precisely, if
it begins with a letter $s_k$ of $\alpha$ then we should have $k>n$. Indeed, if $k\le n$ then the square $R_i$ of the Markov partition
$\mathfrak{R}_{0}$ corresponding to $\theta_i$ is such that ${\cR}^{k-1}(R_i)$ intersects $R_{r_{l}(i)}$ (notice that $k>1$ since, by
definition, $r_{l}(i)\ne i$). Since, for some $(x,v)\in R_{r_{l}(i)}$, there is $0\le t\le t_{+}(x,v)$ such that 
$\phi^t(R_{r_{l}(i)}) \cap \tau_{R^{1/2}_i} \ne \emptyset$. If we take $(y,w)\in R_i$ and $\tilde t=\sum_{i=0}^{k-2} t_i(y,w)$ 
so that ${\cR}^{k-1}(y,w)=\phi^{\tilde t}(y,w)$,  then $\phi^{\tilde t+t}(U_i)\cap \tau_{U^{1/2}_i} \ne \emptyset$, which is a contradiction
with the stated in Remark \ref{R2SGF}. Now we consider three possibilities:

$\bullet$ Assume that a word $\theta_{r_{l}(i)}$ begins with the letter $s_k$ of $\alpha$ (with $k>n$).\\ 
Then, if $\tilde \alpha$ is the factor of 
$\theta_i$  beginning by the letter $s_k$ of $\alpha$ (and also an initial factor of $\theta_{r_{l}(i)}$) associated to an interval of the
size of the order $\epsilon^{1/2}$. The square associated with the word $\theta_{r_{l}(i)}$ belongs to a strip in the unstable direction 
corresponding to the interval (of the size of the order $\epsilon^{1/2}$) in $W^{s}(\Delta_0)$ associated with the word $\tilde \alpha$. 
The previous Lemma implies that there exists a constant $\tilde C>0$ (which depends on the transversality constants in the previous Lemma, 
but is independent of $\epsilon$) such that $\theta_i$ disturbs at most $\tilde C$ squares in this strip. So, given $k$ in this situation, 
there are at most $\tilde C$ possibilities for $\theta_{r_{l}(i)}$. For each such word, the largest part of it will be a factor of 
$\theta_i$, and the remaining will be an initial factor $\hat \alpha$ of $\theta_j$. Let $m'$ be the minimum size of $\hat \alpha$, then
$m'$ is of order  $m$. There is a positive constant $\lambda<1$ (hyperbolicity constant for $\Delta$) such that, for each
$q\ge m'$, if the size of $\hat \alpha$ is $q$, then the number of words $\theta_j$ beginning by $\hat \alpha$ is at most ${\lambda}^q N$. 
Therefore, the number of prohibited transitions $\theta_{i}\theta_j$ in this situation is at most
$$N\cdot \sum_{q\geq m'}\tilde C \cdot \lambda^{q}=\frac{\tilde C N \lambda^{m'}}{1-\lambda}<\frac{\delta N}4.$$

$\bullet$ Assume that a word $\theta_{r_{l}(i)}$ begins with the letter $s'_k$ of $\beta$. \\
Then, if $\tilde \beta$ is the factor of $\theta_i$  beginning with the letter $s'_k$ of $\beta$ (and also an initial factor of $\theta_{r_{l}(i)}$) associated to an interval of the
size of the order of $\epsilon^{1/4}$, the square associated with the word $\theta_{r_{l}(i)}$ belongs to a strip in the unstable direction
corresponding to the interval (of the size of the order of $\epsilon^{1/4}$) in $W^{s}(\Delta_0)$ associated to the word $\tilde \beta$. Since 
the number of intervals of the construction of in $W^{s}(\Delta_0)$ whose sizes are of order $\epsilon^{1/2}$ contained in an
interval of size $\epsilon^{1/4}$ is at most of order $N^{1/4}$. Then, by the discussion of the previous step, the number of squares in 
this strip that are disturbed by $\theta_i$ is at most of the order of $N^{1/4}$. So, given $k$ in this situation, there are at most
$N^{1/4}$ possibilities for $\theta_{r_{l}(i)}$. For each such word, a part of it will be a final factor of $\theta_i$, and the remaining
will be an initial factor $\hat \alpha$ of $\theta_j$, which corresponds to an interval of size at most of the order $\epsilon^{1/2}$. So,
the number of words $\theta_j$ beginning by $\hat \alpha$ is at most of the order of $N^{1/2}$. Since the number of letters in $\beta$ is
of the order $\log N$, the number of prohibited transitions $\theta_{i}\theta_j$ in this situation is $O(\log N \cdot N^{3/4})=o(N)$.\\
\indent $\bullet$ Assume that a  word $\theta_{r_{l}(i)}$ begins with the letter   $s_k''$ of $\gamma$. \\
In this case, a part of the word $\theta_{r_{l}(i)}$ will be a final factor of
$\theta_i$, and the remaining will be an initial factor $\hat \alpha$ of $\theta_j$, which corresponds to an interval of size at most of
order $\epsilon^{3/4}$. So, the number of words $\theta_j$ beginning by $\hat \alpha$ is at most of order $N^{1/4}$. Since
$|P_{\theta_i}|=O(N^{1/2})$, the number of prohibited transitions $\theta_{i}\theta_j$ in this situation is $O(N^{3/4})=o(N)$. Thus, the
total number of prohibited transitions $\theta_{i}\theta_j$ is at most
$$\frac{\delta N}4+o(N)<\frac{\delta N}2.$$
In any case, the total of prohibited transitions is smaller than $\delta N/2$, which implies the result.\\
{We will only give some details of the argument corresponding to the first step: we show that the word $\theta_{r_{l}(i)}$ cannot end too close to the end of $\theta_i$: it should end at least $n$ letters before
it. Indeed, if it ends $k$ letters before the end of $\theta_i$, and $k<m$ then the square $R_{r_{l}(i)}$ of the Markov partition 
$\mathfrak{R}_{0}$ corresponding to $\theta_{r_{l}(i)}$ is such that ${\cR}^{k}(R_{r_{l}(i)})$ intersects $R_i$ (notice that $k>0$ since, 
by definition, $r_{l}(i)\ne i$). Since, for some $(x,v)\in R_{r_{l}(i)}$, there is $0\le t\le t_{+}(x,v)$ such that 
$\phi^t(R_{r_{l}(i)}) \cap \tau_{R^{1/2}_i} \ne \emptyset$. If we take $(y,v)\in R_{r_{l}(i)}$ and $\tilde t=\sum_{i=0}^{k-1} t_i(y,v)$ 
(so that ${\cR}^{k}(y,w)=\phi^{\tilde t}(y,w)$), then $\phi^{t-\tilde t}(U_i)\cap \tau_{U^{1/2}_i} \ne \emptyset$. Thus, we come to a contradiction with what was stated in Remark \ref{R2SGF}.}
\end{proof}

Now we will perform a probabilistic construction. First, fix a parameter $\alpha$ with\break $1/4<\alpha<1/2$. Then, we can state 
the following lemma.
\begin{Le}\label{L4SGF}
Let $f\colon \{1,\dots,\lfloor N^{\alpha} \rfloor\} \to X=\{\theta_1,\dots,\theta_N\}$ a random function \emph{(}i.e., each value $f(i)$ is chosen randomly, with the uniform distribution, and independently from the other\emph{)}. Then $f$ is injective with probability $1-O_{N}(1)$.
\end{Le}
\begin{proof}
The total number of functions $f$ is $N^{\lfloor N^{\alpha} \rfloor}$. The number of injective functions among them is
$$\frac{N!}{(N-\lfloor N^{\alpha} \rfloor)!}=\prod_{j=0}^{\lfloor N^{\alpha} \rfloor-1}(N-j).$$
So, the desired probability is
$$\frac{1}{N^{\lfloor N^{\alpha} \rfloor}}\prod_{j=0}^{\lfloor N^{\alpha} \rfloor-1}(N-j)=\prod_{j=0}^{\lfloor N^{\alpha} \rfloor-1}(1-\frac jN)\ge 1-\frac{\sum_{j=0}^{\lfloor N^{\alpha} \rfloor-1}j}N\ge 1-\frac{(N^{\alpha})^2}{2N}=1-\frac 1{2N^{1-2\alpha}}=1-O_{N}(1).$$
\end{proof}
\noindent Given three indices $i, j, k\le \lfloor N^{\alpha} \rfloor$ with $j\ne k$, we will estimate the probability that given a random 
function $f\colon \{1,\dots,\lfloor N^{\alpha}\rfloor\} \to X$, $f(j)f(k)$ is prohibited by $f(i)$.

We have two cases:

\noindent i)\, From Lemma \ref{L3SGF}, if $i\in \{j, k\}$, the above probability is at most $\delta$.\\
\ \\
ii) \, If $i\notin \{j, k\}$, $i\in \{1,\dots,[N^{\alpha}]\}$, assume that $f(i)$ prohibits $\theta_j\theta_k$. Then the situation is 
as in the following diagram, where we have two representations of the same word:
$$\overbrace{-----}^{\alpha}\overbrace{-----------}^{f(i)}\overbrace{-----}^{\beta}$$
$$\underbrace{----------}_{\theta_j} \underbrace{-----------}_{\theta_k}$$
The number of possibilities for the pair $(\alpha,\beta)$ is $O(N\log N)$ and so, since $|P_{f(i)}|=O(N^{1/2})$, we have 
$$\#\{\theta_j\theta_k:\theta_j\theta_k \ \text{is prohibited by} \ f(i)\}=O(N^{1/2}\cdot N \log N).$$
Therefore, the probability $P_1$ that the transition $f(j)f(k)$ is prohibited by $f(i)$ is 
\begin{equation}\label{E8SGF}
 P_1 = \frac{O(N^{1/2}N\log N)}{N^2}=O(N^{-1/2}\log N).
\end{equation}

The previous estimates imply that, since $\alpha<1/2$, the expected number of prohibited transitions is at most
$$2\delta{\lfloor N^{\alpha}\rfloor}^2+{\lfloor N^{\alpha}\rfloor}^3\cdot O(N^{-1/2}\log N)
=2\delta \lfloor N^\alpha\rfloor^{2}+O(N^{3\alpha-1/2})<3\delta{\lfloor N^{\alpha}\rfloor}^2.$$

Given a function $f$ as Lemma \ref{L4SGF}, we put $\theta_i=f(i)$. It follows that the probability $P$ such that the number of prohibited transitions $\theta_j \theta_k$ with $j\ne k$ is
$\ge 4\delta{\lfloor N^{\alpha}\rfloor}^2$  satisfies  $P\le 3/4$.

\begin{proof}[\textbf{\emph{Proof of Lemma \ref{main-lemma}}}] With the same notation as the last paragraphs.
We consider $\ds A=(a_{ij})$ for $i,j\in\{1,\dots,\lfloor N^{\alpha} \rfloor\}$ the matrix defined by 
\begin{equation}\label{E_Proh}
a_{ij}=
\begin{cases}
1 \ \ \ \ \text{if}\, \, \theta_i\theta_j \ \ \text{is not prohibited}; \cr
\ \cr 
0 \ \ \ \ \text{if} \ \ \theta_i\theta_j \ \text{is prohibited by some} \ 
       \theta_k\in \ \text{Im}f. \cr
\end{cases}
\end{equation}
       
Then, with probability at least $1/4$, $\#\{(i,j): i,j\in\{1,\dots,\lfloor N^{\alpha} \rfloor\} \, \text{with}\,\, a_{ij}=0\}$ is at most $\lfloor N^{\alpha} \rfloor+4\delta{\lfloor N^{\alpha}\rfloor}^2<5\delta{\lfloor N^{\alpha}\rfloor}^2$ (for $N$ large). We assume that $f$  is injective  and that $\delta<1/500$.

Define  $\overline{K}$ the regular Cantor set 
$$\overline{K}:=\{\theta_{i_1} \theta_{i_2}\cdots \theta_{i_n}\cdots | \,\, a_{i_k i_{k+1}}=1,\forall k\ge 1\}\subset K_0^{s}.$$       
By the previous discussion we have $\#\{a_{ij}:a_{ij}=1\}\geq\frac{99}{100}(\lfloor N^{\alpha} \rfloor)^{2}$, so by \cite[Remark 13]{R} we have 
$$HD(\overline{K})\sim \frac{\log \lfloor N^{\alpha} \rfloor}{-\log\epsilon}\sim \alpha HD(K^s_{0}).$$

\noindent Consider the sub-horseshoe of $\Delta_{0}$ defined by 
\begin{equation}\label{Cons of Delta3} 
\Delta_3:=\bigcap_{n\in\mathbb{Z}}\cR^{n}\left(\bigcup_{i,j\le {\lfloor N^{\alpha}\rfloor}, a_{ij}=1} (R(\theta_i)\cap {\cR}^{-1}(R(\theta_j))\right),
\end{equation}
where $R(\theta_i)$ is the square associated to the word $\theta_i$.\\
Since the stable regular Cantor set $K_3^s$ of $\Delta_3$  is equal to $\overline{K}$, by the above discussion we have 
\begin{equation}\label{E9SGF}
HD(K_3^s)\sim \alpha HD(K^s_0).
\end{equation}

\noindent As $HD(\Delta)\sim 2$, then $HD(K^{u})\sim 1$, then by Lemma \ref{ind. of horseshoe} the sub-horseshoe $\Delta_2$ satisfies that 
$\Delta_2\cap \Delta_3=\emptyset$ since $\Delta_3\subset \Delta_1$, and $HD(K^{u}_2)\sim  HD(K^{u})\sim1$. Thus, from 
(\ref{E4'SGF}), (\ref{E4SGF}), and (\ref{E9SGF}) we arrive at $HD(K_3^s)\sim \alpha \frac{1}{4}$. 
Therefore, since $\alpha$ can be taken equal to $\frac{1}{2}-4\epsilon$, with small $\epsilon>0$,
then $HD(K_3^s)\sim \frac{1}{8}-\epsilon$. Thus $$HD({K}^{u}_2)+HD(K_3^s)>1.$$
Therefore, this concludes the proof of the lemma. 
\end{proof}

\subsubsection{Getting the Property V}\label{ProV}
We use sub-horseshoes $\Delta_2$ and $\Delta_3$ given by Lemma \ref{main-lemma} to perform a perturbation of the metric $g_{_{0}}$ to obtain the property $V$ (cf. Section \ref{PropV} and \cite[pp. 19-20]{MY1}).\\
From Remark \ref{Remark-principal-1},  we describe some important properties of the sets ${\mathcal{H}}_{\Delta_3, \Delta_2, \beta}$ and $\widetilde{\mathcal{H}}_{\Delta_3, \Delta_2, \beta}$.
\paragraph{ Properties of ${\mathcal{H}}_{\Delta_3, \Delta_2, \beta}$ and $\widetilde{\mathcal{H}}_{\Delta_3, \Delta_2, \beta}$}\label{Prop.H}
By definition $\widetilde{\mathcal{H}}_{\Delta_3, \Delta_2, \beta}=\{f\in C^2(M,\re): f\circ \pi \in {\mathcal{H}}_{\Delta_3, \Delta_2, \beta}\}$. Given $F\in {\mathcal{H}}_{\Delta_3, \Delta_2, \beta}$ or $f\in \widetilde{\mathcal{H}}_{\Delta_3, \Delta_2,\beta}$, we  denote by $G$, the function $F$ or $f\circ \pi$. Then 
\begin{itemize}
\item[{(i)}] There is a sub-horseshoe $(\Delta_3)_{_{G}}^{s}$ of $\Delta_3$ and a sub-horseshoe $(\Delta_2)_{_{G}}^{u}$ of $\Delta_2$ that satisfies the $\beta$-stable (unstable) property, respectively. 
\item[{(ii)}] There is a Markov partition $R_{G, \Delta_3}^{s}$ of 
$(\Delta_3)_{_{G}}^{s}$, and a Markov partition $R_{G, \Delta_2}^{u}$ of $(\Delta_2)_{_{G}}^{u}$
$$\max  (G)_{\phi}|_{\Sigma\cap(R_{F, \Delta_3}^{^{s}}\cup R_{F, \Delta_2}^{^{u}})}\in H_1(\cR,(\Delta_3)_{_{G}}^{s}\cup(\Delta_2)_{_{G}}^{u})\subset C^{1}(\Sigma\cap (R_{G, \Delta_3}^{^{s}}\cup R_{G, \Delta_2}^{^{u}}),\re).$$
\end{itemize}



Note that $\Delta_2\cap \Delta_3=\emptyset$, $\Delta_2$ has no interference of $\Delta_3$,  and $HD({K}^{u}_2)+HD(K_3^s)>1$, then if $\beta$ is small enough we have that 
\begin{itemize}
\item[(iii)] 
$(\Delta_3)_{_{G}}^{s}\cap (\Delta_2)_{_{G}}^{u}=\emptyset$, $(\Delta_2)_{_{G}}^{u}$ has no interference of $(\Delta_3)_{_{G}}^{s}$,  and 
$$HD(K_{3,G}^{s})+HD(K_{2,G}^{u})>1,$$
where $K_{3,G}^{s}$ is the stable Cantor set of $(\Delta_3)_{_{G}}^{s}$ and $K_{2,G}^{u}$ is the unstable Cantor set of $(\Delta_2)_{_{G}}^{s}$.
\end{itemize}
{\noindent When $G=f\circ \pi$, we denote $(\Delta_3)_{_{G}}^{s}$ by  $(\Delta_3)_{_{f}}^{s}$, $(\Delta_2)_{_{G}}^{u}$ by $(\Delta_2)_{_{f}}^{u}$, $R_{G, \Delta_3}^{s}$ by $R_{f, \Delta_3}^{s}$, $R_{G, \Delta_2}^{u}$ by  $R_{f, \Delta_2}^{u}$, $K_{3,G}^{s}$ by $K_{3,f}^{s}$, and $K_{2,G}^{u}$ by $K_{2,f}^{u}$.}\\


\paragraph{Final Perturbations}\label{FPM}
In \cite{MY1} was proven that if $\varphi$ is a bi-dimensional $C^2$-diffeomorphism having  two disjoint horseshoe $\Omega_1$, $\Omega_2$ with $HD(K_{\Omega_1}^{s})+HD(K_{\Omega_{2}}^u)>1$, where $K_{\Omega_1}^{s}$ is the stable Cantor sets of $\Omega_1$ and $K_{\Omega_{2}}^{{u}}$ is the unstable Cantor set of $\Omega_2$, then, we can perturb $\varphi$ in a suitable Markov partition of $\Omega_1$ without altering the dynamics in ${\Omega}_2$ and such that the new dynamic has a 
horseshoe with the property $V$ (cf. Section \ref{PropV} and \cite[pp. 19-20]{MY1}). \\
At this point, since the properties (iii) of Subsection \ref{Prop.H} fit in the above conditions, then we need to perturb the metric in some suitable opens sets (related to $(\Delta_3)_{_{F}}^{s}$ and $(\Delta_3)_{_{f}}^{s})$) to produce perturbations on $\cR$ with the property $V$.
\ \\

For $F\in{\mathcal{H}}_{\Delta_3, \Delta_2, \beta}$ or $f\in \widetilde{\mathcal{H}}_{\Delta_3, \Delta_2, \beta}$ we denote $\Omega^s_{*}$ the sub-horseshoe of $\Delta_3$, $(\Delta_3)_{_{F}}^{s}$ or $(\Delta_3)_{_{f}}^{s}$, and $\Omega^u_{*}$ the sub-horseshoe of $\Delta_2$, $(\Delta_2)_{_{F}}^{u}$ or $(\Delta_2)_{_{f}}^{u}$, which satisfies the properties (i)-(ii) of Subsection \ref{Prop.H}.
Then by the construction of $\Delta_3$ (see proof of Lemma \ref{main-lemma} or equation (\ref{Cons of Delta3})) we can take a small Markov partition $\{R_1,\dots, R_k\}$ of $\Omega^{s}_{*}$ such that, for all $i,j$
$$\pi(T_{R_j})\cap \pi({R_i^{1/2}})=\emptyset.$$
Our task is to perturb  metric $g_{_{0}}$ in $\pi(R_i^{1/2})$, so
that  this perturbation is independent, \emph{i.e.}, the dynamics of $\cR$ in $R_j$, for $j\neq i$, does not change. Moreover, the dynamic
of $\cR$ in the Markov partition of $\Delta_2$ given Lemma \ref{ind. of horseshoe} also does not change (see Definition \ref{def no dist}).
Since the diameter of $R_i$  is small enough, we can assume that $\pi(R_i^{1/2})$ is contained in a normal coordinate system, 
\emph{i.e.}, there is a point $p\in \pi(R_{i}^{1/2})$, an orthonormal basis $\{e_{1},e_{2}\}$ of $T_{p}M$ (tangent space of $p$) and an open set
$\widetilde{U}_i\subset T_{p}M$ such that $\varphi\colon \widetilde{U}_i\to U^{1/2}_i$ is a diffeomorphism defined by $\varphi(x,y)=\exp_{p}(xe_1+ye_2)$, and  $\pi(R_{i}^{1/2})\subset U^{1/2}_{i}$.\\

Let ${g}$ be a Riemannian metric $C^2$-close to the metric $g_0$ and such that the support of ${g}-g_{_{0}}$ is contained in $U_i^{1/2}$ and satisfies 
\begin{equation}\label{E-fiber disjoint}
\pi_{{g}}^{-1}(x)\cap \pi_{{g_0}}^{-1}(x)=\emptyset \ \ \text{for all}\ \  x\in \pi({R_i^{1/2}}),
\end{equation} 
where $\pi_{{g}}^{-1}(x)=\{v\in T_{x}M: \left\|v\right\|_{{g}}:=\sqrt{{g}(v,v)}=1\}$ and $S^{1}_{g_{_{0}}}(x)=\pi^{-1}(x)$. 
\begin{R}\label{RRRRRR1}
\noindent Using the notation of \emph{Definition \ref{def no dist}}, consider $V_i$ a neighborhood of $R_i$ and
put $\bar{t}_i=\ds \sup _{x\in R_i\cap \Omega^s_{*}} t_{+}(x)$, then we define the   $\mathcal{R}_{1/2}\colon V_i\to \phi^{{\overline{t}}_{i}/2}(V_i)$ as  $\cR_{1/2}(\theta)=\phi^{t(\theta)}(\theta)$, where $\phi^{t(\theta)}(\theta)\in\phi^{{\overline{t}}_{i}/2}(V_i)$  is the first hit. We recall that $\mathcal{S}^{g}\colon S^{g}M\to S^{g_{_{0}}}M$ is a $C^2$-diffeomorphism and denote by $\mathcal{T}^{g}:=(\mathcal{S}^{g})^{-1}$. Thus, analogously,  we consider  $\mathcal{R}_{1/2}^{\tilde{g}}\colon \mathcal{T}^{g}(V_i)\to \phi^{{\overline{t}_{i}}/2}_{{g}}(\mathcal{T}^{g}(V_i))$ the first hit, where $\phi_{{g}}^{t}$ is the geodesic flow of the metric ${g}$.
\end{R}

\begin{Le}\label{L. final}
Let ${g}$ be a Riemannian metric $C^2$-close enough to $g_{_{0}}$ which satisfies the   equation \emph{(\ref{E-fiber disjoint})}. 
 Then 
$$\mathcal{R}^{{g}}({\mathcal{T}^{g}}(W^{s}_{\mathcal{R}}(z)))\cap {\mathcal{T}^{g}}(\mathcal{R}(W^{s}_{\mathcal{R}}(z)))=\emptyset \ \ \text{for all}\ \ z\in R_{i}\cap \Omega^{s}_{*}.$$
\end{Le}
\begin{proof}
First, we show that
\begin{equation}\label{eq of L. Final}
\mathcal{R}^{{g}}_{1/2}({\mathcal{T}^{g}}(W^{s}_{\mathcal{R}}(z)))\cap {\mathcal{T}^{g}}(\mathcal{R}_{1/2}(W^{s}_{\mathcal{R}}(z)))=\emptyset \ \ \text{for all}\ \ z\in R_{i}\cap \Omega^{s}_{*}.
\end{equation}
By contradiction, assume that there is $(x,X)\in \mathcal{R}^{{g}}_{1/2}({\mathcal{T}^{g}}(W^{s}_{\mathcal{R}}(z)))\cap {\mathcal{T}^{g}}(\mathcal{R}_{1/2}(W^{s}_{\mathcal{R}}(z)))$
and
$(x,X)\in \pi_{g}^{-1}(U_i^{1/2})$. Thus, there are $(y,Y), (p,P)\in W^{s}_{\mathcal{R}}(z)$ such that ${\mathcal{T}^{g}}(\mathcal{R}_{\frac 12}(p,P))=\mathcal{R}_{\frac 12}^{{g}}({\mathcal{T}^{g}}(y,Y))=(x,X)$. Since  the support of ${g}-g_{_{0}}$ is contained in $U_i^{1/2}$, then ${\mathcal{T}^{g}}(y,Y)=(y,Y)$, and $$\|\mathcal{R}_{\frac 12}^{{g}}({\mathcal{T}^{g}}(y,Y))\|_{g}=\|\mathcal{R}_{\frac 12}^{{g}}({\mathcal{T}^{g}}(y,Y))\|=\|X\|=1,$$ which implies that $\|{\mathcal{T}^{g}}(\mathcal{R}_{\frac 12}(p,P))\|=1$. On the other hand, if $\mathcal{R}_{\frac 12}(p,P)=(w,W)$, then ${\mathcal{T}^{g}}(w,W)=(w,\frac{W}{\|W\|_{g}})$ and  $1=\|{\mathcal{T}^{g}}(w,W)\|=\frac{\|W\|}{\|W\|_{g}}=\frac{1}{\|W\|_{g}}$. So, ${\mathcal{T}^{g}}(w,W)=(w,W)$ and $\|X\|_{g}=\|{\mathcal{T}^{g}}(w,W)\|_{g}=1$. This implies that $(x,X)\in \pi_{{g}}^{-1}(x)\cap \pi_{{g_0}}^{-1}(x)$ which contradicts (\ref{E-fiber disjoint}).

\noindent Since $\left\|\mathcal{R}_{\frac 12}^{{g}}({\mathcal{T}^{g}}(q, Q))\right\|_{{g}}=1$ for all $(q, Q)\in W^{s}_{\mathcal{R}}(z)$, then by
(\ref{E-fiber disjoint}) we have \\$\mathcal{R}_{\frac 12}^{{g}}({\mathcal{T}^{g}}( W^{s}_{\mathcal{R}}(z)))\cap {\mathcal{T}^{g}}(\pi^{-1}(U_i^{1/2}))=\emptyset$.
Therefore, $\phi^{t}_{{g}}(\mathcal{R}_{\frac 12}^{{g}}({\mathcal{T}^{g}}(q, Q)))=\phi^{t}(\mathcal{R}_{\frac 12}^{{g}}({\mathcal{T}^{g}}(q, Q)))$ for all
$(q, Q)\in W^{s}_{\mathcal{R}}(z)$, where $\phi^{t}_{{g}}$ and $\phi^{t}$ are the geodesic flow of the metrics ${g}$ and $g_{_{0}}$, 
respectively. So, equation (\ref{eq of L. Final}) ends the proof of the lemma.
\end{proof}
 The next step is to exhibit the perturbations or families of perturbations of $g$ that have the property (\ref{E-fiber disjoint}).

For each $i\in\{1, \cdots, k\}$, let $w_i>0$ be a small real parameter. Consider $\alpha_{w_i}(x,y)$ a continuous family of $C^{\infty}$-real function (bump function) with support contained $\widetilde{U}_i$
 ($\widetilde{U}_i$ is the domain of the parametrization $\varphi$), $C^2$-close to the null function (close depend of $w_i$) with
 $\alpha_{w_i}(0,0)=w_i$, and $\alpha_{0}(x,y)\equiv 0$. Moreover, if $w_i\neq 0$, then $\alpha_{w_i}(x,y)\neq 0$ for all
 $(x,y)\in \varphi^{-1}(\pi({R_i^{1/2}}))$. Thus, in local coordinates, we exhibit three family of Riemannian metric $g^{w_i}$ given by \\
\ \\
{(a)} $g^{w_i}=(1+\alpha_{w_i}(x,y))g$,\\
\ 
{(b)} $g^{w_i}=e^{\alpha_{w_i}(x,y)}g$,\\
\ 
{(c)}$({g^{w_i}})_{00}(x,y)=g_{00}(x,y)+\alpha_{w_i}(x,y)$, $(g^{w_i})_{rs}(x,y)=g_{rs}(x,y)  \ \ (r,s)\neq(0,0)$

\noindent which satisfy the property (\ref{E-fiber disjoint}) and therefore, it
satisfy Lemma \ref{L. final}.\\
\ \\
We denote by $\cR^{w_i}\colon\Sigma_{{g^{w_i}}}\to \Sigma_{{g^{w_i}}}$ the Poincar\'e map given by $g^{w_{i}}$ {(see Subsection \ref{Dim. Red})}. Define the following application $\Phi^{w}_{i}$ on $R_i$ by 
$$\Phi^{w_i}(x,v):=\cR^{-1}\circ \mathcal{S}^{{{g^{w_i}}}}\circ \cR^{w_i}\circ \mathcal{T}^{{g^{w_i}}}(x,v) \ \ \text{for} \ \ (x,v)\in R_{i},$$
{which is $C^2$-close to the identity if $w_i$ is small enough.}\\

As a consequence of Lemma \ref{L. final} we have 
\begin{C}\label{C-final}
For $w_i$ small enough, $\Phi^{w_i}(W^{s}_{\cR}(z))\cap W^{s}_{\cR}(z)=\emptyset$, for $z\in R_{i}\cap \Omega^{s}_{*}$. 
\end{C}

Since $i\in \{1, \dots, k\}$, then we consider the a vector $w\in \re^k$, $w=(w_1, w_2, \cdots w_k)$. Then, if $||w||$ is small enough we define a metric $g_{w}$, $C^2$-close to  metric $g$ by

\[ g_{w} = \left\{ \begin{array}{lll}
        g^{w_i}, & \mbox{\text{in a small neighborhood} of 
         $\pi(R_i^{1/2})$ }\\
           & \\
       g, & \mbox{$\text{otherwise.} $}\end{array} \right. \]     

\noindent Put $\Phi^{w}(x,v):=\Phi^{w_i}(x,v)$ if $(x,v)\in R_i$. So, it is easy to see that $\Phi^{w}=\cR^{-1}\circ \mathcal{S}^{g_w}\circ \cR^{w}\circ \mathcal{T}^{g_w}$, where $\cR^{w}$ is  the Poincar\'e map $\cR^{w}\colon\Sigma_{{g_{w}}}\to \Sigma_{{g_{w}}}$. From Lemma \ref{lemma_conjugate} (see the notation of Subsection \ref{Dim. Red})  ${\mathcal{P}_{w}}:=\cR^{J_{_{g_0}}^{w}}=\cR\circ\Phi^{w}$, which is  $C^2$-close to $\cR$, whenever $||w||$ is small enough. \\ Put $\Omega_{*}:=\Omega_{*}^{s}\cup \Omega_{*}^{u}$, then as $\Omega_{*}^{s}\cap \Omega_{*}^{u}=\emptyset$ and the perturbation on the metric is done in a small neighborhood of $\bigcup_{i=1}^{k}\pi({R_i^{1/2}})$, where $\{R_1, \dots, R_k\}$ is a small Markov partition of $\Omega_{*}^{s}$, then the hyperbolic continuation of $\Omega_{*}$ by $\mathcal{P}_{w}$ is $\widetilde{\Omega}_{*,w}:= \widetilde{\Omega}_{*,w}^{s}\cup \Omega^{u}_{*}$, with $\widetilde{\Omega}_{*,w}^{s}$ being the hyperbolic continuation of $\Omega_{*}^{s}$ by $\mathcal{P}_{w}$ and we still have $\widetilde{\Omega}_{*,w}^{s}\cap \Omega^{u}_{*}=\emptyset$.

Therefore, as an immediate consequence of Corollary \ref{C-final} we have 
the following lemma.
\begin{Le}\label{LeProV}
For a set of positive Lebesgue measure of the parameter $w$, with $||w||$ small enough,  the pair $(\mathcal{P}_{w}, \widetilde{\Omega}_{*,w})$  has the property $V$. 
\end{Le}
\begin{proof}
Observe that, by Corollary \ref{C-final}
$$\mathcal{P}_{w}(W^{s}_{\cR}(z))\cap \cR(W^{s}_{\cR}(z))=\emptyset, \, \, \text{for} \, \, z \in \Omega^{s}_{*}.$$
Moreover, as $\Omega_{*}^{s}$ and $\Omega_{*}^{u}$ satisfy the properties (i), (ii), and (iii) of the beginning of the Subsection \ref{Prop.H}, then 
\begin{itemize}
\item $\Omega_{*}^{s}\cap \Omega_{*}^{u}=\emptyset$,
\item $HD(K_{\Omega_{*}^{s}}^{s})+HD(K_{\Omega_{*}^{s}}^{s})>1$, where $K_{\Omega_{*}^{s}}^{s}$ is stable Cantor set of $\Omega_{*}^{s}$ and $K_{\Omega_{*}^{u}}^{u}$ is the unstable Cantor set of $\Omega_{*}^{u}$.
\end{itemize}

\noindent Thus, the family $\mathcal{P}_{w}$ satisfies the condition to get the property 
 $V$ ({cf. Section \ref{PropV} and \cite[Appendix 5.3]{MY1}}).
 \end{proof}

By Lemma \ref{lemma_conjugate}  the map $\mathcal{P}_{w}$ and $\cR^w$ are conjugated. Thus, as in  Subsection \ref{Dim. Red} the set $\Omega_{*,w}=\mathcal{T}^{g_w}(\widetilde{\Omega}_{*,w})=\mathcal{T}^{g_w}(\widetilde{\Omega}_{*,w}^{s})\cup\Omega_{*}^{u}:={\Omega}_{*,w}^{s}\cup\Omega_{*}^{u}$ is the hyperbolic continuation of 
$\widetilde{\Omega}_{*,w}$ by $\cR^w$. Consequently, from Lemma \ref{LeProV}, we have 
\begin{Le}\label{Crucial-Lemma}
For a set of positive Lebesgue measure of the parameter $w$, with $||w||$ small enough, the pair $(\mathcal{R}^{w}, \Omega_{*,w})$  has the property $V$.
\end{Le}

\begin{R}Let $\mathcal{G}\subset \mathcal{V}\subset$ {\boldmath$\mathcal{G}$}$^{3}(M)$ a neighborhood of $g_{w}$. Let $\Omega_{*,w,g}$ be the basic set for $\cR^{w}$ as in \emph{Remark \ref{R1-conjugates}}. Moreover
$$\Omega_{*,w,g}=\Omega_{*,w,g}^{s}\cup \Omega^{u}_{*,g} \, \, \, \, \text{with}\, \, \, \Omega_{*,w,g}^{s}\cap \Omega^{u}_{*,g}=\emptyset,$$
where $\Omega_{*,w,g}^{s}$ and  $\Omega^{u}_{*,g}$ as in  \emph{Remark \ref{R1-conjugates}}. \\

\end{R}

The last discussion, and the robustness of the property $V$ ({cf. Section \ref{PropV} and \cite[Appendix 5.3]{MY1}}), provides that
\begin{Le}\label{Crucial-Lemma-2}
If $\mathcal{G}$ is small enough, then for all $g\in \mathcal{G}$, the pair $(\cR^g, \Omega_{*,w,g})$ has the property $V$.
\end{Le}

These latter lemmas are an important tool for proving Theorems \ref{Theorem 1} and \ref{Theorem 2}, since it allows us to use the Main Theorem in \cite{MR}.
\section{The Final Proofs}\label{P. MT}
Before beginning with the proofs of Theorems \ref{Theorem 1} and \ref{Theorem 2}, we present the \ref{Main Theorem MR}.
\begin{T}[\text{\cite[Main Theorem]{MR}}]\label{Main Theorem MR}
Let $\Lambda$ be a horseshoe associated to a $C^2$-diffeomorphism $\varphi$ of  a surface $M$, such that  $HD(\Lambda)>1$.
If $\varphi_{0}$ is  $C^2$ sufficiently close to $\varphi$  such that the pair $(\varphi_0, \Lambda_{\varphi_{0}})$ has the property $V$,
where $\Lambda_{\varphi_{0}}$ is the hyperbolic continuation of $\Lambda$.   Then, there is $C^{2}$-neighborhood $\mathcal{W}$ of $\varphi_{0}$ such that, if $\Lambda_{\psi}$ denotes the continuation of $\Lambda$ associated to $\psi\in \mathcal{W}$, there is an open and dense set $H_{1}(\psi, \Lambda_{\psi})\subset C^{1}(M,\re)$ such that for all $f\in H_{1}(\psi, \Lambda_{\psi})$, we have 
\begin{equation*}
\emph{int}\,\mathbb{L}(\psi, \Lambda_{\psi}, f )\neq\emptyset \quad \text{and}
\quad   
\emph{int}\, \mathbb{M}(\psi, \Lambda_{\psi}, f) \neq\emptyset,
\end{equation*}
where \emph{int\ A} denotes the interior of the set \emph{A}.\\
\end{T}
\noindent The set $H_{1}(\psi, \Lambda_\psi)$ of the above theorem is described by  (compare with Definition \ref{DLmax00})
\begin{equation*}
H_{1}(\psi, \Lambda_\psi)=\left\{f\in C^{1}(M,\re):\#M_{f}(\Lambda_{\psi})=1, \  z\in M_{f}(\Lambda_{\psi}), \ D\psi_{z}(e_{z}^{j})\neq 0, j=s,u\right\},
\end{equation*}
 where $M_{f}(\Lambda_{\psi}):=\{z\in \Delta: f(z)\geq f(x)\ \text{for all} \, x\in \Lambda_{\psi}\}$ is  the set of maximum points of
 $f$ in $\Lambda_{\psi}$ and $e_{z}^{j}\in E^{j}(z)$ is a unit vector, $j=s,u$. \\
\ \\
An immediate consequence of the above theorem, which will be useful to the proof of Theorem \ref{T3}, is the following corollary 
\begin{C}\label{C-MTMR}
The above theorem is also valid for the set of functions
\begin{equation*}
H_{\infty}(\psi, \Lambda_\psi)=\left\{f\in C^{1}(M,\re):\#M_{f}(\Lambda_{\psi})<\infty, \, \, \, D\psi_{z}(e_{z}^{j})\neq 0, \, \,  z\in M_{f}(\Lambda_{\psi}), \, j=s,u\right\}.
\end{equation*}
\end{C}

\begin{proof}[\bf{Proof of the Theorem \ref{Theorem 1}}]
Let $F\in{\mathcal{H}}_{\Delta_3, \Delta_2, \beta}$ for $\beta$ small, in this case, $\Omega^s_{*}=(\Delta_3)_{_{F}}^{s}$ a sub-horseshoe of $\Delta_3$ and $\Omega^{u}_{*}=(\Delta_{2})_{F}^{u}$ a sub-horseshoe of $\Delta_2$. Then, by Lemma \ref{Crucial-Lemma}, we consider the metric $g_w$ such that $(\mathcal{R}^{w}, \Omega_{*,w})$ has the property $V$ (see Section \ref{PropV}). By Remark \ref{Remark-principal-0}, we have that $F\circ \mathcal{S}^{g_w}\in {\mathcal{H}}_{(\Delta_3)_{w}, {\Delta_2}, \tilde{\beta}}$, where $(\Delta_{3})_{w}$ is the hyperbolic continuation by $\mathcal{R}^{w}$ of $\Delta_3$. Therefore, by the definition of ${\mathcal{H}}_{(\Delta_3)_{w}, {\Delta_2}, \tilde{\beta}}$ we have that 
$$\max  (F\circ \mathcal{S}^{g_w})_{\phi_w}|_{\Sigma_{_{g_w}}\cap (R_{F\circ \mathcal{S}^{w},\Omega_{*,w}^{s}}^{u}\cup \, R_{F\circ \mathcal{S}^{w},\Omega_{*}^{u}}^{s})}\in H_1(\cR^{w},(\Omega_{*,w}^{s}\cup \Omega_{*}^{u})),$$ 
where $R_{F\circ \mathcal{S}^{w},\Omega_{*,w}^{s}}^{u}$, and $R_{F\circ \mathcal{S}^{w},\Omega_{*}^{u}}^{s}$ are the Markov partitions given by the definition of  ${\mathcal{H}}_{(\Delta_3)_{w}, {\Delta_2}, \tilde{\beta}}$.\\
Thus, as  $(\mathcal{R}^{w}, \Omega_{*}^{w})$ has the property $V$, then by the Main Theorem
at \cite{MR} we can conclude that 
\begin{eqnarray}\label{Interior}
\text{int}\, \mathbb{M}(\cR^{w},\Omega_{*}^{w}, \max  (F\circ \mathcal{S}^{g_w})_{\phi_w}|_{\Sigma_{_{g_w}}\cap (R_{F\circ \mathcal{S}^{w},\Omega_{*,w}^{s}}^{u}\cup \, R_{F\circ \mathcal{S}^{w},\Omega_{*}^{u}}^{s})}&\neq &  \emptyset,\\ \label{Interior-1}
\text{int}\, \mathbb{L}(\cR^{w},\Omega_{*}^{w}, \max  (F\circ \mathcal{S}^{g_w})_{\phi_w}|_{\Sigma_{_{g_w}}\cap (R_{F\circ \mathcal{S}^{w},\Omega_{*,w}^{s}}^{u}\cup \, R_{F\circ \mathcal{S}^{w},\Omega_{*}^{u}}^{s})} &\neq &  \emptyset.
\end{eqnarray}

\noindent Let $\mathcal{G}\subset \mathcal{V}\subset$ {\boldmath$\mathcal{G}$}$^{3}(M)$ a neighborhood of $g_{w}$ as in Lemma \ref{Crucial-Lemma-2}, then  for every $g\in \mathcal{G}$, the pair $(\cR^g, \Omega_{*,w,g})$  has the property $V$. 
So, for every $F\in{\mathcal{H}}_{\Delta_3, \Delta_2, \beta}$ the function  $F\circ \mathcal{S}^{g}\in {\mathcal{H}}_{(\Delta_3)_{g}, {\Delta_2}, \tilde{\beta}}$ and by Main Theorem at \cite{MR}, we have that the triplet $(\cR^g, \Omega_{*,w,g}, \max (F\circ \mathcal{S}^{g})_{\phi_g})$ satisfies the  equations (\ref{Interior}) and (\ref{Interior-1}). Therefore, put $\mathcal{H}_{g,\Lambda}:={\mathcal{H}}_{(\Delta_3)_{g}, {\Delta_2}, \tilde{\beta}}$, then  as ${\mathcal{H}}_{\Delta_3, \Delta_2, \beta}$ is dense and $C^2$-open, then we have completed the proof of
Theorem \ref{Theorem 1} (and, \emph{a fortiori}, Theorem \ref{T1}).

\end{proof}
\begin{proof}[\bf{Proof of the Theorem \ref{Theorem 2}}] 
The proof is completely analogous to the proof of Theorem \ref{Theorem 1},  just change ${\mathcal{H}}_{\Delta_3, \Delta_2, \beta}$, $\Omega_{*}^{s}$,  $\Omega_{*}^{u}$, and $F\circ \mathcal{S}^{g}$ by $\tilde{\mathcal{H}}_{\Delta_3, \Delta_2, \beta}$, $(\Delta_3)_{f}^{s}$, $(\Delta_2)_{f}^{u}$, and $f\circ \pi_g$, respectively. The set of functions will be $\widetilde{\mathcal{H}}_{g,\Lambda}:=\widetilde{\mathcal{H}}_{(\Delta_3)_{g}, {\Delta_2}, \tilde{\beta}}$, for some $\tilde{\beta}$ and $g\in \mathcal{G}\subset \mathcal{V}\subset$ {\boldmath$\mathcal{G}$}$^{3}(M)$ a neighborhood of $g_{w}$ as in Lemma \ref{Crucial-Lemma-2}.
\end{proof}
\section{The Spectrum of the Height Function}\label{SHF}
Let \textit{e} be an \emph{end} of the manifold $M$ and $\Gamma(t)$ a ray that defines the \emph{end}\ $\emph{e}$. Thus, the \emph{height function} associated to $e$ is  defined by 
$$ht_{e}(x)=\lim_{t\to+\infty}d(x,\Gamma(t))-t.$$
Usually  $-ht_{e}(x)$ is called the \emph{Busemann  function}  associated to end $e$ and denoted by $b_{\Gamma}(x)$  {(see \cite{Eberlein} for the precise definition of an \textit{end})}.

\subsection{Differentiability of Busemann Function}\label{D. of B. F.}
It is known that the Busemann functions are not always differentiable, however, the points where we lose the differentiability are well known and we are going to use it to show some results where we have the differentiability of the Busseamnn functions, at least in a region close to the set hyperbolic $\Delta$.


\begin{Defi}\label{co-points}
Let $\gamma\colon [0,\infty) \to M$ be a forward ray. Then a forward ray\break $\sigma\colon [0,\infty) \to M$ emanating from 
$x:= \sigma(0)\in M$ is called a \emph{forward coray} $($or a \emph{forward asymptotic ray}$)$ to $\gamma$ if there exists a  divergent sequence of numbers $\{t_j\}$ and a sequence of minimal 
geodesics ${\sigma_j}$ with  $q_j:= \sigma_{j}(0)$,  $\sigma_j(l_j):=\gamma(t_j)$, for some $l_j>0$, and  such that 
$\ds\lim_{j\to +\infty} q_j = \sigma(0)$ and $\sigma'(0) = \ds\lim_{j\to +\infty} \sigma'_{j}(0)$  \emph{(see Figure 1)}. A forward coray of $\gamma$ is called \emph{maximal} if for any $\epsilon > 0$ its extension to $[-\epsilon,+\infty)$ is no longer a coray to $\gamma$. The origin points of
maximal corays of $\gamma$ are called the \emph{co-points to $\gamma$} and it is denoted by $C_{\gamma}$.
\begin{figure}[hbtp]
\label{fig:Figure10.pdf}
\centering
\includegraphics[scale=0.5]{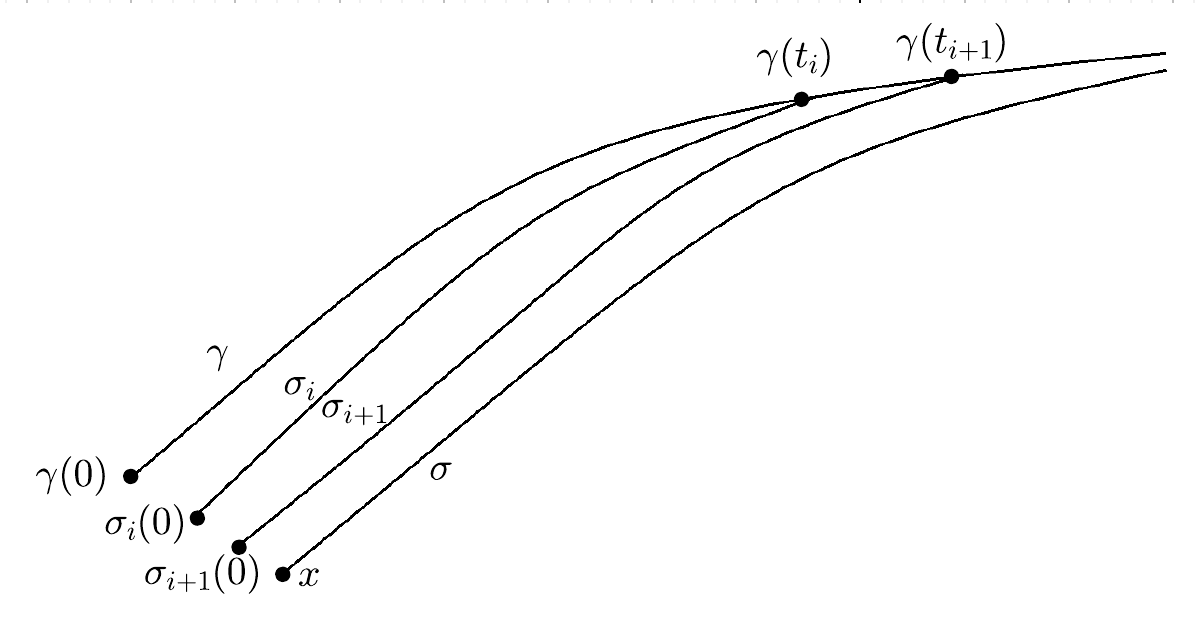}
\caption{Forward Coray}
\end{figure}

\end{Defi}
\begin{Pro}[\text{\cite[Proposition 1]{Innami1}}]
If $p\in M$ is not a co-point to $\gamma$, then the Busemann function $b_{\gamma}$ is differentiable at $p$. 
\end{Pro}

\begin{R}\label{Not-differ}
Let $ND(b_{\gamma})\subset M$  
be the set of non-differentiable points of 
the Busemann function $b_{\gamma}$. Then, 
$ND(b_\gamma)\subset C_{\gamma} \subset \overline{ND(b_{\gamma})}$.
\end{R}
\noindent The following Lemma gives us information on what happens at the end $e$, for an asymptotic ray to $\gamma$ ( \cite[Theorem 3.8.2]{SST})
\begin{Le}\label{Increasing Busemann}
Let $\gamma\colon [0,\infty) \to M$ a ray.
\begin{enumerate}
\item[\emph{(1)}] For any $x\in M$, there exists a ray $\sigma$ asymptotic to $\gamma$ such that $\sigma(0)=x$.
\item[\emph{(2)}] For any ray $\sigma$ asymptotic to $\gamma$ it holds $b_{\gamma}(\sigma(s))=b_{\gamma}(\sigma(0))+s$.
\item[\emph{(3)}] If $b_{\gamma}$ is differentiable at $x\in M$, then $\sigma(s)=\exp_{x}(s\nabla b_{\gamma}(x))$
is a unique ray asymptotic to $\gamma$ emanating from $x$, where $\nabla b_{\gamma}(x)$ is the Riemannian gradient of $b_{\gamma}$ at x.
\end{enumerate}
\end{Le}
\begin{proof}
The proof of item (1) is trivial. Item (3) is a consequence of the item (2). Thus, let us  prove item (2).
Take $\{t_i\}_{i\in\mathbb{N}}$ and $\{\sigma_i\}_{i\in\mathbb{N}}$ as in the definition of  asymptotic. It holds that
$$b_{\gamma}(\sigma(s))=\lim_{t_i\to \infty}\{t_i-d(\sigma(s), \gamma(t_i))\}$$
by the definition of $b_{\gamma}$. We can replace $\sigma(s)$ in the right hand side with $\sigma_{i}(s)$ since
$$|d(\sigma(s), \gamma(t_i))-d(\sigma_{i}(s),\gamma(t_i))|\leq d(\sigma(s),\sigma_{i}(s))\to 0 $$
as $i\to \infty$. Hence we have, by the choice of $\sigma_i$, 
\begin{eqnarray*}
b_{\gamma}(\sigma(s))&=&\lim_{i\to \infty}\{t_i-d(\sigma_{i}(s), \gamma(t_i))\}=\lim_{i\to \infty}\{t_i-d(\sigma_{i}(0), \gamma(t_i))+s\}\\
&=&\lim_{i\to \infty}\{t_i-d(\sigma(0), \gamma(t_i))+s\}=b_{\gamma}(\sigma(0))+s.
\end{eqnarray*}
\end{proof}

\noindent In our case, since $M$ is a surface with pinched curvature and finite volume, then Gromov  \cite{Gromov} and Heintze  \cite{Heintze} showed that $M$ is diffeomorphic to the interior of a compact manifold with boundary. Moreover, for every point $p \in M$, there are only finitely many distinct minimizing geodesic rays emanating from $p$. In particular, if $M$ has only finitely many \textit{ends} and each \textit{end is parabolic}, \emph{i.e.}, the quotient of a horoball of 
$\widetilde{M}$ (the universal covering of $M$) by a group of parabolic isometries with a compact cross-section (see \cite{Eberlein}). Therefore, each \textit{end} of $M$ is quasi-isometric to a ray. 
Thus, suppose that $C_0, C_1, \dots, C_k$ are simple closed geodesic polygons on $M$, no two of which have common point, such that each $C_i$ bounds a tube $U_i$ homeomorphic to a circular disk punctured at the center (or a cylinder), that is, a representation of the \textit{ends} of $M$. Then $M\setminus \cup_{i=0}^{k}U_i$ is a bounded open set. 
\begin{Defi}
A tube $U_i$ is said non-expanding or expanding according to any minimal sequence of closed curves homotopic with the curve $C_i$ being bounded or non.
\end{Defi}
\begin{R}\label{Tubes not-expanding}
In this case, since  $U_i$ represents a \textit{parabolic end}, then any minimal sequence of closed curves homotopic to the boundary $C_i$ is bounded, so each  $U_i$ is {non-expanding} \emph{(cf. \cite{Nasu})}.
\end{R}
\begin{T}[\text{\cite[Theorem 5.2]{Nasu}}]
Suppose that $S$ is a surface that is homeomorphic to a cylinder. Then for
a ray $\gamma$ the set $C_{\gamma}$ is an unbounded simple arc or vacuous according as the
tube $U_0$ which contains $\gamma$ or a subray of $\gamma$, is expanding or non-expanding.
\end{T}
\indent As a consequence of the above theorem, Lemma \ref{Increasing Busemann}, Remark \ref{Not-differ}, and Remark \ref{Tubes not-expanding} we have  {(see \cite[chap 3]{Ballmann}}) that
\begin{Le}\label{L-Nasu}
The Busemann function $b_{\gamma}$ increases in  tube $U_0$ which contains $\gamma$ or a subray of $\gamma$. Moreover,  it is $C^{2}$ in $U_0$.
\end{Le}
\begin{proof}
Item (2) of Lemma \ref{Increasing Busemann} implies the first part. Now, by Remark \ref{Tubes not-expanding} the tube $U_0$ is not expanding, then by
the above theorem, we have that $C_{\gamma}\cap U_0=\emptyset$. Thus, by Lemma \ref{Increasing Busemann} the Busemann function $b_{\gamma}$ is
$C^{1}$ in $U_0$. Now, we will prove that it is also $C^{2}$.
Let $x\in U_0$ be and  Let $t_n$ be a sequence such that $t_n\to \infty$, let $$b_{n}(x)=t_n-d(\gamma(t_n),x).$$
Then $b_n\to b_{\gamma}$ uniformly on compact subsets of $M$. Put  $r_n:=d(x,\gamma(t_n))$, and let $X$ be a smooth vector field on $M$ and let $\alpha_{x}(s)$ be an integral curve of the field $X$ with $\alpha_{x}(0)=x$, then for each $s$ 
there is $v_{n}(s)\in \pi^{-1}(\alpha_{x}(s))$ such that $\exp_{\alpha_{x}(s)}t(s)v_{n}(s)=\gamma(t_n)$ for 
some $t(s)\in\re$, $i=1,\dots,m$. It is  easy to see that each $v_{n}(s)$  is smooth in $s$ and moreover,
\begin{equation}\label{E1-L-Nasu}
\ds\frac{d}{ds}v_{n}(s)\Bigg|_{s=0}=-J'_{n}(x,0),
\end{equation}
where $J_{n}(x,\cdot)$ is the Jacobi vector field along of the geodesic $\exp_{x}\,tv_{n}(0)$ with\break $J_{n}(x,0)=X(x)$ and $J_{n}(x,r_n)=0$.
Note that, as the curvature of $M$ is negative, then $M$ has no conjugate points and therefore $J_{n}(x,\cdot)$ is unique with these properties.

For $s$ small $\alpha_{x}(s)\in U_0$, thus $b_{\gamma}$ is differentiable in $\alpha_{x}(s)$ and by Lemma \ref{Increasing Busemann} we have that 
\begin{equation}\label{E2-L-Nasu}
\nabla b_{\gamma}(\alpha_{x}(s))=\lim_{n\to +\infty} v_{n}(s).
\end{equation}
Then by equations $(\ref{E1-L-Nasu})$ and $(\ref{E2-L-Nasu})$ we have that for every $i\in\{1,\dots,m\}$,
$$D_{X(x)}\nabla \, b_{\gamma}(x)=\frac{\partial}{\partial s} \nabla\, 
b_{\gamma}(\alpha_{x}(s))\Bigg|_{s=0}=\ds\lim_{n\to +\infty}\frac{d}{ds}v_{n}(s)|_{s=0}=-\lim_{n\to +\infty}J'_{n}(x,0).$$
Since $J_{n}(x,0)=X(x)$ and $J_{n}(x,r_{n})=0$ and $M$ has no conjugate points \break $J_{n}(x,\cdot)\to J_{X}(\cdot)$ the unique stable Jacobi vector 
field along $\exp_{x}\,t\nabla \, b_{\gamma}(x)$ with $J_{X}(0)=X(x)$ (cf. Section \ref{sec Geo SM}), and  thus $J'_{n}(x,\cdot) \to J'_{X}(\cdot)$.\\
Therefore, we have 
\begin{equation}\label{E3-L-Nasu}
D_{X}\nabla \, b_{\gamma}(x)=-J'_{X}(0),
\end{equation}
for $x\in U_0$. Thus, as $X$ is smooth, we conclude  that $b_{\gamma}$ is $C^{2}$ in the tube $U_{0}$.
\end{proof}
\subsection{Horseshoe for Tubes}\label{H. for tube}
Fix a ray $\gamma$ and consider $U_{0}$ as in Lemma \ref{L-Nasu}. In this section, we construct a hyperbolic set for $\phi^t$ whose projection on $M$ intersects the tube $U_0$.
\begin{Defi}
We say that an open set $V_{0}\subset U_0$ is a \emph{collar of $U_0$}, if $U_0\setminus V_{0}$ has two connected components. 
\end{Defi}

\begin{Le} \label{Collar}Let $V_0$ be a collar of $U_0$, then there is a basic set $\Lambda_{0}\subset SM$ with $HD(\Lambda_0)$ close enough to $3$ and such that
$$\pi(\Lambda_0)\cap V_0=\emptyset,$$
where $\pi\colon SM\to M$ is the canonical projection.
\end{Le}
\begin{proof}
Since the geodesic flow is transitive, then the construction given in Lemma \ref{L1-HD>1} provides the result.
\end{proof}
\subsubsection{Poincar\'e Map for the Tube}
By Lemma \ref{L-Nasu} and definition of $\Lambda_{0}$, for each $(x,v)\in \Lambda_0$ there is a $\bar{t}(x,v)$ such that 
$\ds\max_{t\in\re}b_\gamma(\pi(\phi^{t}(x,v))=b_\gamma(\pi(\phi^{\bar{t}(x,v)}(x,v))$ and $\phi^{\bar{t}(x,v)}(x,v)\in U_0$, since $\Lambda_0$ is compact invariant set. Thus 
$b_{\gamma}$ is $C^{2}$ in $\phi^{\bar{t}(x,v)}(x,v)$.
Consider the set $$A(x,v):=\{t\in \re: \ b_\gamma(\pi(\phi^{t}(x,v)))=b_\gamma(\pi(\phi^{\bar{t}(x,v)}(x,v)))\}.$$
Is easy to see that $1\leq |A(x,v)|<\infty$, where $|A(x,v)|$ is the cardinality of $A(x,v)$.
\begin{Le}\label{Tangent}
There is $\eta>0$ such that for every $(x,v)\in \Lambda_{0}$ and any  $s,r\in A(x,v)$, we have that $|s-r|> \eta$. 

\end{Le}
\begin{proof}
The proof of the lemma is a consequence of the following claim.\\
\textbf{Claim:} For any $(x,v)\in \Lambda_0$, put $f_{\gamma}(t)= b_\gamma(\pi(\phi^{t}(x,v)))$, then $f''(s)\neq 0$ for all $s\in A(x,v)$.\\ 
\textbf{Proof of Claim:}\\
By definition of $A(x,v)$ and Lemma \ref{L-Nasu} we have that $f(t)$ is $C^{2}$ in a small neighborhood of $A(x,v)$ and  $f'(s)=0$ for all
$s\in A(x,v)$. In this neighborhood, from proof of Lemma \ref{L-Nasu} we can see that
$$f''(t)=-\langle J'_{\gamma_{v}'(t)}(0),\gamma_{v}'(t)\rangle,$$
where $J_{\gamma_{v}'(t)}(\cdot)$ is a stable Jacobi vector field along $\sigma_{\gamma_{v}(t),\gamma}$ with
$J_{\gamma_{v}'(t)}(0)=\gamma'_{v}(t)$. Here $\sigma_{\gamma_{v}(t),\gamma}$ is the only asymptotic geodesic from  $\gamma_{v}(t)$ to
$\gamma$ (see equation (\ref{E3-L-Nasu})).\\
If $f''(t)=0$, since  the curvature is negative, the real function $g(z)=\Vert J_{\gamma_{v}'(t)}(z)\Vert^{2}$ is strictly convex.  Moreover, 
$g(0)=\Vert J_{\gamma_{v}'(t)}(0)\Vert^{2}=\Vert\gamma'_{v}(t)\Vert^{2}=1$ and $g'(0)=2\cdot f''(t)=0$. So, $g$ is a strictly convex function with 
a minimum point in $z=0$, then $g$ is an unbounded function. But since $J_{\gamma_{v}'(t)}(\cdot)$ is a stable Jacobi vector field, then it holds 
that  $\Vert J_{\gamma_{v}'(t)}(z)\Vert$ is bounded, so we have a contradiction. Thus, we concluded the proof of the claim.

Now as $f_{\gamma}$ is $C^{2}$ in a neighborhood of $A(x,v)$, then the claim implies that its critical points are isolated. Thus, we conclude the
proof of the Lemma.
\end{proof}
The basic set $\Lambda_0$ satisfies all the properties of Subsection \ref{Dim. Red}, the and put also $\cR\colon \Sigma \to \Sigma$ the Poincar\'e map, when $\Sigma=\bigsqcup\limits_{i=1}^k \Sigma_i$, each $\Sigma_i$ is a Good Cross-Section, and $\Lambda_0\cap \Sigma$ a basic set for $\cR$ with $HD(\Lambda_0\cap \Sigma)>1$ (in fact close to $2$).
\noindent The Lemma \ref{Tangent} and the construction made in \cite[Section 3]{R}, allow us to choose Good Cross Sections $\Sigma_i$ such that 
$$\ds\Lambda_0\subset \bigcup^{k}_{i=1}\phi^{(-\gamma,\gamma)}(\text{int}\,({\Sigma}_i)),$$  
where  $\gamma>0$ (which depend of $\eta$ at Lemma \ref{Tangent}) and 

\begin{itemize}
\item[(1)] For $(x,v)\in \Lambda_0 \cap \Sigma$, then $\phi^{s}(x,v)\notin \Sigma_{i}$ for all $s\in A(x,v)$ and all $i=1,\dots, k$. 
\item[(2)] Consider the set $\mathcal{A}=\ds\bigcup_{s\in A(x,v)}\bigcup_{(x,v)\in \Lambda_0}\phi^{s}(x,v)$.\\ Then by item (1),
$\Lambda_{0}\cap \Sigma \cap \mathcal{A}=\emptyset$. Therefore,  $0\notin A(y,w)$ for all $(y,w)\in \Lambda_{0}\cap \Sigma$. \\Also, for all 
$(y,w)\in \Lambda_{0}\cap \Sigma$ with $|A(y,w)|\geq 2$ and for any two consecutive values $s<r$ in $A(y,w)$, there are
$\Sigma_{i(s)}^{(y,w)}, \Sigma_{i(r)}^{(y,w)} \in \{\Sigma_1, \dots,\Sigma_k\}$ such that  
$$ \mathcal{R}(\phi^{s}(y,w))\in\Sigma_{i(s)}^{(y,w)} \ \text{and} \ \ \mathcal{R}^{-1}(\phi^{r}(y,w))\in\Sigma_{i(r)}^{(y,w)}.$$
\end{itemize}
   
\noindent Now, let us define a class of differentiable functions that contains all the information on the Lagrange and Markov Dynamical Spectra 
(see \cite{MR}) associated with the Busemann function.

For $(x,v)\in \Lambda_0\cap \Sigma$ consider the sets $A^{+}(x,v):=\{s>0\colon s\in A(x,v)\}$ and $A^{-}(x,v):=\{s<0\colon s\in A(x,v)\}$. If
$|A^{\pm}(x,v)|\geq 1$. Then, we denote $a^{+}(x,v)= \ds \min\{s: s\in A^{+}(x,v)\}$ and $a^{-}(x,v)=\ds\max\{s: s\in A^{-}(x,v)\}$. Thus, $a^{-}(x,v)<a^{+}(x,v)$ are consecutive values in $A(x,v)$, then by item (2), there are  $\Sigma_{i(a^{-}(x,v))}^{(x,v)}$ and $\Sigma_{i(a^{+}(x,v))}^{(x,v)}$ such that 
$$\phi^{s^{-}(x,v)}(x,v)=\mathcal{R}(\phi^{a^{-}(x,v)}(x,v))\in\Sigma_{i(a^{-}(x,v))}^{(x,v)}$$
and 
$$\phi^{s^{+}(x,v)}(x,v)=\mathcal{R}^{-1}(\phi^{a^{+}(x,v)}(y,w))\in\Sigma_{i(a^{+}(x,v))}^{(x,v)}.$$
\begin{itemize}
\item If $A^{-}(x,v)=\emptyset$, then put $r^{+}(x,v)$ such  that $\phi^{r^{+}(x,v)}(x,v)=\mathcal{R}(\phi^{a^{+}(x,v)}(x,v))$ and $\mathcal{R}^{-1}(x,v)=\phi^{t_{-}(x,v)}(x,v)$.
\item If $A^{+}(x,v)=\emptyset$, then put $r^{-}(x,v)$ such  that $\phi^{r^{-}(x,v)}(x,v)=\mathcal{R}^{-1}(\phi^{a^{-}(x,v)}(x,v))$ and $\mathcal{R}(x,v)=\phi^{t_{+}(x,v)}(x,v)$.
\end{itemize}

Thus, we can define the following function $\overline{\max}_{\phi}b_{\gamma}\colon \Sigma \to \re$ as

\[ \ds \overline{\max}_{\phi}b_{\gamma}(x,v) = \left\{ \begin{array}{lllll}
          \ds\max_{[s^{-}(x,v),s^{+}(x,v)]}b_{\gamma}(\pi(\phi^{t}(x,v))) & \mbox{\text{if} $|A^{\pm}(x,v)|\geq 1$},\\
         & \\
          \ds\max_{[t_{-}(x,v),r^{+}(x,v)]}b_{\gamma}(\pi(\phi^{t}(x,v))) &  \mbox{\text{if} $|A^{-}(x,v)|=0$}, \\
         & \\
       \ds\max_{[r^{-}(x,v),t_{+}(x,v)]}b_{\gamma}(\pi(\phi^{t}(x,v))) & \mbox{\text{if}  $|A^{+}(x,v)|=0$.}
       \end{array} \right.\]

\noindent Item (2) above shows that in each interval of the definition of $\overline{\max}_{\phi}b_{\gamma}(x,v)$ the maximum point is 
unique, then function $\overline{\max}_{\phi}b_{\gamma}$ is $C^{1}$.

 \subsubsection{The Behavior of the  Maximum Points of $\overline{\max}_{\phi}b_{\gamma}$} \label{BMb}
 
The next step is to show that the function $\overline{\max}_{\phi}b_{\gamma}$ has a finite number of maximum points (restricted to 
$\Lambda_0\cap \Sigma$) and in every one of these maximum points, its gradient is not parallel to the stable and unstable bundles.

From Lemma \ref{Collar}, the basic set $\Lambda_{0}$ has Hausdorff dimension close to $3$. Thus,  set $\Delta_{0}:=\Lambda_{0}\cap \Sigma$ is a basic set for  $\mathcal{R}$ with $HD(\Delta_{0})$ is close to 
$2$ (see Section \ref{HD>1}).

Moreover, the splitting $E^{s}\oplus \phi\oplus E^{u}$ over a neighborhood of $\Lambda_0$ defines a continuous splitting $E^{s}_{\Sigma}\oplus E^{u}_{\Sigma}$ of the tangent bundle $T\Sigma$, defined by 
\begin{eqnarray}\label{E1-BMb}
E^{s}_{\Sigma}(y)=E^{cs}_{y}\cap T_{y}\Sigma \ and \ E^{u}_{\Sigma}(y)=E^{cu}_{y}\cap T_{y}\Sigma,
\end{eqnarray}
where $E_{y}^{cs}=E^{s}_y\oplus \left\langle \phi(y)\right\rangle$ and $E_{y}^{cu}=E^{u}_y\oplus \left\langle {\phi}(y)\right\rangle$. Also,
$W^{s}(x,\Sigma)=W^{cs}_{y}\cap \Sigma$ and $W^{u}(x,\Sigma)=W^{cu}_{y}\cap \Sigma$ are the stable and unstable manifolds of $x\in \Delta_0$ (cf. \cite[Appendix]{R}).

\ \\
 We denote the set $$M_{\overline{\max}_{\phi}b_{\gamma}}(\Delta_0)=\{(y,w)\in \Delta_0: \overline{\max}_{\phi}b_{\gamma}(y,w)\geq \overline{\max}_{\phi}b_{\gamma}(x,v) \ \ \text{for all} \ \ (x,v)\in \Delta_{0}\}.$$
It is clear that by definition of $\overline{\max}_{\phi}b_{\gamma}$, the set  $M_{\overline{\max}_{\phi}b_{\gamma}}(\Delta_0)\subset U_{0}$.
The set  $U_{0}$ is defined as in Lemma \ref{L-Nasu}.

\begin{Le} \label{L1-BMb}For any $(y,w)\in M_{\overline{\max}_{\phi}b_{\gamma}}(\Delta_0)\cap \Sigma$,  
we have that 
\begin{equation}\label{E2-BMb}
\frac{\partial}{\partial v^{s}}\overline{\max}_{\phi}b_{\gamma}(y,w)\neq 0 \ \ \text{and} \ \ \frac{\partial}{\partial v^{u}}\overline{\max}_{\phi}b_{\gamma}(y,w)\neq 0,
\end{equation}
where $v^{j}\in E^{j}_{\Sigma}(y)$ is a unit vector, $j=s,u$. 
\end{Le}
\begin{proof}
Let $\alpha(r)\subset W^{s}_{\epsilon}((y,w),\Sigma)$ be with $\alpha'(0)=v^{s}$, then there is a differentiable real function $t(r)$ such that 
$$\overline{\max}_{\phi}b_{\gamma}(\alpha(r))=b_{\gamma}\circ\pi(\phi^{t(r)}(\alpha(r))).$$
Put $\tilde{\alpha}(r)=\phi^{t(r)}(\alpha(r))$ e $\beta(r)=\pi(\tilde{\alpha}(r))$. By contradiction, suppose that  
$$0=\frac{\partial}{\partial v^{s}}\overline{\max}_{\phi}b_{\gamma}(y,w)=\frac{d}{dr}\overline{\max}_{\phi}b_{\gamma}(\alpha(r))|_{r=0}=\langle\nabla\,b_{\gamma}(\beta(0)), \beta'(0)\rangle.$$
Also, as the function $t \mapsto b_{\gamma}\circ\pi(\phi^{t}(y,w)))$ has a local maximum in $t=t(0)$, then 
$$0=\frac{d}{dt}b_{\gamma}\circ\pi(\phi^{t}(y)))|_{t=t(0)}=\langle\nabla\,b_{\gamma}(\beta(0)), D\pi_{\phi^{t(0)}(y,w)}\phi(\phi^{t(0)}(y,w))\rangle=\langle\nabla\,b_{\gamma}(\beta(0)), \gamma'_{w}(t(0))\rangle.$$
Thus, as $\nabla\,b_{\gamma}(\beta(0))\neq 0$, then $\beta'(0)$ and $\gamma'_{w}(t(0))$ are parallel.

Now, our next task is to show that the last statement leads to a contradiction. Observe that
$\tilde{\alpha}'(r)=A\phi(\tilde{\alpha}(r))+B\Sigma(r)$ for some $\Sigma(r)\in E^{s}(\tilde{\alpha}(r))\setminus \{0\}$. Then, we take $\theta>0$ 
small and consider the surface 
$$\mathcal{S}_{\theta}^{s}:=\bigcup_{r}\phi^{(-\theta,\theta)}(\tilde{\alpha}(r))\subset W^{cs}_{loc}(\tilde{\alpha}(0)).$$
Then, the tangent space of $\mathcal{S}_{\theta}^{s}$ in $\phi^{\delta}(\tilde{\alpha}(r))$ for $|\delta|<\theta$ is 
\begin{eqnarray*}
T_{\phi^{\delta}(\tilde{\alpha}(r))}\mathcal{S}_{\theta}^{s}&=&\text{span}\{\phi(\phi^{\delta}(\tilde{\alpha}(r)))\}\oplus \text{span}\{ D\phi^{\delta}_{\tilde{\alpha}(r)}(\tilde{\alpha}'(r))\}\\
&=&\text{span}\{\phi(\phi^{\delta}(\tilde{\alpha}(r)))\} \oplus \text{span}\{ D\phi^{\delta}_{\tilde{\alpha}(r)}(\Sigma(r))\} = E^{cs}(\phi^{\delta}(\tilde{\alpha}(r))).
\end{eqnarray*}
Therefore, by Lemma \ref{Lemma 1 ad}, we have that $\mathcal{S}_{\theta}^{s}$ is transverse to the fiber of $\pi$, so
$\pi|_{\mathcal{S}_{\theta}^{s}}$ is a local diffeomorphism. We know that $D\pi_{\tilde{\alpha}(0)}\phi(\tilde{\alpha}(0))=\gamma'_{w}(t(0))$ and $D\pi_{\tilde{\alpha}(0)}(\tilde{\alpha}'(0))=\beta'(0)$ are parallel and 
$T_{\tilde{\alpha}(0)}\mathcal{S}_{\theta}^{s}=\langle\phi(\tilde{\alpha}(0))\rangle\oplus \langle \tilde{\alpha}'(0)\rangle$. This is a
contradiction since $\pi|_{\mathcal{S}_{\theta}^{s}}$ is a local diffeomorphism.

The proof of the unstable case is analogous. Thus, we conclude the proof of the lemma.
\end{proof}
\begin{C}\label{C2-BMb}
The set of maximum points $M_{\overline{\max}_{\phi}b_{\gamma}}(\Delta_0)$ consists of isolated points and therefore is finite. 
\end{C}
\subsection{Proof of Theorem \ref{T3}}\label{PT3}
In this section, we use the family of  perturbations, $g_{w}$, of  metric $g$ as in Subsection \ref{sec PM}, in such a way that if $\mathcal{R}^{w}$  is 
the Poincar\'e map  associated with the geodesic flow of $g_w$ (see paragraph after Corollary \ref{C-final}), then the pair  
$(\mathcal{R}^{w},\Delta_{0}^{w})$ has the property $V$ (see Subsection \ref{PropV}), where
$\Delta_{0}^{w}$ is the hyperbolic continuation of $\Delta_0$. Moreover, by Remark \ref{RBI}, we can assume also that $\mathcal{R}^{w}$ has
the property that the Birkhoff invariant is non-zero for some periodic orbit (see Subsection \ref{Family of Pert}).\\
Let $\gamma$ be a ray as in  Subsections \ref{D. of B. F.} and \ref{H. for tube}. Since the GCS of $\Sigma$ can be taken transverse to the 
fiber of $\pi$, then the surface $\pi^{-1}(\gamma)$ is transverse to $\Sigma$. So
$\ds\pi^{-1}(\gamma)\cap \Sigma$ is a finite number of $C^2$-curves in $\Sigma$. Then by Lemma \ref{LIC}, there is a
sub-horseshoe $\Delta_1$ of $\Delta_{0}$ such that $\Delta_1\cap (\pi^{-1}(\gamma)\cap \Sigma)=\emptyset$, and $HD(\Delta_1)$ is close to $HD(\Delta_0)$, and therefore, $HD(\Delta_1)\sim 2$ (see paragraphs at the beginning in Subsection \ref{BMb}). 
Thus, (using $\Delta_1$ instead of $\Delta_0$) we can assume, without loss of generality, that  the support of 
$g_{w}$, $\text{supp}\, g_{w}:=\overline{\{x\in M: g(x)\neq g_{w}(x)\}}$ does not intersect the geodesic $\gamma$, which  implies that
$\gamma$ is still a geodesic ray of $g_{w}$.

Denote by $b_{\gamma}^{w}$ the Busemann function associated to the ray $\gamma$ with the metric $g_{w}$ and $\phi_w$ the derivative of the geodesic
flow of the metric $g_w$. Is easy to see that all Lemmas of Subsection \ref{BMb} hold for the function $\overline{\max}_{\phi_w}b_{\gamma}^{w}$. In particular: 
\begin{C}\label{C-FINITE}
The set of maximum points $M_{\overline{\max}_{\phi_w}b_{\gamma}^{w}}(\Delta_0^{w})$ consists of isolated points and therefore is finite.
\end{C}

\begin{R}\label{R-FINAL}By the robustness of the property $V$ \emph{({cf. Section \ref{PropV} and \cite[Appendix 5.3]{MY1}})} we can find a small neighborhood, $\mathcal{G}\subset${\boldmath$\mathcal{G}$}$^{3}(M)$, of $g_w$ such that 
\begin{itemize}
\item[\emph{(a)}] The pair $(\cR^g, \Delta_{0,g}^w)$ has the property $V$, where $\Delta_{0,g}^{w}$ is the hyperbolic continuation of $\Delta_{0}^{w}$ by $\cR^g$ \emph{(see Subsection \ref{Dim. Red})}.
\item[\emph{(b)}] Let $\gamma_{_{g}}\in \mathcal{G}$ a ray of $M$ with the metric $g$, which represents the same \emph{end} of $\gamma$, then we can assume $($by the same argument before \emph{Corollary \ref{C-FINITE}}\emph{)} that $\Delta_{0,g}^w\cap \pi^{-1}(\gamma_{_{g}})=\emptyset$.
\item[\emph{(c)}] The Busemann function $b_{\gamma_{_{g}}}$ satisfies all the results of \emph{Subsection \ref{D. of B. F.} and Subsection  \ref{H. for tube}}.
\end{itemize}
\end{R}


\begin{proof}[\bf{Proof of Theorem \ref{T3}}]
\noindent Consider the set
$$H_{\mathcal{R}_{0}^{w}}^{\infty}=\{f\in C^{1}(\Sigma_0,\re):\#M_{f}(\Delta_0^{w})<\infty \ \ \text{and} \ \ Df_{z}(e^{j}_z)\neq 0,  \ \ z\in M_{f}(\Delta_0^{w}), \ \, j=s,u\},$$
where $e^{j}_{z}\in E^{j}_{\Sigma}(z)$ is a unit vector, $j=s,u$.\\
By Lemma \ref{L1-BMb} and Corollary \ref{C2-BMb}, we have that 
$\overline{\max}_{\phi_w}b_{\gamma}^{w}\in H_{\mathcal{R}_{0}^{w}}^{\infty}$. Moreover, the pair $(\mathcal{R}^{w},\Delta_0^{w})$ has the property $V$, then by the Main Theorem in \cite{MR} (see also Corollary \ref{C-MTMR}) we have that the sets $$\mathbb{L}(\overline{\max}_{\phi_w}b_{\gamma}^{w},\Delta_0^{w})=\bigg\{\ds\limsup_{n\to \infty}\overline{\max}_{\phi_w}b_{\gamma}^{w}((\mathcal{R}_{0}^{w})^{n}(x)):x\in \Delta_{0}^{w}\bigg\}$$ and 
$$\mathbb{M}(\overline{\max}_{\phi_w}b_{\gamma}^{w},\Delta_0^{w})=\bigg\{\ds\sup_{n\in \mathbb{Z}}\overline{\max}_{\phi_w}b_{\gamma}^{w}((\mathcal{R}_{0}^{w})^{n}(x)):x\in \Delta_{0}^{w}\bigg\}$$
have nonempty interior. Moreover, since $$\mathbb{L}(\overline{\max}_{\phi_w}b_{\gamma}^{w},\Delta_0^{w})\subset L(b_{\gamma}^{w}, \phi_{w}) \ \ \text{and} \ \ \mathbb{M}(\overline{\max}_{\phi_w}b_{\gamma}^{w},\Delta_0^{w})\subset \mathbb{M}(b_{\gamma}^{w}, \phi_{w}),$$
and  height function $h^{w}_{\gamma}=-b_{\gamma}^{w}$, then 
$$\text{int}\, \mathbb{M}(h_{\gamma}^{w}, \phi_{w})\neq \emptyset \ \ \text{and} \ \ \text{int}\,\mathbb{L}(h_{\gamma}^{w}, \phi_{w})\neq \emptyset.$$
Analogously, by Remark \ref{R-FINAL}, we have that for all $g\in \mathcal{G}$
$$\text{int}\, \mathbb{M}(h_{\gamma_{_{g}}}, \phi_{g})\neq \emptyset \ \ \text{and} \ \ \text{int}\,\mathbb{L}(h_{\gamma_{_{g}}}, \phi_{g})\neq \emptyset,$$
which completes the proof.
\end{proof}

\ \\
\newpage
\ \\
\appendix
\section{Appendix}\label{sec RCSEMAH}
\subsection{Stable and Unstable Manifold}\label{HF}

{
The Stable and Unstable Manifold Theorem \cite{K} implies that, if $\Lambda$ is a hyperbolic set for a flow $\phi^t$, then there is $\epsilon>0$ such that for every $x\in \Lambda$ the set 
$$W^{s}_{\epsilon}(x)=\{y: d(\phi^{t}(x),\phi^{t}(y))\leq \epsilon \ \text{and} \ d(\phi^{t}(x),\phi^{t}(y)) \to 0 \,\,\ \text{as}\,\, t \to +\infty\}$$
and 
$$W^{u}_{\epsilon}(x)=\{y: d(\phi^{t}(x),\phi^{t}(y))\leq \epsilon \ \text{and} \ d(\phi^{t}(x),\phi^{t}(y)) \to 0 \,\,\ \text{as}\,\, t \to -\infty\}$$
are invariant $C^{r}$-manifolds tangent to $E^{s}_{x}$ and $E^{u}_{x}$ respectively at $x$. Then we call
$W^{s}_{\epsilon}(x)$ the local \emph{strong-stable manifold}  and $W^{u}_{\epsilon}(x)$ the local \emph{strong-unstable manifold}, sometimes denoted by $W^{s}_{loc}(x)$ and $W^{u}_{loc}(x)$, respectively. Here $d$ is the distance on $M$ induced by the Riemannian metric. Moreover, the manifolds $W^{s}_{\epsilon}(x)$ and $W^{u}_{\epsilon}(x)$ vary continuously with $x$.
Also, if $x\in \Lambda$ one has that 
$$W^{s}(x)=\bigcup_{t\geq 0}\phi^{-t}(W^{s}_{\epsilon}(\phi^{t}(x))) \ \ \text{and} \ \ W^{u}(x)=\bigcup_{t\leq 0}\phi^{-t}W^{u}_{\epsilon}(\phi^{t}(x))$$
are $C^r$ invariant manifolds immerse in $M$, called of \emph{strong-stable manifold} and \emph{strong-unstable manifold} of $x$, respectively. Finally, the sets 
$$W^{cs}(x)=\bigcup_{t\in\re}W^{s}(\phi^{t}(x)) \ \ \text{and} \ \ W^{cu}(x)=\bigcup_{t\in\re}W^{u}(\phi^{t}(x))$$
are invariant $C^{r}$ manifolds tangent to $E^{s}_{x}\oplus \phi(x)$ and $E^{u}\oplus \phi(x)$, respectively.
}

\subsection{Regular Cantor Sets}\label{RCS}
To keep the notation in the text, we rewrote the definition of regular Cantor sets found in \cite[Section 1.1]{MY}, but an alternative definition can be found in \cite[Chapter 4]{PT}.\\
Let $\mathbb{A}$ be a finite alphabet, $\mathbb{B}$ a subset of $\mathbb{A}^{2}$, and $\Sigma_{\mathbb{B}}$ the subshift of finite type of
$\mathbb{A}^{\mathbb{Z}}$ with allowed transitions $\mathbb{B}$. We will always assume that $\Sigma_{\mathbb{B}}$ is topologically mixing
and that every letter in $\A$ occurs in $\Sigma_{\mathbb{B}}$.

\indent An {\it expansive map of type\/} $\Sigma_{\mathbb{B}}$ is a map $g$ with the following properties:
\begin{itemize}
\item[(i)] the domain of $g$ is a disjoint union
$\ds\bigcup_{\mathbb{B}}I(a,b)$. Where for each $(a,b)$,\,\, $I(a,b)$ is a compact subinterval of $I(a) := [0,1]\times\{a\}$;
\item[(ii)] for each $(a,b) \in \mathbb{B}$, the restriction of $g$ to $I(a,b)$ is a smooth diffeomorphism onto $I(b)$
satisfying $|Dg(t)| > 1$ for all $t$.
\end{itemize} 
\noindent The {\it regular Cantor set\/} associated with $g$ is the maximal invariant set
$$K = \bigcap_{n\ge0} g^{-n}\bigg(\bigcup_{\B} I(a,b)\bigg).$$
\noindent Let $\Sigma^+_{\mathbb{B}}$ be the unilateral subshift associated with $\Sigma_{\mathbb{B}}$. There exists a unique homeomorphism $h\colon \Sigma^{+}_{\B} \to K$ such that
$$
h(\underline{a}) \in I(a_0), \text{ for } \underline{a} = (a_0,a_1,\dots) \in \Sigma^+_{\mathbb{B}}
\ \ and \ \ 
h\circ\sigma =g \circ h,$$
where $\sigma^{+}\colon \Sigma_{\B}^{+} \to \Sigma_{\B}^{+}$, is defined as  $\sigma^{+}((a_{n})_{n\geq 0})=(a_{n+1})_{n\geq0}$.

\subsection {Markov Partitions and Stable and Unstable Cantor Sets}\label{sec EMAH}
Some good references in this section are the books \cite{PT} and 
\cite[Section 3.2.3]{OurBook}.
\subsubsection{Markov Partition}\label{MarPar}
\begin{Defi}
Let $\Lambda$ be a horseshoe associated to a $C^{2}$-diffeomorphism $\varphi$ on a surface $M$. A Markov partition $R$ for $\Lambda$ is a finite collection of compact sets $\{R_1, R_2, \dots, R_k\}$ $($all (diffeomorphic to a square$)$ such that:
\begin{enumerate}
\item[\emph{1.}] $ \Lambda \subset \bigcup_{i=1}^{k}R_i$,
\item[\emph{2.}] $\emph{int}R_i\cap \emph{int}R_j= \emptyset$, $i\neq j$, 
\item[\emph{3.}] If $x\in \emph{int}R_i$ and $\varphi(x)\in \emph{int}R_j$, then  $W^u(\varphi(x))\cap R_j\subset \varphi(W^u(x)\cap R_i)$ and \\
$\varphi(W^s(x)\cap R_i)\subset W^s(\varphi(x))\cap R_j$,
\item[\emph{4.}] There is a positive integer $n$ such that $\varphi^{n}(R_i)\cap R_j\neq \emptyset$, for all $1\leq i, j\leq k$ (mixing property).
\end{enumerate}

\end{Defi}
Markov partitions play an important role in dynamical systems as they allow us to encode dynamics. Moreover, in this case, the diameter of $R_i's$ can be chosen arbitrarily small (cf. \cite[Theorem 2 at Appendix]{PT}).  

The dynamical on a horseshoe $\Lambda$ can be codified using a \emph{matrix of transition} associated to a Markov partition $R$ of $\Lambda$. In fact, the matrix of transition $A=(a_{ij})$ is defined by 

\[ a_{ij}= \left\{ \begin{array}{lll}
          1 \, \, \, \,\,\text{if}\, \, \varphi(R_i)\cap R_j \neq \emptyset; & \\
          \, & \\
          0 \, \, \, \,  \text{otherwise}.
       \end{array} \right.\]
       
Consider the space of symbols $\Xi = \{1, \dots, k\}^{\mathbb{Z}}$. We say that a finite word $(x_1, x_2, \dots, x_n)$ in $\Xi$ is admissible if $a_{x_{i}x_{i+1}}=1$, for $i=1, 2, \dots, n-1$. The admissible words are important since it help to localize the dynamic on the Markov partition. More generally, we consider the space of bi-infinite admissible words $\Xi_{A}=\{(x_n)\in \Xi: a_{x_{n}x_{n+1}}=1\}$. It is not difficult to show that the $\varphi|_{\Lambda}$ and the shift $\sigma \colon \Xi_{A} \to \Xi_{A}$ are conjugated, \emph{i.e.}, there is a homeomorphism $h\colon \Lambda \to \Xi_{A}$ such that 
$$h\circ \varphi = \Sigma\circ h,$$
where $\sigma(x_n)=(x_{n+1})$.

\subsubsection{Stable and Unstable Cantor Sets}

\noindent If $\Lambda$ is a horseshoe associated to a $C^{2}$-diffeomorphism $\varphi$ on a surface $M$ and consider a finite collection 
$(R_{a})_{a\in\mathbb{A}}$ of disjoint rectangles of $M$, which are a Markov partition of $\Lambda$ (cf. \cite[Appendix 2]{PT} for more details). Define the sets 
 $$W^{s}(\Lambda,R)=\bigcap_{n\geq0}\varphi^{-n}(\bigcup_{a\in \mathbb{A}}R_{a}),$$
$$W^{u}(\Lambda,R)=\bigcap_{n\leq0}\varphi^{-n}(\bigcup_{a\in \mathbb{A}}R_{a}).$$
There is a $r>1$ and a collection of $C^{r}$-submersions $(\pi_{a}:R_{a}\rightarrow I(a))_{a\in\mathbb{A}}$, satisfying the following property:\\
\ \\
If $z,z^{\prime}\in R_{a_{0}}\cap \varphi^{-1}(R_{a_{1}})$ and $\pi_{a_{0}}(z)=\pi_{a_{0}}(z^{\prime})$, then we have $$\pi_{a_{1}}(\varphi(z))=\pi_{a_{1}}(\varphi(z^{\prime})).$$

\noindent In particular, the connected components of $W^{s}(\Lambda,R)\cap R_{a}$ are the level lines of $\pi_{a}$. Then we define a mapping $g^{u}$ of class $C^{r}$ (expansive of type $\Sigma_{\mathbb{B}}$) by the formula
$$g^{u}(\pi_{a_{0}}(z))=\pi_{a_{1}}(\varphi(z))$$
\noindent for $(a_{0},a_{1})\in \B$, $z\in R_{a_{0}}\cap\varphi^{-1}(R_{a_{1}})$.
The regular Cantor set $K^{u}$ defined by $g^{u}$, describes the geometry transverse of the stable foliation $W^{s}(\Lambda,R)$.
Analogously, we can describe the geometry  transverse of the unstable foliation $W^{u}(\Lambda,R)$, using a regular Cantor set $K^{s}$ defined by a mapping $g^{s}$ of class $C^{r}$ (expansive of type $\Sigma_{\B}$).\\

\noindent Also, the horseshoe $\Lambda$ is locally the product of two regular Cantor sets $K^{s}$ and $K^{u}$. So, the Hausdorff dimension of $\Lambda$, $HD(\Lambda)$ is equal to
$HD(K^{s}\times K^{u})$, but for regular Cantor sets, we have that $HD(K^{s}\times K^{u})=HD(K^{s})+HD(K^{u})$. Thus 
$HD(\Lambda)=HD(K^{s})+HD(K^{u})$ (cf. \cite[Ch. 4]{PT}).

\subsection{Geometry of $TM$ and $SM$}\label{sec Geo SM}

\noindent The following two subsections can be found in \cite{P}:
\subsubsection{Vertical and horizontal subbundles}\label{Sasaki}
Let $\pi\colon TM \to M$ the canonical projection, \emph{i.e.}, if $\theta=(x,v)\in TM$, then $\pi(\theta)=x$.\\
The vertical subbundle is defined by $V(\theta)=ker(d\pi_{\theta})$. It is easy to see that $\phi$ is transverse to the fibers $\pi^{-1}(\cdot)$.\\ 
Denote by  $TTM$ the tangent bundle of $TM$. The connection map $$K\colon TTM\to TM,$$
is defined as follows. Let $ \xi\in T_{\theta}TM$ and $\rho \colon (-\epsilon, \epsilon) \to TM$ be a curve adapted to $\xi$, that
is, with initial conditions as follows:
$$\left\{\begin{array}{rc}
\rho(0)=\theta;&\mbox{  \ } \\ 
\rho'(0)=\xi. &\mbox{ \  } \\ 
\end{array}\right.
$$

\noindent Such a curve gives rise to a curve $\nu\colon (-\epsilon , \epsilon) \to M$, $\nu:=\pi\circ \rho$, and a vector field $\Upsilon$ along to $\nu$, equivalently, $\rho(t) = (\nu(t), \Upsilon(t))$.
Define

$$K_{\theta}(\xi)=(\nabla_{\nu}\Upsilon)(0),$$ 
where $\nabla$ is the Levi-Civita connection, and $TTM$ is the tangent bundle of $TM$.\\
The horizontal subbundle is the subbundle of $TTM$ whose fiber at $\theta$ is given by
$H(\theta)=ker K_{\theta}$. The vertical and the horizontal subbundle hold that 
$$T_{\theta}TM=H(\theta)\oplus V(\theta),$$
and that the map $j_{\theta}\colon T_{\theta}TM  \to T_{x}M \times T_{x}M$ given by
$$j_{\theta}(\xi)=(d\pi_{\theta}(\xi),K_{\theta}(\xi)),$$
is a linear isomorphism.\\
By writing $\xi=(\xi_h,\xi_v)$,  where $\xi_h=d\pi_{\theta}(\xi)$ and $\xi_v=K_{\theta}(\xi)$, we identify $\xi$ with
$j_{\theta}(\xi)$.

Using the decomposition $T_{\theta}TM=H(\theta)\oplus V(\theta)$, we can define
naturally a Riemannian metric on $TM$ that makes $H(\theta)$ and $V(\theta)$
orthogonal. This metric is called the \emph{Sasaki metric} and is given by
$${\left\langle \left\langle \xi,\eta \right\rangle\right\rangle}_{\theta}={\left\langle d\pi_{\theta}(\xi),d\pi_{\theta}(\eta)\right\rangle}_{\pi(\theta)}+{\left\langle K_{\theta}(\xi),K_{\theta}(\eta)\right\rangle}_{\pi(\theta)}.$$

\noindent The one-form $\alpha$ of $TM$ is defined by 
$$\alpha_{\theta}(\xi)=\left\langle\left\langle \xi , \phi(\theta)\right\rangle\right\rangle={\left\langle d\pi_{\theta}(\xi) ,v \right\rangle}_{x},$$
such that $\alpha$ restricted to $SM$ (the unit tangent bundle) it becomes a contact form whose characteristic flow is the geodesic flow restricted
to $SM$.

\subsubsection{Jacobi Fields and the Differential of the Geodesic Flow}
In this section, for $\theta \in SM$, we shall describe an isomorphism between the tangent space $T_{\theta}TM$
and the Jacobi fields along the geodesic $\gamma_{\theta}$. Using the decomposition of $T_{\theta}TM$ in
vertical and horizontal subspaces, we shall give a very simple expression for the
differential of the geodesic flow in terms of Jacobi fields.
Recall that a Jacobi vector field along the geodesic $\gamma_{\theta}$ is a vector field along
$\gamma_{\theta}$ that is obtained as the variational vector field of a variation of $\gamma_{\theta}$ through geodesics. It is well known that $J$ is a Jacobi vector field along $\gamma_{\theta}$ if and only it
satisfies the \textit{Jacobi equation}
$$J''+R(\gamma_{\theta}',J)\gamma_{\theta}'=0,$$
where $R$ is the Riemann curvature tensor of $M$ and  $\ '$ denotes covariant derivatives along $\gamma_{y}$.

\noindent Let $\xi \in T_{\theta} TM$ and $\rho \colon  (-\epsilon, \epsilon) \to TM$ be an adapted curve to $\xi$. Then the
map $(s, t) \to \pi\circ\phi^{t}(\rho(s))$ gives rise to a variation of $\gamma_{\theta}=\pi\circ\phi^{t}(\theta)$. The curves
$t \to \pi\circ\phi^{t}(\rho(s))$ are geodesics and, therefore, the corresponding variational vector
field $J_{\xi}(t)=\frac{\partial}{\partial s}|_{s=0}\pi\circ\phi^{t}(\rho(s))$ is a Jacobi vector field with initial conditions
given by
$$\left\{\begin{array}{ll}
J_{\xi}(0)=&\frac{\partial}{\partial s}|_{s=0}\pi\circ\phi^{t}(\rho(s))|_{t=0}=d\pi_{\theta}(\xi);\mbox{  \ } \\ 
J'_{\xi}(0)=&\frac{\partial}{\partial t}|_{t=0}\frac{\partial}{\partial s}|_{s=0}\pi\circ\phi^{t}(\rho(s)) \mbox{ \  } \\ 
 {\ \ \ \ \ \ \ } =& \frac{\partial}{\partial s}|_{s=0}\frac{\partial}{\partial t}|_{t=0}\pi\circ\phi^{t}(\rho(s))=\frac{\partial}{\partial s}|_{s=0} \Upsilon(s)=K_{\theta}(\xi).\\
\end{array}\right.
$$
Using the above representation, we can describe the differential of the geodesic flow in terms of
Jacobi fields and the splitting of $T_{\theta}TM$ into horizontal and vertical subbundles we have 

\noindent \textbf{Claim:} Given $\theta\in TM$, $\xi\in T_{\theta}TM$, and $t\in \re$, we have
$$d\phi^{t}_{\theta}(\xi)=(J_{\xi}(t),J'_{\xi}(t)).$$

The following Lemma can be found in (cf. \cite[p. 42]{P}) for the compact case. For the non-compact case, the proof still holds with some adaptations.

\begin{Le}\label{Lemma 3 ad}
Let $X \subset SM$ be a hyperbolic set. Then for any $\theta\in X$ $$E^{s}(\theta)\oplus E^{u}(\theta) = \emph{ker} \alpha_{\theta}.$$ 
\end{Le}
\begin{proof}
Let us show that $E^{s}(\theta)\subset \text{ker} \alpha_{\theta}$. For $E^{u}(\theta)$ the proof is analogous. Since $\phi^{t}$ preserves the contact form $\alpha$, then given
$\eta\in E^{s}(\theta)$, we have
\begin{eqnarray*}
\alpha_{\theta}(\eta)&=&\alpha_{\phi^{t}(\theta)}(d\phi^{t}_{\theta}(\eta))\\
&=&\left\langle d \pi_{\phi^{t}(\theta)}(d\phi^{t}_{\theta}(\eta)),\gamma'_{\theta}(t)\right\rangle=\left\langle d(\pi\circ \phi^{t})_{\theta}(\eta), \gamma'_{\theta}(t)\right\rangle\\
&=&\left\langle J_{\eta}(t),\gamma'_{\theta}(t)\right\rangle.
\end{eqnarray*}
Since $\left\|d\phi^{t}_{\theta}(\eta)\right\|^{2}=\left\|J_{\eta}(t)\right\|^{2}+\left\|J'_{\eta}(t)\right\|^{2}$ and
$\left\|d\phi^{t}_{\theta}(\eta)\right\|\to 0$ when $t\to \infty$, then\break $\left\|J_{\eta}(t)\right\|\to 0$ when $t \to \infty$. Also,
$\left|\alpha_{\theta}(\eta)\right|\leq\left\|J_{\eta}(t)\right\|$, so we have that $\alpha_{\theta}(\eta)=0$ showing that
$E^{s}(\theta)\subset \text{ker} \alpha_{\theta}$. Since $E^{s}(\theta)\oplus E^{u}(\theta)$ and $\text{ker} \alpha_{\theta}$ have the same dimensions, then we conclude that
$E^{s}(\theta)\oplus E^{u}(\theta)=\text{ker} \alpha_{\theta}$.
\end{proof}
\subsubsection{Stable and Unstable Jacobi Fields}\label{stable and unstable JF}
Let $\theta=(x,v)$ and $w$ orthogonal to $v$ and let $J_{w}^{T}(t)$ be the unique Jacobi filed on $\gamma_{\theta}(t)$ such that 
$$J_{w}^{T}(0)=w \ \ \text{and} \ \ J_{w}^{T}(T)=0.$$ 
The limit $J_{w}^{s}(t):=\ds \lim_{T\to \infty}J_{w}^{T}(t)$ exists and is a Jacobi vector field on $\gamma_{\theta}(t)$ (cf. \cite{Eb}). Clearly,
$J_{w}^{s}(0)=w$ and $J_{w}^{s}(t)\neq 0$ for all $t>0$. We call $J_{w}^{s}(t)$  \textit{the stable Jacobi field}.\\
The \textit{unstable Jacobi field} $J_{w}^{u}(t)$ along $\gamma_{v}(t)$ is obtained  by considering taking the limit as $T\to -\infty$, \emph{i.e.},
$$J_{w}^{u}(t):=\ds \lim_{T\to-\infty}J_{w}^{T}(t).$$
\noindent The subspaces (using the identification $T_{\theta}SM=H(\theta)\oplus V(\theta)$)
\begin{eqnarray*}
E^{s}({\phi^{t}(\theta)})&=&\{(J(t),J'(t))\in T_{\phi^{t}(\theta)}SM| J \ \text{is a stable Jacobi field}\},\\
E^{u}({\phi^{t}(\theta)})&=&\{(J(t),J'(t))\in T_{\phi^{t}(\theta)}SM| J \ \text{is an unstable Jacobi field}\}
\end{eqnarray*}
are called the \emph{Green subbundles on $\gamma_{y}$}, which are also the stable and unstable subbundles of the definition of hyperbolicity of the geodesic flow on $SM$ (cf. \cite{Eb}).  

In the pinched negative condition of the curvature ($-a^2\leq K_M\leq-b^2<0$), we have that the Jacobi fields that vanish at some point are always divergent, in other words, can not be bounded (\cite{Eb}).

\begin{Le}\label{Lemma 1 ad}
If $w\in T_{\theta}SM\setminus\{0\}$ is a vertical vector, then $w\notin E^{s}(\theta)\oplus\phi(\theta)$ and $w\notin E^{u}(\theta)\oplus\phi(\theta)$, where  $E^{s}(\theta)$, $E^{u}(\theta)$ are the stable and unstable space, respectively. 
 \end{Le}
\begin{proof}
From Lemma \ref{Lemma 3 ad} we have  $E^{s}(\theta)\oplus E^{u}(\theta)=ker\alpha_{\theta}$. By contradiction, 
assume that $w= \alpha \xi^{s}+\beta\phi(\theta)$ with $\xi^{s}\in E^{s}(\theta)$, then 
$$0=\left\langle d\pi_{\theta}(w),v\right\rangle=\left\langle \alpha d\pi_{\theta}(\xi^{s})+\beta d\pi_{\theta}(\phi(\theta)),v\right\rangle=\alpha \left\langle d\pi_{\theta}(\xi^{s}),v\right\rangle+\beta \left\langle v,v\right\rangle=\beta.$$
Therefore, $\xi^{s}$ is a vertical vector, but 
$$E^{s}(\theta)=\{(J_{s}(0), J'_{s}(0))\in H(\theta)\oplus V(\theta):J_{s} \ \text{is a stable Jacobi field}\}.$$ 
If $\xi^{s}=(J_{s}(0), J'_{s}(0))$, then $J_s(0)=0$, and hence, the condition on the curvature implies that $J_{s}$ is a divergent Jacobi filed, which is a contradiction since $J_s$ is bounded because it is a  stable Jacobi field. \\
The proof for $E^{u}(\theta)\oplus\phi(\theta)$ is analogous.
\end{proof}

\subsection{Generic Properties of Geodesic Flows}\label{Generic Properties}
\subsubsection{Klingerberg and Takens Theorem}\label{KTT}
Let $J^{k}_{s}(2n)$ be the set of $k$-jets of symplectic automorphisms
of $(\re^{2n},0)$ (with $\omega=\\ \sum dx_i \wedge dx_j$; let Q be an invariant subset
of $J^{k}_{s}(2n)$ (\emph{i.e.} for all $\sigma \in J^{k}_{s}(2n)$, $\sigma Q \sigma^{-1}=Q$).
Let $Q\subset J^{k}_{s}(2n)$ and $\gamma$ a closed geodesic for the metric $g$, we say that $\phi_{g}$ (the vector of the geodesic flow)
has the property $Q\subset J^{k}_{s}(2n)$ if, for some $x\in \gamma$ and some cross-section $\Sigma$ to $\phi_g$ in $x$, we have that the symplectic 
automorphism $DP_{x}\in Q$, where $P$ is  the Poincar\'e map associate to the section $\Sigma$, (we denote by $P_Q$  for specifying the 
property $Q$). It is important to note that this definition is not  dependent on the cross-section $\Sigma$ (cf. \cite{KT}).
\begin{T}[Klingerbeng-Takens]Let $Q\subset J_s^{k}(2n)$ be open dense and invariant. Then the following 
property $P_Q$ is $C^{k+1}$-generic, \emph{i.e.}, $\phi_g$  has property $P_Q$ if the Poincar\'e map of every closed geodesic of $g$ belongs to $Q$.
\end{T}
\subsubsection{The Birkhoff Invariant}\label{BI}
Let $f\colon (\re^2,0)\to (\re^2,0)$ be a germ of diffeomophism area-preserving (in dimension two is symplectic) and $0$ a hyperbolic fixed point with eigenvalues $\lambda$ and $\lambda^{-1}$, then the Birkhoff normal form (cf. \cite{Moser}) says that there is an area-preserving change of coordinates $\Phi$ such that $\Phi^{-1}\circ f \circ \Phi=N$, where 
$N(x,y)=(U(xy)x,U^{-1}(xy)y)$ and $U(xy)$ is a power series $\lambda+U_2xy+\cdots$ convergent in a neighborhood of $x=y=0$. In other words, in this coordinates $f$ 
can be written as
\begin{equation}\label{EBI}
f(x,y)=(\lambda x(1+axy+\mathcal{O}(\Vert(x,y)\Vert^{4})),\lambda^{-1} y(1-axy+\mathcal{O}(\Vert(x,y)\Vert^{4}))),
\end{equation}
where the constant $a$ is called the \textit{Birkhoff Invariant} of $f$.
\begin{Le}\label{Birkhoff Invariant}
The Birkhoff invariant for diffeomorphism area-preserving in $(\re^2,0)$ only depends of $3$-jets in $0$, $J^{3}(0)$. Moreover, the set of
diffeomorphism area-preserving in $(\re^2,0)$ such that the Birkhoff invariant is non-zero is open, dense, and invariant in $J^{3}(0)$.
\end{Le}
\begin{proof}
The proof of \cite[ Theorem 1 and  2]{Moser}, show  the first part of the lemma and also the opening.  For density, suppose that for some  $f\colon (\re^2,0)\to (\re^2,0)$, the Birkhoff invariant is zero, then for $\epsilon>0$ we consider the function $N_\epsilon(x,y):=(\lambda x(1+\mathcal{O}(\Vert(x,y)\Vert^{4})),\lambda^{-1} y(1+\mathcal{O}(\Vert(x,y)\Vert^{4})))+\epsilon(x^2y,-xy^2)$, then the function $f_{\epsilon}=\Phi\circ N_\epsilon \circ \Phi^{-1}$ is area-preserving diffeomorphism close to $f$ with the Birkhoff invariant is $\epsilon$. \\
Let $f$, $g$ be as above and suppose that the Birkhoff invariant for $f$ is non-zero, then $g^{-1}\circ f\circ g$ has the Birkhoff invariant non zero. Indeed, by the Birkhoff Normal Form \cite[Theorem 1]{Moser},  there is an area-preserving change of coordinates $\Phi$ such that $\Phi^{-1}\circ g^{-1}\circ f\circ g \circ \Phi$ has the form (\ref{EBI}), then $(g\circ \Phi)^{-1}\circ f \circ (g\circ \Phi)$ has the form (\ref{EBI}), in other words, there is another area-preserving change of coordinates 
$g\circ \Phi$ such that $f$ has the form (\ref{EBI}), but by the uniqueness of the Birkhoff normal form (see \cite[page 674]{Moser}), we have that the Birkhoff invariant of $g^{-1}\circ f\circ g$ is equal to the Birkhoff invariant of $f$, therefore non-zero.
\end{proof}

$$\bf{Acknowledgments}$$
The authors are thankful to Instituto de Matem\'atica Pura e Aplicada (IMPA) for the excellent academic environment during the preparation of
this manuscript. Also, we thank  Carlos Matheus, Gonzalo Contreras, and  Jairo Bochi   for helpful discussions about this work.  This work was
financially supported by CNPq-Brazil, Capes, and the Palis Balzan Prize. 

\bibliographystyle{alpha}	
\bibliography{Third-Paper}

\noindent \textbf{Carlos Gustavo T. de A. Moreira}\\
School of Mathematical Sciences - Nankai University \\
Tianjin 300071 - P. R. China 
and \\
Instituto de Matem\'atica Pura e Aplicada (IMPA)\\ 
Estrada Dona Castorina 110, cep 22460-320\\
Rio de Janeiro-Brasil \\
E-mail: gugu@impa.br
\ \\ 

\noindent \textbf{Sergio Augusto Roma\~na Ibarra}\\
Universidade Federal do Rio de Janeiro\\
Av. Athos da Silveira Ramos 149, Instituto de Matem\'atica, Centro de Tecnologia \ - Bloco C \ - Cidade Universit\'aria \  Ilha do Fund\~ao, cep 21941-909 \\
Rio de Janeiro-Brasil\\
E-mail: sergiori@im.ufrj.br\\

\end{document}